\documentclass{article}
\usepackage[T1]{fontenc}
\usepackage{titlesec}
\usepackage[utf8]{inputenc}
\usepackage{amsmath}
\usepackage{amssymb}
\usepackage{amsfonts}
\usepackage{amsthm}
\usepackage{geometry}
\usepackage{babel}
\usepackage{xcolor}
\newtheorem{lemma}{Lemma}
\newtheorem{theorem}{Theorem}
\newtheorem{proposition}{Proposition}

\numberwithin{equation}{section}
\title{PhD}
\author{}
\date{January 2021}

\begin{document}

\begin{center}
    \textbf{Two CLTs for Sparse Random Matrices}
\end{center}

\begin{center}
     Simona Diaconu\footnote{Courant Institute, New York University, simona.diaconu@nyu.edu}
\end{center}

\begin{abstract}
Let \(G=G(n,p_n)\) be a homogeneous Erdös-Rényi graph, and \(A\) its adjacency matrix with eigenvalues \(\lambda_1(A) \geq \lambda_2(A) \geq ... \geq \lambda_n(A).\) Local laws have been used to show that \(\lambda_2(A)\) can exhibit fundamentally different behaviors: Tracy-Widom (\(p_n \gg n^{-2/3}\)), normal (\(n^{-7/9} \ll p_n \ll~n^{-2/3}\)), and a mix of both (\(p_n=cn^{-2/3}\)). 
Additionally, this technique renders the largest eigenvalue \(\lambda_1(A),\) separated from the rest of the spectrum for \(p_n \gg n^{-1},\) has Gaussian fluctuations when \(p_n \geq n^{-1}(\log{n})^{6+c}\) for some \(c>0.\) This paper shows this remains true in the range \(Bn^{-1}(\log{n})^4 \leq p_n \leq 1-Bn^{-1}(\log{n})^4\) with \(B>0\) universal, the tool behind it being a central limit theorem for the eigenvalue statistics of \(A\) that is justified via the method of moments.
\end{abstract}

\section{Introduction}\label{intro}
Random graphs are crucial in disciplines such as computer science (\cite{feigeofek}) and algorithmic graph theory (\cite{alon}), and consequently the adjacency matrices as well as the Laplacian transforms of several models have been closely studied. The homogeneous Erdös-Rényi graph \(G(n,p_n)\) is by far the most classical example in this field: an undirected simple graph 
on \(n\) vertices, each edge being present independently of the rest with probability \(p_n\) (see \cite{bollobastal} for inhomogeneous graphs, i.e., \(\mathbb{P}(ij \in E(G(n,p_n)))=p_{n,ij}\)). The focus of this paper is on the adjacency matrix \(A \in \mathbb{R}^{n \times n}\) of \(G(n,p_n),\) which is sparse when \(p_n \ll 1\) (in expectation, \(A\) has \(np_n=o(n)\) nonzero entries per row when \(p_n=o(1)\)). Denote its eigenvalues by \(\lambda_1(A) \geq \lambda_2(A) \geq ... \geq \lambda_n(A),\) and \(||A||=\max{(|\lambda_1(A)|,|\lambda_n(A)|)};\) most of the existing results in the literature are concerned with \(\lambda_1(A)\) or \(\lambda_2(A).\) Füredi and Komlós~\cite{furedikomlos} showed that when \(p_n=p \in (0,1),\) 
\begin{equation}\label{e1}
    \lambda_1(A)-[(n-1)p+1-p] \Rightarrow N(0,2p(1-p)).
\end{equation}
Subsequently, Krivelevich and Sudakov~\cite{krivelevichsudakov} proved
\begin{equation}\label{krivsud}
    \lambda_1(A)=(1+o_p(1)) \cdot\max{(\sqrt{\Delta},np_n)},
\end{equation}
where \(\Delta=\Delta(G(n,p_n))\) is the maximal vertex degree of \(G(n,p_n):\) this together with what was known about \(\Delta\) suggested a change in behavior for \(\lambda_1(A)\) when \(np_n \propto \log{n}\) (in the sparse regime \(np_n \ll \log{n},\) there is no concentration: some vertices have degrees much larger than the expectation \(np_n\) and other much lower, whereas in the dense regime \(np_n \gg \log{n}\), the graph is roughly regular: see~\cite{vu}, \cite{lupeng} for almost sure upper bounds on \(||A||,\max_{1 \leq i \leq n}{|\lambda_i(A)-\lambda_i(\mathbb{E}[A])|}\)). The second largest eigenvalue has also been analyzed, especially the centered version \(\lambda_2(A-\mathbb{E}[A]),\) whose asymptotic law is fairly well understood and for which the threshold \(\log{n}\) is anew crucial. Benaych-Georges et al.~\cite{benaychgeorgesetal1}, \cite{benaychgeorgesetal2} showed that for \(np_n \gg \log{n},\) 
\begin{equation}\label{big}
    \lambda_2(A-\mathbb{E}[A]) \xrightarrow[]{p} 2,
\end{equation}
while for \(np_n \ll \log{n},\) the largest eigenvalues are of order greater than \(\sqrt{\log{n}}:\) 
\begin{equation}\label{small}
    \lambda_k(A-\mathbb{E}[A]) \approx \sqrt{\frac{\log{(n/k)}}{\log{[(\log{n})/(np_n)]}}}, \hspace{0.5cm}
    k \geq n^{1-\epsilon}, \hspace{0.2cm} \epsilon>0.
\end{equation}
Alt et al.~\cite{altetal} complemented these two regimes by analyzing \(np_n \approx c\log{n},\) and showing the transition from (\ref{big}) to (\ref{small}) occurs at \(c=\frac{1}{\log{(4/e)}}:\) this threshold has also been discovered through a different approach by Tikhomirov and Youssef~\cite{tikhomirovyoussef}. 
\par
Fluctuations of the largest eigenvalues in more general settings than \(p_n=p\) (covered by Füredi and Komlós) have been obtained as well. Erdös et al.~\cite{erdosetal} extended (\ref{e1}) when \(p_n \geq n^{-1}(\log{n})^{6+c}\) for some \(c>0\) under the normalization suggested by it:
\begin{equation}\label{e2}
    \frac{1}{\sqrt{p_n(1-p_n)}} \cdot (\lambda_1(A)-\mathbb{E}[\lambda_1(A)]) \Rightarrow N(0,2).
\end{equation}
It must be mentioned that despite the similarity between (\ref{e1}) and (\ref{e2}), they are justified through fundamentally different means. On the one hand, two stages render (\ref{e1}): finding the centering by the method of moments (i.e., computing first-order approximations of \(\mathbb{E}[tr(A^{2m})]\) for large \(m \in \mathbb{N}\)), and obtaining the fluctuations via two observations: \(\lambda_1(A)\) is much larger than all the other eigenvalues of \(A,\) and \(v_n=[1 1 \hspace{0.05cm} ... \hspace{0.05cm} 1]^T \in \mathbb{R}^n\) is roughly a leading eigenvector of \(A,\) \(Av_n \approx \lambda_1(A)v_n.\) On the other hand, local laws (i.e., approximations of the Stieltjes transform at complex values with imaginary parts positive and decaying at rates depending on \(n\)) produce (\ref{e2}): after centering the entries of \(A\) and normalizing them, the authors of \cite{erdosetal} prove both the edge eigenvalues of the new matrix \(H\) and its bulk spectrum behave like their counterparts in a Gaussian orthogonal ensemble. This together with Cauchy interlacing inequalities and the eigenvectors of \(H\) being delocalized provide the behavior of both \(\lambda_1(A)\) and \(\lambda_2(A).\) Furthermore, other uses of local laws led to the discovery of a transition at \(p_n \propto n^{-2/3}:\) Huang et al.~\cite{huangetal} showed that for \(n^{-7/9} \ll p_n \ll n^{-2/3}\) its distribution is asymptotically normal,
\[n^{1/2} \cdot (\lambda_2(A)-[2(np_n)^{1/2}+(np_n)^{-1/2}-\frac{5(np_n)^{-3/2}}{4}]) \Rightarrow N(0,1),\]
while at \(p_n=cn^{-2/3},\) it is a combination of a normal and a real-valued Tracy-Widom law. This threshold is optimal insomuch as for \(p_n \gg n^{-2/3}\) the fluctuations are given by the latter:
\[\frac{n^{2/3}}{\sqrt{np_n}} \cdot (\lambda_2(A)-2\sqrt{np_n(1-p_n)}) \Rightarrow TW_1\]
(Erdös et al.~\cite{erdosetal2} covered \(p_n \gg n^{-1/3},\) and Lee and Schnelli~\cite{leeschnelli} extended this range to \(p_n \gg n^{-2/3}\)). 
\par
In this paper, the convergence (\ref{e2}) is justified when \(Bn^{-1}(\log{n})^4 \leq p_n \leq 1-Bn^{-1}(\log{n})^4\) for a universal constant \(B>0\) (Theorem~\ref{th2}), coming closer to the conjecture Erdös et al. made in \cite{erdosetal} that (\ref{e2}) holds as long as \(np_n \gg \log{n}\) than the currently known results (this seems the natural threshold in light of (\ref{krivsud}), and the results mentioned earlier on the concentration of the vertex degrees of \(G(n,p_n)\)). Theorem~\ref{th2} is based on a CLT for statistics of \(A,\) whose proof relies on the counting technique pioneered by Sinai and Soshnikov in \cite{sinaisosh}. This machinery was originally designed to understand traces of powers of real-valued Wigner matrices \(A=(a_{ij})_{1 \leq i,j \leq n}:\) suppose \((a_{ij})_{1 \leq i \leq j \leq n}\) are i.i.d., \(a_{11}\stackrel{d}{=}-a_{11}, \mathbb{E}[a^2_{11}]=\frac{1}{4n}, \mathbb{E}[|a_{11}|^{2k}] \leq (\frac{Ck}{n})^k\) for all \(k \in \mathbb{N},\) and consider the centered traces of the powers of \(A.\) 
In \cite{sinaisosh}, it was shown that when \(m=o(n^{1/2}),m \to \infty,\)
\begin{equation}\label{traceclts}
    tr(A^m)-\mathbb{E}[tr(A^m)] \Rightarrow N(0,\frac{1}{\pi}),
\end{equation}
and the aforementioned authors extended the range for this convergence to \(m=o(n^{2/3})\) in \cite{sinaisosh2}. Moreover, Soshnikov~\cite{sohnikov} employed (\ref{traceclts}) with \(m \propto n^{2/3}\) to derive universality of the asymptotic joint distribution of the edge eigenvalues of Wigner matrices (both real- and complex-valued), and the counting arsenal justifying (\ref{traceclts}) has also been exploited in other contexts, including upper bounds for trace expectations of large powers of \(A\) (e.g., matrices with heavy-tailed entries in Auffinger et al.~\cite{auffinger} and with \(n^{\mu}, \mu \in (0,1],\) nonzero entries per row in Benaych-Georges and Péché~\cite{benaychpeche}) and universality results for finite rank perturbations of \(A\) (Féral and Péché~\cite{feralpeche}).
\par
A considerable difference between the current framework and the previous contexts in which this technique has been used is the lack of symmetry. As remarked in \cite{erdosetal}, one difficulty the method of moments faces with adjacency matrices is their entries not being even centered: particularly, universality oftentimes has been proven in two stages. First, justifying the claims in the Gaussian case (for which formulae can be utilized), and second, showing certain quantities, which determine the asymptotic behavior of the object of interest, depend solely on few moments of the underlying distributions, an example of this being the four moment theorem by Tao and Vu (theorem \(15\) in \cite{taovu}). In the present case, relaxing symmetry to a vanishing first moment suffices to detect the primary contributors in the expectations of interest, leading to a trace CLT for sparse matrices that, to the author's knowledge, has no close equivalent in the existing literature.

\begin{theorem}\label{th1}
Suppose \(A=(a_{ij})_{1 \leq i,j \leq n}\) is symmetric with \((a_{ij})_{1 \leq i \leq j \leq n}\) independent, Bernoulli distributed with parameter \(p_n,np_n(1-p_n)\to \infty,\) and let \(\tilde{p}_n:=\min{(p_n,1-p_n)}.\) There exists \(c_0>0\) such that for \(m \in \mathbb{N}\) with \( \frac{\log{n}}{c_0\log{(n\tilde{p}_n)}} \leq m \leq c_0(n\tilde{p}_n)^{1/4},\) as \(n \to \infty,\)
\begin{equation}\label{conv1}
    \frac{tr(A^{2m})-\mathbb{E}[tr(A^{2m})]}{2m \cdot (np_n)^{2m-1} \cdot 
    \sqrt{p_n(1-p_n)}} \Rightarrow N(0,2).
\end{equation}
\end{theorem}

It can be easily shown using Lindeberg's CLT that
\begin{equation}\label{meq1}
    \frac{tr(A^{2})-\mathbb{E}[tr(A^{2})]}{n\sqrt{p_n(1-p_n)}} \Rightarrow N(0,2).
\end{equation}
The threshold \(\frac{1}{\log{(np_n)}}\) is key for \(\mathbb{E}[tr(A^{2m})]\) in the sense that the terms yielding its first-order approximation 
differ for \(m \gg \frac{1}{\log{(np_n)}}\) and \(m \ll \frac{1}{\log{(np_n)}}:\) this distinction becomes apparent in the proof below (see Proposition~\ref{proptra} below) and explains the  discrepancy between (\ref{conv1}) and (\ref{meq1}). Some comments on the inequalities on \(p_n,m\) are in order. The condition \(np_n(1-p_n) \to \infty\) is employed in the analysis of higher moments of the left-hand side term in (\ref{conv1}) (given the division by \(\sqrt{p_n(1-p_n)},\) a restraint of this kind is unavoidable: it must be noted this is considerably weaker than \(p_n \leq 1-c_2,\) assumed in~\cite{erdosetal}), while the upper and lower bounds on \(m\) are needed when dealing with second-order terms in \(\mathbb{E}[tr(A^{2m})]\) 
and the moments of the ratio of interest, respectively. Regarding odd powers of \(A,\) the current bounds do not identify the leading terms (see end of section~\ref{sect1} for a conjecture on this regard).
Besides, although \(A\) is not the adjacency matrix of a simple graph (it allows loops), the justification of Theorem~\ref{th1} entails its diagonal entries play no role in the first-order term contributors, the result thus holding for adjacency matrices of simple graphs as well.
\par
The convergence in Theorem~\ref{th1} together with the separation between \(\lambda_1(A)\) and the rest of the spectrum of \(A\) (\(\lambda_1(A) \approx np_n,\) whereas \(\lambda_2(A) \approx 2\sqrt{np_n}\)) recoup the asymptotic behavior of the former.

\begin{theorem}\label{th2}
Under the assumptions of Theorem~\ref{th1} on \(A,p_n,\) and \(Bn^{-1}(\log{n})^4 \leq p_n \leq 1-Bn^{-1}(\log{n})^4\) for \(B>0\) universal,
\begin{equation}
    \frac{1}{\sqrt{p_n(1-p_n)}}(\lambda_1(A)-\mathbb{E}[\lambda_1(A)]) \Rightarrow N(0,2).
\end{equation}
\end{theorem}

\par
The remainder of the paper consists of proofs: section~\ref{sect1} treats Theorem~\ref{th1}, and section~\ref{sect2} deduces Theorem~\ref{th2} by means of (\ref{conv1}).
\vspace{0.4cm}
\par
\textbf{Acknowledgments:} The bulk of this work has been completed in the author's graduate studies at Stanford University. She is particularly grateful to George Papanicolaou and Lenya Ryzhik for their feedback on earlier versions of this paper.

\section{Trace CLT}\label{sect1}

This section justifies (\ref{conv1}). For the sake of simplicity, take \(p=p_n,\)
and let \(v=[1 1 \hspace{0.05cm} ... \hspace{0.05cm} 1]^T \in \mathbb{R}^n, \Tilde{a}_{ij}=a_{ij}-p.\) Then
\[A=(p+\Tilde{a}_{ij})_{1 \leq i,j \leq n}:=pvv^T+\Tilde{A},\]
\[\mathbb{E}[\Tilde{a}^q_{11}]=p(1-p) \cdot [(1-p)^{q-1}-(-p)^{q-1}], \hspace{0.2cm} q \in \mathbb{N};\]
in particular,
\[\mathbb{E}[\Tilde{a}_{11}]=0, \hspace{0.5cm} \mathbb{E}[\Tilde{a}^2_{11}]=p(1-p),\]
and by convention, \(A\) is assumed to have this structure throughout the remainder of this work. As mentioned in the introduction, (\ref{conv1}) is proved by computing moments: for \(l \in \mathbb{N},\) it is shown that
\begin{equation}\label{convhighmom}
    \lim_{n \to \infty}{\mathbb{E}[( \frac{tr(A^{2m})-\mathbb{E}[tr(A^{2m})]}{2m \cdot (np)^{2m-1} \cdot \sqrt{p(1-p)}})^l]}= \begin{cases}
    0, & l=2l_0-1,l_0 \in \mathbb{N}\\
    2^{l/2}(l-1)!!, & l=2l_0,l_0 \in \mathbb{N}
    \end{cases}
\end{equation}
(this suffices to conclude the desired convergence in law insomuch as normal distributions satisfy Carleman's condition: see, for instance, lemmas \(B.1\) and \(B.2\) in Baik and Silverstein~\cite{baisilvbook}).
\par
For (\ref{convhighmom}), the technique developed by Sinai and Soshnikov in \cite{sinaisosh} is crucial since it constructs a change of summation in
\begin{equation}\label{genexp}
    \mathbb{E}[tr(M^{q})]=\sum_{(i_0,i_1, \hspace{0.05cm} ... \hspace{0.05cm}, i_{q-1})}{\mathbb{E}[m_{i_0i_1}m_{i_1i_2}...m_{i_{q-1}i_0}]}:=\sum_{\mathbf{i}=(i_0,i_1, \hspace{0.05cm} ... \hspace{0.05cm}, i_{q-1},i_0)}{\mathbb{E}[m_{\mathbf{i}}]},
\end{equation}
from cycles \(\mathbf{i}:=(i_0,i_1, \hspace{0.05cm} ... \hspace{0.05cm}, i_{q-1},i_0)\) to tuples of nonnegative integers \((n_1,n_2, \hspace{0.05cm} ... \hspace{0.05cm},n_q),\) and this can be employed to infer the main contributors in (\ref{genexp}). In particular, the distribution of \(a_{11}\) being symmetric is also key in the computations in \cite{sinaisosh} due to this property entailing \(\mathbf{i}\) with \(\mathbb{E}[m_{\mathbf{i}}] \ne 0\) has all undirected edges of even multiplicity: i.e., for \(u,v \in \{1,2, \hspace{0.05cm} ... \hspace{0.05cm},n\},\) 
\[|\{0 \leq t \leq q-1,i_ti_{t+1}=uv\}| \in \{2m, m \in \mathbb{Z}\},\]
where \(uv=vu\) denotes an undirected edge between \(v,u.\) This simplification is substantial, and although in the present situation, neither \(a_{11}\) nor \(\Tilde{a}_{11}\) has symmetric law, the advantage of the latter over the former is being centered. Up to a large extent, the primary upper hand of symmetric distributions is turning most summands on the right-hand side of (\ref{genexp}) negligible. A weaker assumption, a vanishing first moment, accomplishes a considerable reduction as well: despite some cycles \(\mathbf{i}\) with \(\mathbb{E}[m_{\mathbf{i}}]\ne 0\) containing edges of odd multiplicity, graphs alike have at most \(\lfloor q/2 \rfloor\) pairwise distinct undirected edges, a property shared with the first-order contributors in the symmetric case. This observation and the growth of the moments of \(\Tilde{a}_{11}\) will imply that when \(q\) is even, the dominant terms in (\ref{genexp}) for \(M=\Tilde{A}\) are the same as if \(\Tilde{a}_{11}\) were symmetric. 
\par
Before proceeding with the proof of (\ref{conv1}), recall some terminology from Sinai and Soshnikov~\cite{sinaisosh}, necessary in what is to come. Interpret \(\mathbf{i}=(i_0,i_1, \hspace{0.05cm} ...\hspace{0.05cm} ,i_{q-1},i_0)\) as a directed cycle with vertices among \(\{1,2, \hspace{0.05cm} ... \hspace{0.05cm}, n\},\) call \((i_{k-1},i_k)\) its \(k^{th}\) edge for \(1 \leq k \leq q,\) where \(i_{q}:=i_0,\) denote by \((u,v), uv\) directed and undirected edges, respectively, from \(u\) to \(v\) for \(u,v \in \{1,2, \hspace{0.05cm} ... \hspace{0.05cm}, n\}\) (the former edges are the building blocks of the cycles appearing in (\ref{genexp}), while the latter determine their expectations): in particular, \(uv=vu,\) and its \textit{multiplicity} is
\[m(uv):=|\{t: 0 \leq t \leq q-1,i_ti_{t+1}=uv\}|.\]
An edge \((i_k,i_{k+1})\) and its right endpoint \(i_{k+1}\) are \textit{marked} if an even number of copies of \(i_ki_{k+1}\) precedes them: i.e., \(|\{t \in \mathbb{Z}: 0 \leq t \leq k-1, i_ti_{t+1}=i_ki_{k+1}\}| \in 2\mathbb{Z}.\) Call a cycle \(\mathbf{i}\) \textit{even} if each undirected edge has even multiplicity in it, and \textit{odd} otherwise. If the entries of \(M\) are symmetric, then exclusively the first category contributes to (\ref{genexp}): when they are centered, any cycle \(\mathbf{i}\) containing no edge of multiplicity \(1\) can have \(\mathbb{E}[m_{\mathbf{i}}] \ne 0.\) In light of this difference, additional terminology is introduced next.
\par
Say \(e,v\) are \textit{marked jointly} if there exists a marked copy of \(e\) in \(\mathbf{i}\) with \(v\) as its right endpoint: that is, there is \(0 \leq t \leq q-1\) with \((i_t,i_{t+1})\) marked, \(i_ti_{t+1}=e,i_{t+1}=v,\) and if the \(t^{th}\) edge \((i_t,i_{t+1})\) is unmarked, then call the \(s^{th}\) edge its \textit{marked counterpart} for \(0 \leq s \leq t-1, i_{s}i_{s+1}=i_ti_{t+1}, \{q: s<q<t, i_qi_{q+1}=i_ti_{t+1}\}=\emptyset\) (i.e., the last copy of \(i_ti_{t+1}\) preceding the \(t^{th}\) edge, which must be marked). In such cases, the \(t^{th}\) edge is referred to as the unmarked counterpart of the \(s^{th}\) edge: note all unmarked edges have counterparts, but this might fail for marked edges (this occurs exactly when the edges in question are the last copies of edges of odd multiplicity). 
\par
Return to the change of summation mentioned earlier, \(\mathbf{i} \to (n_1,n_2, \hspace{0.05cm} ... \hspace{0.05cm},n_q),\) undertaken in several stages (steps \(1\)-\(5\) in subsection~\ref{1.1}). Any vertex \(j \in \{1,2, \hspace{0.05cm} ... \hspace{0.05cm}, n\}\) of \(\mathbf{i},\) \(i_0\) being the only potential exception, is marked at least once (the first edge of \(\mathbf{i}\) containing \(j\) is of the form \((*,j)\) since \(i_0 \ne j,\) and \(j\) is incident with no earlier edge), for \(0 \leq k \leq q,\) denote by \(N_{\mathbf{i}}(k)\) the set of \(j \in \{1,2, \hspace{0.05cm} ... \hspace{0.05cm}, n\}\) marked exactly \(k\) times in \(\mathbf{i},\) and let \(n_k:=|N_{\mathbf{i}}(k)|.\) Then
\begin{equation}\label{tuplecond}
    \sum_{0 \leq k \leq q}{n_k}=n, \hspace{0.2cm} \sum_{1 \leq k \leq q}{kn_k}= \frac{q+o}{2},
\end{equation}
where \(o\) is the number of pairwise distinct undirected edges in \(\mathbf{i}\) of odd multiplicity (the difference between the number of marked and unmarked edges is \(o,\) and their sum is \(q\)). It must be said the case \(o=0,2|q\) corresponds to even cycles and gives rise to subsets that are essential in understanding most such objects of interest: for \(m \in \mathbb{N},\) let \(\mathcal{C}(m)\) be the set of cycles (up to vertex isomorphism) of length \(2m\) for which \(n_1=m\) and the first vertex is unmarked. These cycles are known to yield the first-order term in \(\mathbb{E}[tr(M^{2m})]\) when the matrix \(M\) is Wigner (see, for instance, \cite{sinaisosh}): in particular, 
\[|\mathcal{C}(m)|=C_m=\frac{1}{m+1} \binom{2m}{m},\] 
and their structure admits a recursive description (see \cite{oldpaper}).
\par
In the remainder of this section, subsection(s)
\begin{itemize}
    \item \ref{1.1} and \ref{1.2} deal with the primary contributors in (\ref{genexp}) for \(M=\Tilde{A}\) when \(2|q,2|q+1,\)
    respectively;
    
    \item \ref{1.3} introduces \textit{lacunary} cycles, used to
    transition from \(\Tilde{A}\) to \(A;\)
    
    \item \ref{1.4} completes the analysis of \(\mathbb{E}[tr(A^{2m})];\)
    
    \item \ref{1.5} concludes the proof of Theorem~\ref{th1}.
    
\end{itemize}
\par
A couple of remarks prior to proceeding with proofs are in order. First, in the forthcoming propositions and lemmas, the growth of \(p\) relative to \(n\) is encompassed by \(w>0,\) which can depend on \(n,\) the primary condition being 
\[\tilde{p}:=\min{(p,1-p)} \geq n^{-1+w},\] 
and the results hold so long as \(q^2\) does not grow faster than \(\min{(\tilde{p}n^{1-w},n^w)}\) (see Propositions~\ref{prop1} and \ref{prop2}). Second, the asymptotic behavior of \(\mathbb{E}[tr(\Tilde{A}^{2q+1})]\) is more complicated to pinpoint than that of its even counterpart, \(\mathbb{E}[tr(\Tilde{A}^{2q})].\) Judging by what happens with the latter, the former should have its leading term given by cycles of length \(2q-2\) with three identical loops attached at some of its vertices: however, the estimate (\ref{oddmom}) is about \(\sqrt{\frac{n}{\tilde{p}}}\) times larger than the contribution of such configurations. 

\subsection{Even Powers of \(\Tilde{A}\)}\label{1.1} 

This subsection, concerned with the behavior of \(\mathbb{E}[tr(\Tilde{A}^{2m})],\) consists of proving the following statement. 

\begin{proposition}\label{prop1}
There exist \(C,c>0\) such that if \(w>0, \tilde{p} \geq n^{w-1}, m^2 \leq c\min{(n^w,\tilde{p}n^{1-w})},\) then
\begin{equation}\label{evenmom}
    C_m(p(1-p))^{m}n^{m+1} \cdot (1-\frac{2m^2}{n}) \leq \mathbb{E}[tr(\Tilde{A}^{2m})] \leq C_m(p(1-p))^{m}n^{m+1} \cdot [1+C \cdot (\frac{m^2}{\tilde{p}n^{1-w}})^{1/3}],
\end{equation}
where \(C_m=\frac{1}{m+1}\binom{2m}{m}\) is the \(m^{th}\) Catalan number.
\end{proposition}
\par
To justify (\ref{evenmom}), use the terminology from Sinai and Soshnikov~\cite{sinaisosh} and its addenda presented in the beginning of the section. 
By convention, \(\mathbf{i}\) is a fixed cycle satisfying certain properties,  \(e\) denotes an undirected edge, \(u,v\) are vertices of \(\mathbf{i}\) (i.e., \(\exists \hspace{0.05cm} 0 \leq t_1,t_2 \leq q, u=i_{t_1}, v=i_{t_2}).\) 
To change summation in
\[\mathbb{E}[tr(\Tilde{A}^{2m})]=\sum_{(i_0,i_1, \hspace{0.05cm} ... \hspace{0.05cm}, i_{2m-1})}{\mathbb{E}[\Tilde{a}_{i_0i_1}\Tilde{a}_{i_1i_2}...\Tilde{a}_{i_{2m-1}i_0}]}:=\sum_{\mathbf{i}=(i_0,i_1, \hspace{0.05cm} ... \hspace{0.05cm}, i_{2m-1},i_0)}{\mathbb{E}[\Tilde{a}_{\mathbf{i}}]},\]
let 
\[\mathcal{O}:=\mathcal{O}(\mathbf{i})=\{e:e=uv, 2|m(e)+1\}
\hspace{0.5cm} 2o:=|\mathcal{O}|,\]
and suppose no undirected edge in \(\mathbf{i}\) has multiplicity one (else \(\mathbb{E}[\Tilde{a}_{\mathbf{i}}]=0\)). Recall (\ref{tuplecond}) which yields
\begin{equation}\label{nomarkedvtx}
    \sum_{1 \leq k \leq m+o}{kn_k}=m+o.
\end{equation}
Fix now \((o,n_1,n_2, \hspace{0.05cm} ... \hspace{0.05cm},n_{m+o}),\) and analyze next how many cycles are mapped to this tuple as well as the size of their corresponding expectations. This is split in \(5\) stages.
\par
\underline{Step 1:} Map the edges of \(\mathbf{i}\) to a tuple \((s_1,s_2, \hspace{0.05cm} ... \hspace{0.05cm},s_{2m})\) with \(m+o, m-o\) elements \(1\) and \(-1,\) respectively, according to \(s_k=1\) if \((i_{k-1},i_k)\) is marked, and \(s_k=-1\) if \((i_{k-1},i_k)\) is unmarked. The ensuing result yields there are 
\begin{equation}\label{sigmamo}
    \sigma(2o,m-o)=\binom{2m-1}{m-o}-\binom{2m-1}{m-o-2}
\end{equation}
such configurations.

\begin{lemma}\label{analoguedyckpaths}
For \(m,n \geq 0,m+n>0,\) let \(\sigma(m,n)\) be the number of tuples of length \(m+2n\) with \(m+n\) entries \(1,\) \(n\) entries \(-1,\) and all partial sums nonnegative: i.e.,
\begin{equation}\label{defsigma}
    \sigma(m,n)=|\{(s_1,s_2, \hspace{0.05cm} ... \hspace{0.05cm},s_{m+2n}), \sum_{j \leq m+2n}{s_j}=m, s_{i} \in \{-1,1\}, \sum_{j \leq i}{s_j} \geq 0, 1 \leq i \leq m+2n\}|.
\end{equation}
Then 
\[\sigma(m,n)=\binom{m+2n-1}{n}-\binom{m+2n-1}{n-2},\]
where \(\binom{m}{q}=0\) for \(m \in \mathbb{N},q<0.\) 
\end{lemma}

\begin{proof}
Use induction on \(n:\) note that \(\sigma(0,n)=C_n, \sigma(m,0)=1, \sigma(m,1)=m+1.\) The claim thus ensues for \(n \leq 1\) or \(m=0.\) Fix now \(n \geq 2:\) the last term in any tuple of interest is \(-1\) or \(1,\) whereby
\begin{equation}\label{rec1}
    \sigma(m,n)=\sigma(m+1,n-1)+\sigma(m-1,n).
\end{equation}
Suppose the result holds for \(\sigma(m-1,n).\) By virtue of the induction hypothesis and the claim for \(\sigma(0,n)\) being also true, showing the right-hand side of the desired formula satisfies the recurrence in (\ref{rec1}) suffices. This follows from
\[(\binom{m+2n-1}{n}-\binom{m+2n-1}{n-2})-(\binom{m+2n-2}{n}-\binom{m+2n-2}{n-2})=\binom{m+2n-2}{n-1}-\binom{m+2n-2}{n-3}\]
by employing \(\binom{s}{t}=\binom{s-1}{t}+\binom{s-1}{t-1}\) for \(s \in \mathbb{N}, t \in \mathbb{Z}.\) This completes the induction step and the proof.
\end{proof}

\par
\underline{Step 2:} Enumerate the marked vertices: in light of (\ref{nomarkedvtx}), the number of possibilities is
\[\frac{(m+o)!}{\prod_{k \geq 1}{(k!)^{n_k}}} \cdot \frac{1}{\prod_{k \geq 1}{n_k!}}.\]
\par
\underline{Step 3:} Choose the vertices of \(\mathbf{i},\) 
\[V(\mathbf{i}):={\{i_j, 0 \leq j \leq 2m\}}:\]
note
\begin{equation}\label{unamrked1}
    |V(\mathbf{i})|=\chi_{i_0 \in N_{\mathbf{i}}(0)}+\sum_{k \geq 1}{n_k}
\end{equation}
because any vertex in \(\mathbf{i},\) the first being the sole possible exception, is marked at least once.
\par
\underline{Step 4:} Establish the order of the remaining vertices: similarly to what occurs with even cycles, only the right endpoints of the unmarked edges are yet to be chosen. Lemma \(1\) in \cite{oldpaper} still holds as it does not use the parity of \(\mathbf{i},\) giving that for \(v \in N_{\mathbf{i}}(k)\) the number of possibilities at each such appearance of \(v\) is at most
\[\begin{cases}
    1, \hspace{0.6cm} k \leq 1,\\
    2k, \hspace{0.4cm} k \geq 2,
\end{cases}\]
(its proof entails there are at most \(2k\) marked edges that could be the marked counterpart of \((v,*)\)), and the overall bound for this stage remains 
\begin{equation}\label{bdstep444}
    \prod_{k \geq 2}{(2k)^{kn_k}},
\end{equation}
because there are at most \(k\) unmarked edges of the type \((v,*):\) denote by \(U,M\) the positions of the unmarked, marked edges, respectively (the \(t^{th}\) edge is \((i_{t-1},i_t)\)), and let 
\[\alpha(v)=|\{t:i_t=v,(t,t+1) \in M\times U\}|,\hspace{0.8cm} \beta(v)=|\{t: i_t=v,(t,t+1) \in U\times M\}|,\]
\[\gamma(v)=|\{t: i_t=v,(t,t+1) \in M\times M\}|,\hspace{0.8cm}  \delta(v)=|\{t: i_t=v,(t,t+1) \in U\times U\}|;\]
the number of interest is \(\alpha(v)+\delta(v) \leq k=\alpha(v)+\gamma(v)\) since \(v\) is incident with at least as many marked edges as unmarked, \(\alpha(v)+\beta(v)+2\gamma(v) \geq \alpha(v)+\beta(v)+2\delta(v).\)
\par
\underline{Step 5:} For \(q \in \mathbb{N}, q \geq 2,\) 
\begin{equation}\label{boundtilde}
    |\mathbb{E}[\Tilde{a}^q_{11}]|=p(1-p) \cdot |(1-p)^{q-1}-(-p)^{q-1}| \leq \tilde{p}(1-\tilde{p})^{q} \cdot (1+(\frac{\tilde{p}}{1-\tilde{p}})^{q-1}) \leq \tilde{p}(1-\tilde{p})^{q} \cdot(1+\frac{\tilde{p}}{1-\tilde{p}}).
\end{equation}
Since \((1-2p)^q \cdot \mathbb{E}[\Tilde{a}^q_{11}] \geq 0\) for all \(q \in \mathbb{N},\) all the summands in 
\[\mathbb{E}[tr(\Tilde{A}^{s})]=\sum_{\mathbf{i}=(i_0,i_1,\hspace{0.05cm}... \hspace{0.05cm},i_{s-1},i_0)}{\mathbb{E}[\Tilde{a}_{\mathbf{i}}]}\]
have the sign \((1-2p)^s.\) In particular, all the expectations can be replaced with their absolute values in the current case at no cost for the value of interest. Let 
\(e(\mathbf{i}):=|\{i_ki_{k+1},0 \leq k \leq 2m-1\}|,\) and notice 
\begin{equation}\label{unmarked2}
    e(\mathbf{i}) \geq \sum_{k \geq 1}{n_k}-\chi_{i_0 \not \in N_{\mathbf{i}}(0)}
\end{equation}
(take \(v \in \cup_{k \geq 1}{N_{\mathbf{i}}(k)}, v \ne i_0\) to the edge \(e=i_ti_{t+1},\) where \(t\) is minimal with \(i_t=v\) or \(i_{t+1}=v;\) then \(i_{t+1}=v\) because \(v \ne i_0,\) and \(i_j \ne v\) for \(0 \leq j \leq t;\) this yields the mapping is injective since \(v\) is the right endpoint of the first copy of \(e\) in \(\mathbf{i}\)). This last inequality in conjunction with (\ref{boundtilde}) and \(e(\mathbf{i}) \leq m\) entails
\begin{equation}\label{boundexpp}
    |\mathbb{E}[\Tilde{a}_{\mathbf{i}}]| \leq \tilde{p}^{e(\mathbf{i})}(1-\tilde{p})^{2m} \cdot (\frac{1}{1-\tilde{p}})^{e(\mathbf{i})} \leq \tilde{p}^{-1+\chi_{i_0 \in N_{\mathbf{i}}(0)}+\sum_{k \geq 1}{n_k}}(1-\tilde{p})^{m}.
\end{equation}
\par
It becomes apparent, when putting together steps \(1-5,\) that a connection between \((n_k)_{k \geq 1}\) and \(|\mathcal{O}|\) is necessary. This is established next.

\begin{lemma}\label{l1} 
If \(\mathbf{i}\) contains no edge of multiplicity \(1,\) then
\begin{equation}\label{parta}
    (a) \hspace{0.2cm} \sum_{k \geq 2}{kn_k} \geq \frac{3 \cdot |\mathcal{O}|}{2},
\end{equation}
\begin{equation}\label{unmarked3}
    (b) \hspace{0.2cm} \sum_{k \geq 2}{(k-1)n_k} \geq |\mathcal{O}|-\chi_{i_0 \not \in N_{\mathbf{i}}(0)}.\
\end{equation}
\end{lemma}

\begin{proof}
\((a)\) Let \(T\) be the number of edges \(e=uv \in \mathcal{O}\) with \(u \in N_{\mathbf{i}}(1)\) and \(e,u\) marked jointly. It suffices to show:
\begin{equation}\label{ineq2}
    \sum_{k \geq 2}{kn_k} \geq 2 \cdot |\mathcal{O}|-T,
\end{equation}
\begin{equation}\label{ineq22}
    T \leq \frac{|\mathcal{O}|}{2}.
\end{equation}
\par
For (\ref{ineq2}), proving no edge has both endpoints in \(N_{\mathbf{i}}(1)\) and marked jointly with it is enough because the number of all marked copies of elements of \(\mathcal{O}\) would be on the one hand, at least \(2 \cdot |\mathcal{O}|,\) while on the other hand, at most \(T+\sum_{k \geq 2}{kn_k}.\) Suppose for the sake of a contradiction that there exists such an edge \(e=uv \in \mathcal{O}, u,v \in N_{\mathbf{i}}(1);\) \(u \ne v\) insomuch as \(m(uv) \geq 3,\) and let \(t\) be minimal with \(i_t \in \{u,v\}.\) If \(t>0,\) then \((i_{t-1},i_t)\) is marked (it is the first edge in \(\mathbf{i}\) incident with \(i_t\)) and \(i_{t-1}i_t \ne e,\) implying \(i_t \in \{u,v\} \cap (\cup_{k \geq 2}{N_{\mathbf{i}}(k)}),\) contradiction. Hence \(e=i_0i_1 \in \mathcal{O},\) and \(i_0,i_1 \in N_\mathbf{i}(1)\) are both marked jointly with \(i_0i_1.\) Because \(m(i_0i_1) \geq 3\) and the first marked copy of \(e\) has \(i_1\) as its right endpoint, there are \(0<t_1<t_2\) with \(i_{t_1}i_{t_1+1}=i_0i_1\) unmarked, \((i_{t_2},i_{t_2+1})=(i_1,i_0)\) marked. Since the second time \(i_0\) appears in the cycle is marked unless its predecessor is \(i_1,\) this must occur at \(t_1+1:\) \((i_{t_1},i_{t_1+1})=(i_1,i_0).\) If \(t_1>1,\) consider the first appearance of \(i_1\) in \([2,t_1];\) this is the right endpoint of a marked edge (\(i_0\) does not show up in this interval), contradicting \(i_1 \in N_{\mathbf{i}}(1);\) hence \(t_1=1\) and the first three vertices of \(\mathbf{i}\) are \((i_0,i_1,i_0);\) take now the second time \(i_1\) appears in \(\mathbf{i}\) (there is a third copy of \(i_1i_0,\) and \(i_0 \ne i_1\)); this is the right endpoint of a marked edge, again absurd. 
\par
Continue with (\ref{ineq22}). \(T\) is the number of vertices \(u \in  N_{\mathbf{i}}(1)\) for which there exists \(e=uv \in \mathcal{O}\) such that the cycle contains a marked copy of \((v,u).\) Since \(\mathbf{i}\) is a cycle, there are at least two elements of \(\mathcal{O}\) with which \(u\) is incident (either \(u\) is incident with an even number of edges, or \(uu\) appears in \(\mathbf{i}:\) the latter case cannot occur because \(m(uu) \geq 3\) entails \(u \in \cup_{k \geq 2}{N_{\mathbf{i}}(k)}\)), and the previous paragraph yields any edge in \(\mathcal{O}\) is incident with at most such a vertex \(u,\) whereby \(2T \leq |\mathcal{O}|.\) 
\par
\((b)\) List the marked edges in \(\mathbf{i}\) by reading its edges in increasing order, and for each \(v \in \cup_{k \geq 1}{N_{\mathbf{i}}(k)},\) delete the first edge \(e(v)\) with which it is marked jointly. The number of remaining edges is 
\[\sum_{k \geq 1}{(k-1)n_k}=\sum_{k \geq 2}{(k-1)n_k},\]
and showing each undirected edge has at most one deleted copy, with one potential exception when \(i_0\) is marked, suffices since this entails the number of remaining edges is at least
\[2|\mathcal{O}|+|\mathcal{E}|-(|\mathcal{O}|+|\mathcal{E}|+\chi_{i_0 \in \cup_{k \geq 1}{N_{\mathbf{i}}}(k)})=|\mathcal{O}|-\chi_{i_0 \not \in N_{\mathbf{i}}(0)},\]
where \(\mathcal{E}\) is the set of undirected edges with even multiplicities (i.e., \(\mathcal{E}=\{e: e=uv,2|m(e),m(e)>0\}\)). 
\par
Suppose \(u,v\) satisfy \(u \ne v\) and \(e(v)=e(u):\) then \(e(u)=e(v)=uv.\) Assume without loss of generality that \(u\) appears for the first time in the cycle before \(v\) does, and  let \(t_1<t_2\) be these positions. If \(t_1>0,\) then \((i_{t_1-1},i_{t_1})=(*,u)\) is marked because it is the first edge incident with \(u,\) contradicting \(e(u)=uv.\) Hence \(t_1=0,u=i_0,\) concluding the justification of the claim above.                     
\end{proof}

Now the proof of Proposition~\ref{prop1} can be completed.

\begin{proof}
The final bound becomes, in light of steps \(1-5,\) and \(\sum_{k \geq 1}{n_k}=m+o-\sum_{k \geq 2}{(k-1)n_k},\) 
\[[\binom{2m-1}{m-o}-\binom{2m-1}{m-o-2}] \cdot n^{m} (\tilde{p}(1-\tilde{p}))^{m}\cdot \]
\[\sum_{(n_k)_{k \geq 1}}{n^{\chi_{i_0 \in N_{\mathbf{i}}(0)}} \cdot \tilde{p}^{-1+\chi_{i_0 \in N_{\mathbf{i}}(0)}} \cdot \tilde{p}^{o-\sum_{k \geq 2}{(k-1)n_k}} \cdot n^{o-\sum_{k \geq 2}{(k-1)n_k}} \cdot \frac{(m+o)!}{\prod_{k \geq 1}{(k!)^{n_k}}} \cdot \frac{1}{\prod_{k \geq 1}{n_k!}} \cdot \prod_{k \geq 2}{(2k)^{kn_k}}}.\]
Inequality (\ref{unmarked3}) gives
\[\sum_{k \geq 2}{(k-1)n_k} \geq w \cdot (2o-1+\chi_{i_0 \in N_{\mathbf{i}}(0)})+(1-w)\sum_{k \geq 2}{(k-1)n_k},\]
whereby the bound is at most
\[[\binom{2m-1}{m-o}-\binom{2m-1}{m-o-2}] \cdot n^{1+m} (p(1-p))^{m}\cdot\]
\begin{equation}\label{intbound}
    \sum_{(n_k)_{k \geq 1}}{(\tilde{p}n^{1-2w})^o \cdot \tilde{p}^{-\sum_{k \geq 2}{(k-1)n_k}} \cdot n^{-(1-w)\sum_{k \geq 2}{(k-1)n_k}} \cdot \frac{(m+o)!}{\prod_{k \geq 1}{(k!)^{n_k}}} \cdot \frac{1}{\prod_{k \geq 1}{n_k!}} \cdot \prod_{k \geq 2}{(2k)^{kn_k}}}
\end{equation}
because \(\tilde{p}(1-\tilde{p})=p(1-p),\) and
\begin{equation}\label{savings}
    n^{\chi_{i_0 \in N_{\mathbf{i}}(0)}} \cdot \tilde{p}^{-1+\chi_{i_0 \in N_{\mathbf{i}}(0)}} \cdot n^{w(1-\chi_{i_0 \in N_{\mathbf{i}}(0)})}=n (\tilde{p}^{-1}n^{-(1-w)})^{1-\chi_{i_0 \in N_{\mathbf{i}}(0)}} \leq n.
\end{equation}
Since \((m+o)!=(\sum_{k \geq 1}{kn_k})! \leq n_1! \cdot (m+o)^{\sum_{k \geq 2}{kn_k}}, k! \geq (k/e)^k,\)
\[\sum_{(n_k)_{k \geq 1}}{\tilde{p}^{-\sum_{k \geq 2}{(k-1)n_k}} \cdot n^{(w-1)\sum_{k \geq 2}{(k-1)n_k}} \cdot \frac{(m+o)!}{\prod_{k \geq 1}{(k!)^{n_k}}} \cdot \frac{1}{\prod_{k \geq 1}{n_k!}} \cdot \prod_{k \geq 2}{(2k)^{kn_k}}} \leq
\]
\[\leq \sum_{(n_k)_{k \geq 2}}{\tilde{p}^{-\sum_{k \geq 2}{(k-1)n_k}} \cdot n^{(w-1)\sum_{k \geq 2}{(k-1)n_k}} \cdot \prod_{k \geq 2}{\frac{(m+o)^{kn_k}(2k)^{kn_k}}{n_k!(k!)^{n_k}}}} \leq \sum_{(n_k)_{k \geq 2}}{\prod_{k \geq 2}{\frac{1}{n_k!} \cdot [\frac{(2e)^k(m+o)^{k})}{\tilde{p}^{k-1}n^{(1-w)(k-1)}}}]^{n_k}}.\]
\par
For any \(\epsilon \in [0,1],\)
\begin{equation}\label{prod}
    \prod_{k \geq 2}{\frac{1}{n_k!} \cdot [\frac{(2e)^k(m+o)^{k})}{\tilde{p}^{k-1}n^{(1-w)(k-1)}}}]^{\epsilon n_k} \leq \exp(\sum_{k \geq 2}{\frac{(2e)^{k\epsilon}(m+o)^{k\epsilon}}{\tilde{p}^{(k-1)\epsilon}n^{(1-w)(k-1)\epsilon}}}),
\end{equation}
and part \((a)\) in Lemma~\ref{l1} implies for \(\tilde{p} \geq n^{w-1}, 16m \leq (\tilde{p}n^{1-w})^{1/2},\)
\begin{equation}\label{prod22}
    \prod_{k \geq 2}{ [\frac{(2e)^k(m+o)^{k})}{\tilde{p}^{k-1}n^{(1-w)(k-1)}}}]^{(1-\epsilon) n_k} \leq \prod_{k \geq 2}{ [\frac{2e \cdot (m+o)}{\tilde{p}^{1/2}n^{(1-w)/2}}}]^{(1-\epsilon)k n_k} \leq (\frac{4em}{\tilde{p}^{1/2}n^{(1-w)/2}})^{(1-\epsilon) \cdot 3o}.
\end{equation}
As
\[(\frac{4em}{\tilde{p}^{1/2}n^{(1-w)/2}})^{3-3\epsilon} \cdot \tilde{p}n^{1-2w}=(4em)^{3-3\epsilon}\tilde{p}^{(3\epsilon-1)/2}n^{(3\epsilon-1)/2-w(3\epsilon+1)/2},\]
\(\epsilon=\frac{1}{3}\) implies the bound (\ref{intbound}) for \(o \geq 0, m^2 \leq  c\min{(n^{w},\tilde{p}n^{1-w})}\) is at most 
\[[\binom{2m-1}{m-o}-\binom{2m-1}{m-o-2}] \cdot n^{1+m} (p(1-p))^{m} \cdot ((4em)^{2}n^{-w})^o \cdot \exp(\frac{Cm^{2/3}}{(\tilde{p}n^{1-w})^{1/3}}) \leq \]
\begin{equation}\label{obound}
    \leq 2 \cdot [\binom{2m-1}{m-o}-\binom{2m-1}{m-o-2}] \cdot n^{1+m} (p(1-p))^{m} \cdot ((4em)^{2}n^{-w})^o.
\end{equation}
\par
For \(\delta \in (0,\frac{1}{3}],\) \(f(o)=[\binom{2m-1}{m-o}-\binom{2m-1}{m-o-2}] \cdot \delta^o\) decreases in \(o \in \mathbb{Z}_{\geq 0}\) because
\(f(o)=\delta^o \cdot \frac{(2o+1)(2m)!}{(m+o+1)!(m-o)!},\)
whereby
\[\frac{f(o+1)}{f(o)}=\delta \cdot \frac{2o+3}{2o+1} \cdot \frac{m-o}{m+o+2} <\delta \cdot 3 \cdot 1 \leq 1.\]
Notice (\ref{obound}) is tight up to a constant factor when \(o=0\) inasmuch as the contribution of the elements of \(\mathcal{C}(m)\) (i.e., \(n_1=m, i_0\) unmarked: see section \(2.3\) in \cite{oldpaper} for a recursive description of \(\mathcal{C}(m)\)) is 
\[C_m (p(1-p))^{m}\prod_{0 \leq i \leq m}{(n-i)}=[\binom{2m-1}{m}-\binom{2m-1}{m-2}] \cdot n^{1+m} (p(1-p))^{m} (1-O(\frac{m^2}{n})),\]
whereby the lower bound in (\ref{evenmom}) ensues because all the moments of \(\Tilde{a}_{11}\) can be taken to be nonnegative in light of the observation below (\ref{boundtilde}), \(1-x \geq e^{-2x}\) for \(x \in [0,\frac{1}{2}]\) from \((e^{2x}(1-x))'=e^{2x}(1-2x) \geq 0,\) and when \(m \leq \frac{n}{2},\)
\begin{equation}\label{exppeq}
    1 \geq \prod_{0 \leq i \leq m}{\frac{n-i}{n}} \geq \exp(-\sum_{0 \leq i \leq m}{\frac{2i}{n}}) \geq \exp(-\frac{2m^2}{n}) \geq 1-\frac{2m^2}{n}.
\end{equation}
For the upper bound in (\ref{evenmom}), it suffices to show the contributions when \(o>0\) and \(o=0,\mathbf{i} \not \in \mathcal{C}(m)\) can be absorbed by the last term in (\ref{evenmom}). When \(o>0,\) this follows from (\ref{obound}) and \(m^2n^{-w} \leq m^{2/3}(\tilde{p}n^{1-w})^{1/3}\) (a consequence of \(m^{4/3}=m \cdot m^{1/3} \leq n^{w/2} \cdot (\tilde{p}n^{1-w})^{1/3} \leq n^w\cdot (\tilde{p}n^{1-w})^{1/3}\)). When \(o=0,\mathbf{i} \not \in \mathcal{C}(m),\) this entails \(n_1 \ne m\) or \(i_0 \not \in N_{\mathbf{i}}(0),\) whereby either \(\sum_{k \geq 2}{kn_k}>0,\) in which case the product in (\ref{prod}) is at most
\begin{equation}\label{expbound}
    \exp(\sum_{k \geq 2}{\frac{(2e)^{k\epsilon}(m+o)^{k\epsilon}}{\tilde{p}^{(k-1)\epsilon}n^{(1-w)(k-1)\epsilon}}})-1 \leq \frac{2 \cdot (2e)^{2/3}m^{2/3}}{(\tilde{p}n^{1-w})^{1/3}},
\end{equation}
or a factor of
\[(\tilde{p}n^{1-w})^{-1} \leq \frac{m^{2/3}}{(\tilde{p}n^{1-w})^{1/3}}\] 
can be added in virtue of (\ref{savings}). The proof of the proposition is now complete.
\end{proof}

\subsection{Odd Powers of \(\Tilde{A}\)}\label{1.2}

This subsection considers \(\mathbb{E}[tr(\Tilde{A}^{2m+1})]\) and consists of proving the following statement. 

\begin{proposition}\label{prop2}
There exist \(C,c>0\) such that if \(w>0, \tilde{p} \geq n^{w-1}, m \leq c(np)^{1/4},\) then
\begin{equation}\label{oddmom}
    \mathbb{E}[tr(\Tilde{A}^{2m+1})] \leq C_{m+1}(p(1-p))^{m}n^{m+1} \cdot \frac{Cm^2}{(\tilde{p}n)^{1/2}}.
\end{equation}    
\end{proposition}

\begin{proof}
Proceed as in the proof of Proposition~\ref{prop1}. Fix a cycle \(\mathbf{i}\) of length \(2m+1,\) let 
\[\mathcal{O}:=\mathcal{O}(\mathbf{i})=\{e:e=uv, 2|m(e)+1\},
\hspace{0.5cm} 2o+1:=|\mathcal{O}|,\]
and suppose anew no undirected edge in \(\mathbf{i}\) has multiplicity \(1.\) Before proceeding with 
steps \(1-5,\) note (\ref{tuplecond}) gives that there are \(m+o+1\) marked edges, i.e.,
\begin{equation}\label{knkeq}
    \sum_{k \geq 1}{kn_k}=m+o+1.
\end{equation}
\par
\underline{Step 1:} Map the edges of \(\mathbf{i}\) to a tuple \((s_1,s_2, \hspace{0.05cm} ... \hspace{0.05cm},s_{2m+1})\) with \(m+o+1, m-o\) elements \(1\) and \(-1,\) respectively, according to \(s_k=1\) if \((i_{k-1},i_k)\) is marked, and \(s_k=-1\) if \((i_{k-1},i_k)\) is unmarked. Lemma~\ref{analoguedyckpaths} entails the number of such paths is
\[\sigma(2o+1,m-o)=\binom{2m}{m-o}-\binom{2m}{m-o-2}.\]
\par
\underline{Step 2:}  Enumerate the marked vertices: in light of (\ref{knkeq}), the number of possibilities is
\[\frac{(m+o+1)!}{\prod_{k \geq 1}{(k!)^{n_k}}} \cdot \frac{1}{\prod_{k \geq 1}{n_k!}}.\]
\par
\underline{Steps 3,4:} Idem.
\par
\underline{Step 5:} 
(\ref{boundexpp}) continues to hold, 
and Lemma~\ref{l1} gives
\begin{equation}\label{oddsum}
    \sum_{k \geq 2}{kn_k} \geq 3o+2, \hspace{0.8cm} \sum_{k \geq 2}{(k-1)n_k} \geq 2o+1-\chi_{i_0 \not \in N_{\mathbf{i}}(0)}.
\end{equation}
\par
The final bound becomes for fixed \(o,\)
\[[\binom{2m}{m-o}-\binom{2m}{m-o-2}] \cdot n^{m+1} (p(1-p))^{m}\cdot \]
\begin{equation}\label{bd1}
    \sum_{(n_k)_{k \geq 1}}{\tilde{p}^{o+\chi_{i_0 \in N_{\mathbf{i}}(0)}-\sum_{k \geq 2}{(k-1)n_k}} \cdot n^{o+\chi_{i_0 \in N_{\mathbf{i}}(0)}-\sum_{k \geq 2}{(k-1)n_k}} \cdot \frac{(m+o+1)!}{\prod_{k \geq 1}{(k!)^{n_k}}} \cdot \frac{1}{\prod_{k \geq 1}{n_k!}} \cdot \prod_{k \geq 2}{(2k)^{kn_k}}}.
\end{equation}

\begin{lemma}\label{lemma38}
    If \(\mathbf{i}\) has no edge of multiplicity \(1,\) then
    \[-o-\chi_{i_0 \in N_{\mathbf{i}}(0)}+\sum_{k \geq 1}{(k-1)n_k} \geq \frac{1}{2}\sum_{k \geq 1}{(k-1)n_k}.\]
\end{lemma}

\begin{proof}
This is a consequence of Lemma~\ref{lemma388}: in this case, \(s=1,\) and note that among the configurations for which the inequality does not hold there are no cycles (\(u \ne v\)).
\end{proof}

Return to (\ref{bd1}). Lemma~\ref{lemma38} and \(\tilde{p}n \geq n^w \geq 1\) provide the upper bound
\[[\binom{2m}{m-o}-\binom{2m}{m-o-2}] \cdot n^{m+1} (p(1-p))^{m}\cdot \sum_{(n_k)_{k \geq 1}}{ (\tilde{p}n)^{-\sum_{k \geq 2}{\frac{(k-1)n_k}{2}}}\cdot \frac{(m+o+1)!}{\prod_{k \geq 1}{(k!)^{n_k}}} \cdot \frac{1}{\prod_{k \geq 1}{n_k!}} \cdot \prod_{k \geq 2}{(2k)^{kn_k}}}.\]
Reason as in the previous subsection: for any \(\epsilon \in [0,1],\) the analogs of (\ref{prod}) and (\ref{prod22}) yield the above is at most 
\[[\binom{2m}{m-o}-\binom{2m}{m-o-2}] \cdot n^{m+1} (p(1-p))^{m}\cdot \sum_{(n_k)_{k \geq 1}}{\prod_{k \geq 2}{\frac{(2e(m+o)(\tilde{p}n)^{-\frac{k-1}{2k}})^{kn_k}}{n_k!}}} \leq\]
\[\leq [\binom{2m}{m-o}-\binom{2m}{m-o-2}] \cdot n^{m+1} (p(1-p))^{m}\cdot [\exp(\frac{Cm^{2\epsilon}}{(\tilde{p}n)^{\epsilon/2}})-1] \cdot (\frac{4em}{(\tilde{p}n)^{1/4}})^{(1-\epsilon)(3o+2)} \leq\]
\[\leq C(\epsilon) \cdot [\binom{2m}{m-o}-\binom{2m}{m-o-2}] \cdot n^{m+1} (p(1-p))^{m}\cdot (\frac{m}{(\tilde{p}n)^{1/4}})^2 \cdot (\frac{4em}{(\tilde{p}n)^{1/4}})^{3(1-\epsilon)o}\]
employing \(e^x-1 \leq 2x\) for \(x \in [0,\frac{1}{2}],\) and \(\sum_{k \geq 2}{n_k}>0\) from \(\sum_{k \geq 2}{kn_k} \geq 3o+2>0\) (see (\ref{oddsum})). Let anew \(\epsilon=1/3,\) and thus for \(m \leq c(n\tilde{p})^{1/4},\)
\begin{equation}\label{cval}
    \delta=(\frac{4em}{(n\tilde{p})^{1/4}})^{3(1-\epsilon)} \in (0,\frac{1}{4}),
\end{equation}
whereby using that for 
\(g(o)=[\binom{2m}{m-o}-\binom{2m}{m-o-2}] \cdot \delta^o=\delta^o \cdot \frac{2(2m+1)!(o+1)}{(m-o)!(m-o+2)!},\) 
\[\frac{g(o+1)}{g(o)}=\delta \cdot \frac{o+2}{o+3} \cdot \frac{m-o+2}{m-o+1} \leq \delta \cdot 1 \cdot 2<\frac{1}{2},\]
the final bound 
becomes 
\[C(\epsilon) \cdot C_{m+1} n^{m+1} (p(1-p))^{m}\cdot 2 \cdot \frac{m^2}{(\tilde{p}n)^{1/2}}\]
using \(g(0)=\sigma(1,m) \leq \sigma(0,m+1)=C_{m+1}\) (any tuple in the set underlying \(\sigma(1,m)\) can be extended to a tuple with \(m+1\) entries \(1\) and \(m+1\) entries \(-1\) by adding \(1\) at its end: see (\ref{defsigma}) for the definition of \(\sigma(\cdot,\cdot)\)). This completes the proof of the claimed inequality.
\end{proof}


\subsection{Lacunary Cycles}\label{1.3}

For \(m \in \mathbb{N},\) call 
\[\mathbf{i}=(i_0,i_1,i_2,\hspace{0.05cm}...\hspace{0.05cm},i_{m-1},i_0,e_1,e_2,\hspace{0.05cm}...\hspace{0.05cm},e_{m}), \hspace{0.2cm} i_0,i_1,\hspace{0.05cm}...\hspace{0.05cm},i_{m-1} \in \{1,2,\hspace{0.05cm}...\hspace{0.05cm},n\}, \hspace{0.2cm} e_1,e_2,\hspace{0.05cm}...\hspace{0.05cm},e_{m} \in \{0,1\}\]
a \textit{lacunary cycle}, let
\[a_{\mathbf{i}}:=\prod_{0 \leq t \leq m-1}{(\chi_{e_{t+1}=1} \cdot a_{i_ti_{t+1}}+\chi_{e_{t+1}=0} \cdot p)},\]
and for \(0 \leq k \leq m-1,\) call 
\((i_k,i_{k+1})\) \textit{erased} (\textit{unerased}) if \(e_{k+1}=0\) \((e_{k+1}=1)\) 
with \(i_{m}:=i_0.\)
\par
The traces of interest are sums over lacunary cycles:
\begin{equation}\label{trlac}
    \mathbb{E}[tr(A^{m})]=\mathbb{E}[tr((\Tilde{A}+(p)_{1 \leq i,j \leq n})^{m})]=\sum_{\mathbf{i}=(i_0,i_1,i_2,\hspace{0.05cm}...\hspace{0.05cm},i_{m-1},i_0,e_1,e_2,\hspace{0.05cm}...\hspace{0.05cm},e_{m})}{\mathbb{E}[a_{\mathbf{i}}]}.
\end{equation}
Note \(\mathbb{E}[a_{\mathbf{i}}]\) depends solely on the unerased edges of \(\mathbf{i}.\) Say \(\mathbf{i}\) has \(s\) \textit{components} if there are contiguous segments with endpoints \(((l_t,r_t))_{1 \leq t \leq s},\) called the \textit{split points of} \(\mathbf{i},\) giving the erased edges, i.e., \(s\) is minimal such that 
\begin{equation}\label{zeroededges}
    \{q: 1 \leq q \leq m, e_{q}=0\}=\cup_{1 \leq t \leq s}{\{l_t+1,l_t+2,\hspace{0.05cm}...\hspace{0.05cm},r_t\}},
\end{equation}
treat the remaining \(s\) paths with endpoints \(((r_i,l_{i+1}))_{1 \leq i \leq s}\) as a graph (\(l_{s+1}:=l_1\)), denoted by \(cut(\mathbf{i}),\) with edge and vertex labels set (equivalent to fixing \(((l_i,r_i))_{1 \leq i \leq s}\)) and inherited from the bigger cycle in which it is embedded. Furthermore, refer to \((i_j)_{j \in \cup_{1 \leq t \leq s}\{l_{t}+1,\hspace{0.05cm}...\hspace{0.05cm},r_t-1\}}\) as its \textit{interior vertices}, and say erased edges in \(\mathbf{i}\) do not to belong to or do not appear in \(\mathbf{i}_c,\) whereas naturally the rest belong to or appear in it. For instance, when \(n=6, (l_1,r_1,l_2,r_2)=(1,3,4,5),\) \(cut(\mathbf{i})\) is a union of two paths, \((i_3,i_4),(i_5,i_0,i_1),\) and most of the terminology used for cycles remains valid: 
the sole difference is that vertex and edge positions are inherited from \(\mathbf{i}\) (e.g., \((i_3,i_4)\) is the fourth edge, not the first: this is done to bypass relabelings and cumbersome notation).
\par
Since
\begin{equation}\label{rel}
    \mathbb{E}[a_{\mathbf{i}}]=p^{\sum_{1 \leq t \leq s}{(r_t-l_t)}}\mathbb{E}[\prod_{q: 1 \leq q \leq m, e_q=1}{a_{i_{q-1}i_{q}}}]:=p^{\sum_{1 \leq t \leq s}{(r_t-l_t)}} \cdot \mathbb{E}[a_{cut(\mathbf{i})}],
\end{equation}
the contribution of lacunary cycles with fixed split points is analyzed next.

\begin{proposition}\label{propqs}
There is \(c>0\) such that if \(w>0, \tilde{p} \geq n^{w-1},\) then lacunary cycles with \(s\) components, split points and interior vertices fixed, and \(2q, 2q+1\) unerased edges contribute
\begin{equation}\label{laccyc1}
   n^{q+s}(1-O(\frac{q^2}{n}))\tilde{p}^{m-q}(1-\tilde{p})^q C_q \cdot [C_q^{-1}\prod_{1 \leq j \leq s}{C_{m_j}}+O(\frac{q^2}{n^w})+
   O(\frac{q^{2/3}}{(\tilde{p}n^{1-w})^{1/3}})],
\end{equation}
\begin{equation}\label{laccyc2}
   O(n^{q+s}\tilde{p}^{m-q-1}(1-\tilde{p})^qC_{q}),
\end{equation}
in (\ref{trlac}), respectively, where \((2m_j)_{1 \leq j \leq s}\) are the lengths of the \(s\) components of \(cut(\mathbf{i}_c)\)  (\(C_l=0\) for \(l \not \in \mathbb{Z}\)), \(q^2 \leq c\min{(\tilde{p}n^{1-w},n^w)}\) in (\ref{laccyc1}) and \(q \leq c(n\tilde{p})^{1/4}\) in (\ref{laccyc2}).    
\end{proposition}

\par
\textit{Proof of (\ref{laccyc1}):}
In light of (\ref{rel}) and \(\sum_{1 \leq t \leq s}{(r_i-l_i)}=m-2q,\) (\ref{laccyc1}) is tantamount to
\begin{equation}\label{sub231}
   \sum{\mathbb{E}[a_{cut(\mathbf{i})}]}=n^{q+s}(1-O(\frac{q^2}{n}))(\tilde{p}(1-\tilde{p}))^q C_q \cdot [C_q^{-1}\prod_{1 \leq j \leq s}{C_{m_j}}+O(\frac{q^2}{n^w})+
   O(\frac{q^{2/3}}{(\tilde{p}n^{1-w})^{1/3}})], 
\end{equation}
where the summation is over lacunary cycles \(\mathbf{i}\) as described in Proposition~\ref{propqs}. Let \(\mathbf{i}\) be such an object, and for ease of notation, \(\mathbf{i}_c:=cut(\mathbf{i})\) (i.e., the original cycle with the erased edges left out). 
Repeat the analysis from subsection~\ref{1.1} for \(\mathbf{i}_c:\) begin marking edges and their right endpoints at \(r_1,\) let \(2o:=|\mathcal{O}|,\) and recall that
\[\sum_{k \geq 1}{kn_k}=q+o.\]
\par
\underline{Step 1:} For any \(h, 1 \leq h \leq 2q,\) there are at least as many marked edges as unmarked among the first \(h\) edges: hence the rationale and consequently the number of possibilities encompassed by (\ref{sigmamo}), employed in subsections~\ref{1.1} and \ref{1.2}, continues to hold here as well (despite \(\mathbf{i}_c\) being likely a disconnected graph).
\par
\underline{Step 2:} Idem.
\par
\underline{Step 3:} Any vertex not belonging to \(\{i_{r_q}, 1 \leq q \leq s\}\) is marked at least once: for \(v \not \in \{i_{l_q},i_{r_q}, 1 \leq q \leq s\},\) argue in the same vein as for cycles, while for \(v \in \{i_{l_q}, 1 \leq q \leq s\}-\{i_{r_q}, 1 \leq q \leq s\},\) the first appearance of \(v\) in \(\mathbf{i}_c\) is the right endpoint of an edge which must be marked. Hence \(n(0) \leq s\) vertices are unmarked, and
\[|V(\mathbf{i}_c)|=n(0)+\sum_{1 \leq i \leq k}{n_k}.\]
\par
\underline{Step 4:} Anew solely the right endpoints of the unmarked edges are yet to be chosen. In the classical case (cycles with no erased edges), each vertex \(v \in N_{\mathbf{i}}(k)\) generates a factor of 
\[\begin{cases}
    1, \hspace{1.2cm} k \leq 1\\
    (2k)^k, \hspace{0.5cm} k \geq 2:
\end{cases}\]
in the current case, \(v \in N_{\mathbf{i}}(k)\) yields a factor of 
\[(2k+r(v))^{k+r(v)/2}, \hspace{0.5cm} k \geq 2,\] 
where
\begin{equation}\label{defrv}
    r(v):=|\{q: 1 \leq q \leq s,i_{r_q}=v\}, \hspace{0.5cm} v \in \{1,2,\hspace{0.05cm}...\hspace{0.05cm},n\}.
\end{equation}
To see this, note lemma \(1\) from \cite{oldpaper} can be modified by weakening the key inequality in its proof, 
\[a_{j+1} \leq (a_j+l_j+2)-(u_j+l_j)=a_j+2-u_j\]
to
\[a_{j+1} \leq (a_j+l_j+2)-(u_j+l_j-q_j)=a_j+2-u_j+q_j\]
for \(q_j=|\{t: e_{t+1}=0,n_j \leq t<n_{j+1}\}|\)
(i.e., \(q_j\) counts the erased edges between the \(j^{th}\) and \((j+1)^{th}\) appearances of \(v\)). This entails \(2k\) can be replaced by \(2k+r(v),\) and the power by  
\begin{equation}\label{bdknew}
    k+r(v)/2:
\end{equation} 
reason as below (\ref{bdstep444}); apart from \(\alpha(v),\beta(v),\gamma(v),\delta(v),\) let \(D=\{k:e_k=0\},\) and
\[\xi_1(v)=|\{t: i_t=v,(t,t+1) \in D\times M\}|, \hspace{0.8cm} \xi_2(v)=|\{t: i_t=v,(t,t+1) \in D\times U\}|,\]
\[\xi_3(v)=|\{t: i_t=v,(t,t+1) \in M\times D\}|,\hspace{0.8cm} \xi_4(v)=|\{t: i_t=v,(t,t+1) \in U\times D\}|;\]
then the number of interest is 
\[\alpha(v)+\delta(v)+\xi_2(v) \leq \alpha(v)+\gamma(v)+\xi_3(v)+\frac{\xi_1(v)+\xi_2(v)}{2}=k+r(v)/2\] since \(v\) is incident with at least as many marked edges as unmarked, 
\[\alpha(v)+\beta(v)+2\gamma(v)+\xi_1(v)+\xi_3(v) \geq \alpha(v)+\beta(v)+2\delta(v)+\xi_2(v)+\xi_4(v).\]
The overall bound for \(\cup_{k \geq 2}{N_\mathbf{i}(k)}\) becomes 
\begin{equation}\label{bdstep484}
    \prod_{(v,k): k \geq 2, v \in N_\mathbf{i}(k)}{(2k+r(v))^{k+r(v)/2}}.
\end{equation}
\par
In the classical case (i.e., cycles), there is no choice for vertices marked once (so long as they are not the first vertex) or the first vertex (conditional on not being marked): that is, the elements of \(N_{\mathbf{i}}(0) \cup (N_{\mathbf{i}}(1)-\{i_0\})\) can be discarded. In the current context, this might not be true.


\begin{lemma}\label{lemma1}
Suppose \(\mathbf{i}_c\) has no edge of multiplicity \(1.\) Then for \(\alpha \in \{0,1\}, v \in N_{\mathbf{i}_c}(\alpha),\) there are at most
\[\alpha+r(v)+\chi_{r(v)>0}\]
more marked edges than unmarked preceding an unmarked edge \((v,*),\) where \(r(v)\) is given by (\ref{defrv}).
\end{lemma}

\begin{proof}
\underline{\(\alpha=1\)}: 
Initially suppose that \(vv\) does not appear in \(\mathbf{i}_c.\) A simple observation, used implicitly in the first two cases below, \(v \not \in \{i_{l_j},i_{r_j}, 1 \leq j \leq s\}\) and \(v \in \{i_{l_j}, 1 \leq j \leq s\}-\{i_{r_j}, 1 \leq j \leq s\},\) is that
no unerased edge \((*,v)\) is also of type \((v,*).\) 
\par
Let \(t_d<...<t_{d_1+1} \leq l_1<r_1 \leq t_1<t_2<\hspace{0.05cm}...\hspace{0.05cm}<t_{d_1}\) be all the appearances of \(v\) in \(\mathbf{i}_c,\) that is
\[\{t_1,t_2,\hspace{0.05cm}...\hspace{0.05cm},t_d\}=\{t: 0 \leq t \leq m-1, i_{t}=v, e_t+e_{t+1}>0, e_0:=e_m\}.\] 
\par
Take \(v \not \in \{i_{l_j},i_{r_j}, 1 \leq j \leq s\}:\) an almost identical rationale to the one for the classical case applies, and since edges now can have odd multiplicity, a slight modification is needed.
This entails \(t_1>r_1:\)
\((i_{t_1-1},i_{t_1})\) is marked, and thus \(((i_{t_k-1},i_{t_k}))_{2 \leq k \leq d}\) are unmarked. For any \(0 \leq h \leq 2q,\) among the first \(h\) edges of \(\mathbf{i}_c\) there are at least as many marked edges incident with \(v\) as unmarked, entailing that in \(\mathbf{i}_c,\) there are either equally many or two more marked than unmarked (the sum of these two numbers is \(2d,\) which is even). The latter case entails no unmarked edge of the type \((v,*)\) exists, while in the former situation, the other unmarked edge is \((i_{t_d},i_{t_d+1})\) due to \((i_{t_z},i_{t_z+1})\) being unmarked implying \((i_{t_y-1}i_{t_y},i_{t_y}i_{t_y+1})_{1 \leq y \leq z}\) can be arranged in pairs of identical edges (there are as many marked as unmarked edges among these \(2z\) edges since their signs, given by the mapping in step \(1,\) are
\[(+,+),(-,+),(-,+),\hspace{0.05cm}...\hspace{0.05cm},(-,+),(-,-)),\] 
whereby \(z=d\) or \((i_{t_{z+1}-1},i_{t_{z+1}})\) is marked, the second option impossible due to \(v \in N_{\mathbf{i}_c}(1);\) since the signs are as above, the only unmarked edge of the type \((v,*)\) has its marked counterpart fully determined by the edges preceding it, entailing the claim.  
\par
Consider now \(v \in \{i_{l_j}, 1 \leq j \leq s\}-\{i_{r_j}, 1 \leq j \leq s\}.\) In this situation,
\(((i_{t_z-1},i_{t_z}))_{1 \leq z \leq d}\) appear in \(\mathbf{i}_c,\) \(d-1\) of them are unmarked, whereas at least one edge among \(((i_{t_z},i_{t_z+1}))_{1 \leq z \leq d}\) is erased. Reasoning as in the previous paragraph gives the difference between the number of marked and unmarked edges is at most \(1\) (one of the edges that could be marked does not belong to \(\mathbf{i}_c,\) implying the difference is strictly smaller than \(2\)). When the difference is \(1,\) there is no unmarked edge of the type \((v,*):\) suppose next it is \(0.\) Then \(v\) is incident with an even number of edges in \(\mathbf{i}_c,\) which can occur solely when
\begin{equation}\label{ineqv2}
   |\{q: 1 \leq q \leq s, i_{l_q}=v\}|=2: 
\end{equation}
the set on the left-hand side has size positive and even, and if it is at least \(3,\) 
then there are at most \(1+(d-3)=d-2<d-1\) marked edges, absurd. 
(\ref{ineqv2}) gives that \(2d-2\) edges in \(\mathbf{i}_c\) are incident with \(v\) and the ones unmarked are among \(((i_{t_z-1},i_{t_z}))_{1 \leq z \leq d},\) entailing no edge of the type \((v,*)\) is unmarked. 
\par
Take \(v \in \{i_{r_j}, 1 \leq j \leq s\}.\) The primary difference between this situation and the cases above is that the absence of some edges among \(((i_{t_z-1},i_{t_z}))_{1 \leq z \leq d}\) can lead to too few unmarked edges, causing the prior arguments to fail. 
Consider the following sequence of sets. Let \(T_1\) consist of the indices \(t \in \{t_1-1,t_1\}\) for which \((i_t,i_{t+1})\) belongs to \(\mathbf{i}_c,\) and when \(|T_1|=2, i_{t_1-1}=i_{t_1+1},\) take instead \(T_1=\emptyset.\) For \(z \geq 2,\) construct \(T_{z}\) from \(T_{z-1}\) as follows: initially let \(T_z=T_{z-1}-\{t\}\) for \((i_t,i_{t+1})\) the marked counterpart of \((i_{t_z-1},i_{t_z})\) when such \(t\) exists and \(T_z=T_{z-1}\) otherwise (i.e., when \((i_{t_z-1},i_{t_z})\) is marked or erased); subsequently for \(t \in \{t_z-1,t_z\},\) add \(t\) to \(T_{z}\) if \((i_{t},i_{t+1})\) is marked. Note
\begin{equation}\label{growth}
    |T_{z+1}|>|T_z|
\end{equation}
entails \((i_{t_{z}-1},i_{t_{z}})\) is the edge marked jointly with \(v,\) or
\((i_{t_{z}},i_{t_{z}+1})\) is marked with \((i_{t_{z}-1},i_{t_{z}})\) not appearing in \(\mathbf{i}_c.\) Furthermore, \(|T_z|\) is the difference between the number of marked and unmarked edges incident with \(v\) and preceding \((i_{t_z},i_{t_z+1}).\)
\par
Let \(1 \leq z_1<z_2<...<z_y \leq d\) be the indices \(\xi\) with \(1 \leq \xi \leq d,\) \(t_\xi>0,\) \((i_{t_\xi-1},i_{t_\xi})\) erased, and suppose \((i_{t_{\xi_0}-1},i_{t_{\xi_0}})\) is the edge marked jointly with \(v.\)  It is shown next via induction on \(\beta\) that for \(1 \leq \beta \leq y,\) 
\begin{equation}\label{finaltz}
    |T_{z_\beta}| \leq 1+\beta+\chi_{z_\beta >\xi_0}.
\end{equation}
\par
For the base case \(\beta=1,\) either \(z_1=1,\) in which case the inequality is clear, or \(z_1 \geq 2\) with 
\par
\((a)\) \(t_1>r_1,\) whereby \(\xi_0=1,\) all edges preceding \((i_{t_{z_1}-1},i_{t_{z_1}})\) belong to \(\mathbf{i}_c\) and have signs
\[(+,*),(-,*),\hspace{0.05cm}...\hspace{0.05cm},(-,*),\]
which gives \(|T_{z_1-1}| \leq 2\) (in light of the observation above on \(|T_z|\)), and \(|T_{z_1}| \leq |T_{z_1-1}|+1\) from \((i_{t_z-1},i_{t_z})\) being erased, or
\par
\((b)\) \(t_1=r_1,\) all the edges preceding \((i_{t_{z_1}-1},i_{t_{z_1}})\) appear in \(\mathbf{i}_c\) and have signs
\[(+),(s_1,*),\hspace{0.05cm}...\hspace{0.05cm},(s_q,*),\]
with all but at most one of \((s_j)_{1 \leq j \leq q}\) being minus, in turn rendering 
\[|T_{z_1-1}| \leq 1+\chi_{z_1-1 \geq \xi_0}\] 
and thus \(|T_{z_1}| \leq |T_{z_1-1}|+1 \leq 2+\chi_{z_1-1 \geq \xi_0}=2+\chi_{z_1> \xi_0}\) (\((i_{t_{z_1}-1},i_{t_{z_1}})\) is erased).
\par
Suppose next \(\beta \geq 2:\) \(((i_{t_j-1},i_{t_j}))_{z_{\beta-1}<j<z_{\beta}}\) appear in \(\mathbf{i}_c\) with at most one of them marked, and \((i_{t_{z_\beta}-1},i_{t_{z_{\beta}}})\) is erased, whereby
\[|T_{z_{\beta}}|\leq |T_{z_{\beta-1}}|+1+\chi_{z_{\beta-1}<\xi_0<z_\beta} \leq (1+\beta-1+\chi_{z_{\beta-1}> \xi_0})+1+\chi_{z_{\beta-1}<\xi_0<z_\beta}=1+\beta+\chi_{z_{\beta}>\xi_0}\]
(there are at most two new marked edges with yet no unmarked counterparts, \((i_{t_{\xi_0}-1},i_{t_{\xi_0}}),(i_{t_{z_\beta}},i_{t_{z_\beta}+1}),\) and \(\xi_0 \not \in \{z_1,z_2,\hspace{0.05cm}...\hspace{0.05cm},z_y\}\)), completing the induction step. 
\par
Note
\[\max_{1 \leq w \leq d}{|T_w|} \leq \max{(2,\max_{1 \leq \beta \leq y}{|T_{z_{\beta}}|})}:\]
for \(z<\max{(z_1,\xi_0)},\) the analysis above yields \(|T_z| \leq 1+\chi_{z \geq \xi_0},\) while for \(z_l<z<z_{l+1},1 \leq l \leq y-1,z>\xi_0,\) \(|T_z| \leq |T_{z-1}|\) in virtue of the necessary condition for (\ref{growth}). This inequality and (\ref{finaltz}) entail
\begin{equation}\label{btzzz}
    \max_{1 \leq w \leq d}{|T_w|} \leq 2+y,
\end{equation}
concluding the justification of the lemma for \(\alpha=1\) and \(vv\) not belonging to \(\mathbf{i}_c.\) 
\vspace{0.2cm}
\par 
Lastly, suppose \(vv\) appears in \(\mathbf{i}_c.\) Then \(v \in N_{\mathbf{i}_c}(1)\) and \(r(v)>0:\) \(v \in \cup_{k \geq 2}{N_{\mathbf{i}_c}(k)}\) when \(m(vv) \geq 3,\) and otherwise \(m(vv)=2, v \in N_{\mathbf{i}_c}(1),\) which can solely occur when \(r(v)>0\) (otherwise the first edge incident with \(v\) is marked jointly with it and is not \(vv,\) implying \(v \in \cup_{k \geq 2}{N_{\mathbf{i}_c}(k)}\)).
Construct the above sets \((T'_z)\) for the lacunary cycle \(\mathbf{i}'_c\) obtained from \(\mathbf{i}_c\) after deleting the two copies of \(vv,\) and employ the same reasoning. The indicator in (\ref{finaltz}) is always \(0\) because \(v\) is not marked in \(\mathbf{i}'_c,\) whereby (\ref{btzzz}) is strict for \((T'_z)\) and so
\[\max_{1 \leq w \leq d}{|T_w|} \leq 1+\max_{1 \leq w \leq d}{|T'_w|} \leq 1+(1+y)=2+y,\] 
completing the case \(\alpha=1.\)

\vspace{0.2cm}
\par
\underline{\(\alpha=0\)}: The rationale above continues to hold (notice that in this situation, \(v \in \{i_{r_q}, 1 \leq q \leq s\}).\) As there is no edge marked jointly with \(v,\) the indicator in (\ref{finaltz}) is always \(0,\) making the inequality (\ref{btzzz}) strict.
\end{proof}

Consider now vertices in \(\cup_{\alpha \in \{0,1\}}{N_{\mathbf{i}_c}(\alpha)},\) and let 
\[m_{\alpha j}=|\{v \in N_{\mathbf{i}_c}(\alpha): |\{t: 1 \leq t \leq s, i_{r_t}=v\}|=j\}|, \hspace{0.4cm} \alpha \in \{0,1\}, j \in \mathbb{N}.\] 
Lemma~\ref{lemma1} entails
alike vertices contribute in this step at most
\begin{equation}\label{bdstep498}
    \prod_{j \geq 1}{j^{jm_{0j}/2}} \cdot \prod_{j \geq 1}{(2+j)^{(1+j/2)m_{1j}}}\ \leq e^{\sum_{j \geq 1}{(j-1)m_{0j}}+\sum_{j \geq 1}{(j+1)m_{1j}}}\prod_{j \geq 1}{(j!)^{m_{0j}}} \cdot \prod_{j \geq 1}{((1+j)!)^{m_{1j}}}
\end{equation}
using (\ref{bdknew}) together with
\[j^{j/2} \leq j^j \leq j!e^{j-1},\]
for \(j \geq 1\) via \(\frac{(j+1)^{j+1}}{(j+1)!} \cdot \frac{j!}{j^j}=(1+\frac{1}{j})^j \leq e,\) and
\[\frac{(2+j)^{1+j/2}}{(1+j)!e^{j+1}} \leq \frac{(2+j)^{1+j/2}}{e(1+j)^{1+j}}=(1+\frac{1}{1+j})^{1+j/2}e^{-1}(1+j)^{-j/2} \leq e \cdot e^{-1}=1, \hspace{0.5cm} j \geq 1.\]
Induction on \(T\) yields that for \(T,l_1,l_2,\hspace{0.05cm}...\hspace{0.05cm},l_T \in \mathbb{N},\)
\begin{equation}\label{factorial}
    \prod_{1 \leq t \leq T}{l_t!} \leq (\sum_{1 \leq t \leq T}{l_t}-T+1)!
\end{equation}
(for \(T \geq 3\) employ the induction hypothesis in conjunction with \((l_1+...+l_{T-1}-(T-1)+1)+l_T-1=l_1+...+l_T-T+1,\) \(T=1\) is clear, and \(T=2\) follows from 
\[\frac{(a+b-1)!}{a! \cdot b!}=\frac{(a+1)(a+2)...(a+b-1)}{(1+1)(1+2)...(1+b-1)} \geq 1, \hspace{0.4cm} a,b \in \mathbb{N}).\]
Three applications of (\ref{factorial}) 
yield
\begin{equation}\label{bdstep4}
    \prod_{j \geq 1}{(j!)^{m_{0j}}} \cdot \prod_{j \geq 1}{((1+j)!)^{m_{1j}}} \leq (1-n(0)+\sum_{j \geq 1}{jm_{0j}})! \cdot (1+\sum_{j \geq 1}{jm_{1j}})! \leq (1-n(0)+\sum_{j \geq 1,\alpha \in \{0,1\}}{jm_{\alpha j}})!
\end{equation}
insomuch as
\begin{equation}\label{jalphanm}
    n(0)=\sum_{j \geq 1}{m_{0j}},\hspace{0.4cm} \sum_{j \geq 1}{m_{1j}} \leq s-n(0), \hspace{0.4cm} \sum_{j \geq 1,\alpha \in \{0,1\}}{jm_{\alpha j}} \leq s.
\end{equation}
The final bound for this step becomes, by virtue of (\ref{bdstep484}),  (\ref{bdstep498}), (\ref{bdstep4}),
\[e^{\sum_{j \geq 1}{(j-1)m_{0j}}+\sum_{j \geq 1}{(j+1)m_{1j}}}\cdot(1-n(0)+\sum_{j \geq 1,\alpha \in \{0,1\}}{jm_{\alpha j}})! \cdot \prod_{(v,k): k \geq 2, v \in N_\mathbf{i}(k)}{(2k+r(v))^{k+r(v)/2}}.\]
\par
\underline{Step 5:} Let \(e(\mathbf{i}_c)\) be the number of pairwise distinct undirected edges appearing in \(\mathbf{i}_c:\) 
\[e(\mathbf{i}_c) \geq \sum_{k \geq 1}{n_k}-s+n(0)\] 
via the same rationale as for cycles (now the correction is given by \(s-n(0),\) the number of marked vertices in \(\{i_{r_w}, 1 \leq w \leq s\}).\) Together with \(e(\mathbf{i}_c) \leq q\) and (\ref{boundtilde}), this gives
\begin{equation}\label{boundexp}
    |\mathbb{E}[\Tilde{a}_{\mathbf{i}}]| \leq \tilde{p}^{e(\mathbf{i}_c)}(1-\tilde{p})^{2q} \cdot (\frac{1}{1-\tilde{p}})^{e(\mathbf{i_c})} \leq \tilde{p}^{-s+n(0)+\sum_{k \geq 1}{n_k}}(1-\tilde{p})^q.
\end{equation}
For the analogue of Lemma~\ref{l1}, the exact reasoning gives
\begin{equation}\label{ineq20}
    \sum_{k \geq 2}{kn_k} \geq 4o-T-n(\mathbf{i}), \hspace{0.5cm} \sum_{k \geq 2}{kn_k} \geq 3T-n(\mathbf{i}), \hspace{0.5cm} \sum_{k \geq 2}{(k-1)n_k} \geq 2o-n(\mathbf{i})
\end{equation}
for
\[n(\mathbf{i}):=|N_{\mathbf{i}}(1) \cap \{i_{r_j}, 1 \leq j \leq s\}|,\]
since \(v \in N_{\mathbf{i}}(1) \cap \{i_{r_j}, 1 \leq j \leq s\}\) is incident with at most one edge \(e \in \mathcal{O}\) with a copy marked jointly with \(v,\) the only edges for which the argument in the proof of the lemma breaks down (the possible exception was \(i_0i_1\) for cycles; this is replaced now by edges of multiplicity \(3\) with their first copy having as left endpoint an element of \(\{i_{r_j}, 1 \leq j \leq s\}\)). Hence
\begin{equation}\label{anl1}
    \sum_{k \geq 2}{kn_k} \geq \frac{3}{4} \cdot (4o-T-n(\mathbf{i}))+\frac{1}{4} \cdot (3T-n(\mathbf{i}))=3o-n(\mathbf{i}).
\end{equation}
Both (\ref{ineq20}) and (\ref{anl1}) as well as \(n(\mathbf{i}) \leq s-n(0)\) will be used shortly.
\par
The bound for the contribution of \(\mathbf{i}_c\) becomes, using steps \(1-5,\) \(\sum_{k \geq 1}{n_k}=q+o-\sum_{k \geq 2}{(k-1)n_k},\)
\[[\binom{2q-1}{q-o}-\binom{2q-1}{q-o-2}] \cdot n^{q} (\tilde{p}(1-\tilde{p}))^{q} \cdot \sum_{(n_k)_{k \geq 1}}{\tilde{p}^{n(0)-s}n^{n(0)} \cdot \tilde{p}^{o-\sum_{k \geq 2}{(k-1)n_k}} \cdot}\]
\[\cdot n^{o-\sum_{k \geq 2}{(k-1)n_k}} \cdot \frac{(q+o)!}{\prod_{k \geq 1}{(k!)^{n_k}}} \cdot \frac{1}{\prod_{k \geq 1}{n_k!}} \prod_{k \geq 2}{(2k)^{kn_k}} \cdot e^{\sum_{j \geq 1}{(j-1)m_{0j}}+\sum_{j \geq 1}{(j+1)m_{1j}}} \cdot \]
\begin{equation}\label{bounddee}
    \cdot (1-n(0)+\sum_{j \geq 1,\alpha \in \{0,1\}}{jm_{\alpha j}})! \prod_{(k,v): k \geq 2, v \in N_{\mathbf{i}_c}(k)}{(1+\frac{r(v)}{2k})^{k} \cdot (2k+r(v))^{r(v)/2}}.
\end{equation}
Since for \(x \in \mathbb{R},\) \(x+1 \leq \exp(x),\)
\[\prod_{(k,v): k \geq 2, v \in N_{\mathbf{i}_c}(k)}{(1+\frac{r(v)}{2k})^{k}} \leq \prod_{(k,v): k \geq 2, v \in N_{\mathbf{i}_c}(k)}{\exp(\frac{r(v)}{2})} \leq \exp(\frac{s-n(0)}{2}) \leq 2^{s-n(0)},\]
\[\prod_{(k,v): k \geq 2, v \in N_{\mathbf{i}_c}(k)}{(2k+r(v))^{r(v)}} \leq (2q)^{\sum_{(k,v): k \geq 2, v \in N_{\mathbf{i}_c}(k)}{r(v)}} \leq (2q)^{s-\sum_{j \geq 1,\alpha \in \{0,1\}}{jm_{\alpha j}}}\]
and in light of (\ref{jalphanm}),
\[e^{\sum_{j \geq 1}{(j-1)m_{0j}}+\sum_{j \geq 1}{(j+1)m_{1j}}} \cdot (1-n(0)+\sum_{j \geq 1,\alpha \in \{0,1\}}{jm_{\alpha j}})! \leq e^{s-2n(0)+\sum_{j \geq 1, \alpha \in \{0,1\}}{jm_{\alpha j}}} \cdot (1+s-2n(0)+\sum_{j \geq 1,\alpha \in \{0,1\}}{jm_{\alpha j}})!,\]
yielding the product of the last three factors in (\ref{bounddee}) is at most
\[2^{s-n(0)}\max_{x \in [n(0),s] \cap \mathbb{Z}}{[e^{s-2n(0)+x} \cdot (1+s-2n(0)+x)! \cdot (2q)^{s-x}]} \leq \]
\[\leq 2^{s-n(0)}(3e)^{s-n(0)} \cdot e^{s-n(0)} \cdot (1+s-n(0))! \cdot (2q)^{s-n(0)} \leq (2 \cdot 3e^2 \cdot 2(s-n(0)) \cdot 2q)^{s-n(0)} \leq (256qs)^{s-n(0)}\]
using that for \(x \leq s-1,\)
\[\frac{e^{s-2n(0)+x+1} \cdot (1+s-2n(0)+x+1)!(2q)^{s-x-1}}{e^{s-2n(0)+x} \cdot (1+s-2n(0)+x)!(2q)^{s-x}} \leq \frac{e(1+s-2n(0)+s)}{2q} \leq \frac{3es}{2q} \leq \frac{3es}{s}=3e\]
(\(\mathbf{i}_c\) has \(s\) components and \(2q\) unerased edges, whereby \(s \leq 2q\)), and \((k+1)! \leq (2k)^k.\)
\par
Thus (\ref{bounddee}) can be replaced by
\[[\binom{2q-1}{q-o}-\binom{2q-1}{q-o-2}] \cdot n^{q} (\tilde{p}(1-\tilde{p}))^{q} \cdot \sum_{(n_k)_{k \geq 1}}{\tilde{p}^{n(0)-s}n^{n(0)}n^{wn(\mathbf{i})}  \cdot n^{o-2ow}\tilde{p}^o \cdot}\]
\[\cdot (256qs)^{s-n(0)} \cdot \tilde{p}^{-\sum_{k \geq 2}{(k-1)n_k}} \cdot n^{-(1-w)\sum_{k \geq 2}{(k-1)n_k}} \cdot \frac{(q+o)!}{\prod_{k \geq 1}{(k!)^{n_k}}} \cdot \frac{1}{\prod_{k \geq 1}{n_k!}} \cdot \prod_{k \geq 2}{(2k)^{kn_k}}\]
with the aid of (\ref{ineq20}), which gives
\[\sum_{k \geq 2}{(k-1)n_k} \geq w\cdot (2o-n(\mathbf{i}))+(1-w) \cdot \sum_{k \geq 2}{(k-1)n_k}.\]
Via the rationale ensuing (\ref{intbound}) and the replacement of (\ref{bounddee}) above, the contribution of a given \(o\) is at most 
\[2 \cdot [\binom{2q-1}{q-o}-\binom{2q-1}{q-o-2}] \cdot n^{q+s} (p(1-p))^{q} \cdot ((4eq)^{2}n^{-w})^o \cdot\]
\[\cdot \max_{n(0),n(\mathbf{i})}{[(256qs)^{s-n(0)}\tilde{p}^{n(0)-s}n^{n(0)-s}n^{wn(\mathbf{i})} \cdot (\frac{2e(q+o)}{(\tilde{p}n^{1-w})^{1/2}})^{-2n(\mathbf{i})/3}]}.\]
The maximum can be replaced by \(1\) since \(n(\mathbf{i}) \leq s-n(0)\) and \(\frac{2e(q+o)}{(\tilde{p}n^{1-w})^{1/2}} \leq \frac{4eq}{(\tilde{p}n^{1-w})^{1/2}} \leq 1,\) from which
\[(256qs)^{s-n(0)}\tilde{p}^{n(0)-s}n^{n(0)-s}n^{wn(\mathbf{i})} \cdot (\frac{2e(q+o)}{(\tilde{p}n^{1-w})^{1/2}})^{-2n(\mathbf{i})/3} \leq\]
\begin{equation}\label{decay}
    \leq [(512q^2)^{-1} \cdot \tilde{p}n^{1-w} \cdot \frac{(\tilde{p}n^{1-w})^{1/3}}{(2eq)^{2/3}}]^{n(0)-s} \leq (\frac{\tilde{p}n^{1-w}}{128q^2})^{4(n(0)-s)/3} \leq 1.
\end{equation}
Therefore, the range \(o>0\) can be absorbed in the second term of the desired identity (\ref{laccyc1}), and anew the largest contribution corresponds to \(o=0,\) which is looked at next. 
\par
When each component (i.e., the paths contained in \(\mathbf{i}_c\) with endpoints given by \((r_t,l_{t+1})_{1 \leq t \leq s}, l_{s+1}:=l_1\)) belongs to \(\cup_{l \in \mathbb{N}}{\mathcal{C}(l)},\) any two edge disjoint, the contribution is
\[n^{q+s}(1-O(\frac{q^2}{n}))(\tilde{p}(1-\tilde{p}))^q C_{m_1}...C_{m_s},\]
using \(q^2 \leq cn^{1/2},\) and (\ref{exppeq}) for \(m=q.\) If \(n(0)<s,\) then a factor of \((\frac{128q^2}{\tilde{p}n^{1-w}})^{4/3} \leq (\frac{128q^{2}}{\tilde{p}n^{1-w}})^{1/3}\) can be added in virtue of (\ref{decay}): else \(i_{r_q}=i_{l_{q+1}}\) for all \(1 \leq q \leq s\) (\(o=0\) and \(|\{i_{r_q}, 1 \leq q \leq s\}|=s\) entail 
\[\{i_{l_q}, 1 \leq q \leq s\}=\{i_{r_q}, 1 \leq q \leq s\},\] 
and if \(q_0\) is minimal with \(i_{l_{q_0+1}} \ne i_{r_{q_0}},\) then \(i_{l_{q_0+1}} \in \{i_{r_q}, 1 \leq q \leq s\}\) is marked, absurd). If \(\sum_{k \geq 2}{n_k}>0,\) then the bound can be multiplied with 
\(O(\frac{q^{2/3}}{(\tilde{p}n^{1-w})^{1/3}})\)
(in light of the rationale succeeding (\ref{intbound}), especially (\ref{expbound})). Otherwise, each undirected edge has multiplicity \(2\) with two of the \(s\) cycles (\(i_{r_q}=i_{l_{q+1}}\) for all \(q\)) sharing an edge: two such components can be merged into one by deleting two edges (see (\ref{paste1}) and (\ref{paste2})), and thus, by induction on \(q\) for (\ref{sub231}), the contribution is at most
\[8s^2q^2 \cdot p(1-p)\cdot C C_{q-1}n^{(q-1)+(s-1)}(p(1-p))^{q-1}=8s^2q^2n^{-2} \cdot CC_{q-1}n^{q+s}(p(1-p))^{q}\leq CC_qn^{q+s}(p(1-p))^{q}\]
using \(8s^2q^2 \leq 32q^4 \leq 32c^2\tilde{p}n \leq 32c^2n\) (\(2s^2(2q)^2\) accounts for the sizes of the preimages: \(s^2\) corresponds to the positions of the segments that were glued into one, and \(2q \cdot 2q \cdot 2\) to the separation of one of the segments into two cycles by selecting the position of the shared edge, the first vertex of the second cycle, and its orientation). The desired result follows.

\vspace{0.4cm}
\par
\textit{Proof of (\ref{laccyc2}):} 
Similarly to the even case, after discounting 
the erased edges, the bound becomes
\[O(n^{q+s}(\tilde{p}(1-\tilde{p}))^qC_q).\]
Employ the same marking procedure as in the proof of (\ref{laccyc2}), let now \(2o+1=|\mathcal{O}|,\) and proceed as in subsection~\ref{1.2} with the aid of (\ref{ineqq}).

\begin{lemma}\label{lemma388}
    If \(\mathbf{i}_c\) has no edge of multiplicity \(1,\) then
    \begin{equation}\label{ineqq}
        -o+(s-1)-\sum_{v \in \{i_{r_j}: 1 \leq j \leq s\}}{\chi_{v \in N_{\mathbf{i}_c}(0)}}+\sum_{k \geq 1}{(k-1)n_k} \geq \frac{1}{2}\sum_{k \geq 1}{(k-1)n_k},
    \end{equation}
    unless \(\mathbf{i}_c\) is a union of \(s\) vertex disjoint paths of the type \((u,\mathcal{L}_1,u,v,\mathcal{L}_2,v,u,\mathcal{L}_3,u,v,\mathcal{L}_4,v)\) with \(u\ne v,\) each of the cycles 
    \[(u,\mathcal{L}_1,u),(v,\mathcal{L}_2,v),(u,\mathcal{L}_3,u),(v,\mathcal{L}_4,v)\] 
    either has length \(0\) or belongs to \(\cup_{q \in \mathbb{N}}{\mathcal{C}(q)},\) and no two have nontrivial vertex intersections.
\end{lemma}

\begin{proof}
Note that
\[-o+(s-1)-\sum_{v \in \{i_{r_j}: 1 \leq j \leq s\}}{\chi_{v \in N_{\mathbf{i}_c}(0)}}+\sum_{k \geq 1}{(k-1)n_k}=s-\frac{1}{2}-\sum_{v \in \{i_{r_j}: 1 \leq j \leq s\}}{\chi_{i_{r} \in N_{\mathbf{i}_c}(0)}}+(-o-\frac{1}{2}+\sum_{k \geq 1}{(k-1)n_k})=\]
\[=s-\frac{1}{2}-\sum_{v \in \{i_{r_j}: 1 \leq j \leq s\}}{\chi_{v \in N_{\mathbf{i}_c}(0)}}+\sum_{k \geq 1, v \in N_{\mathbf{i}}(k)}{(k-1-\frac{c(v)}{2})},\]
where \(c(v)\) is the number of last marked copies of elements of \(\mathcal{O}\) marked jointly with \(v\) (that is, for each \(e \in \mathcal{O},\) consider solely the last marked copy of \(e,\) and \(c(v)\) counts such directed edges whose right endpoint is \(v\)). For \(v \in N_{\mathbf{i}}(k), v \not \in \{i_{r_j}, 1 \leq j \leq s\},\) 
\[c(v) \leq k-1\]
(the first edge \(e\) marked jointly with \(v\) is the first appearance of \(e,\) and all edges of odd multiplicity are marked at least twice), 
and clearly \(c(i_{r_j}) \leq k_0\) when \(i_{r_j} \in N_{\mathbf{i}}(k_0),\) whereby
\[-o+(s-1)-\sum_{v \in \{i_{r_j}: 1 \leq j \leq s\}}{\chi_{v \in N_{\mathbf{i}_c}(0)}}+\sum_{k \geq 1}{(k-1)n_k}=s-\frac{1}{2}-\sum_{v \in \{i_{r_j}: 1 \leq j \leq s\}}{\chi_{v \in N_{\mathbf{i}_c}(0)}}+\sum_{k \geq 1, v \in N_{\mathbf{i}}(k)}{(k-1-\frac{c(v)}{2})}\geq\]
\begin{equation}\label{lastineq2}
    \geq (s-\frac{1}{2})-\sum_{1 \leq j \leq s}{(\chi_{i_{r_j} \in N_{\mathbf{i}_c}(0)}+\frac{c(i_{r_j})}{2} \cdot \chi_{i_{r_j} \in N_{\mathbf{i}_c}(1)}+\frac{1}{2} \cdot \chi_{i_{r_j} \in N_{\mathbf{i}_c}(k_j), c(i_{r_j})=k_j>1})}+\sum_{k \geq 2, v \in N_{\mathbf{i}}(k)}{\frac{k-1}{2}}.
\end{equation}
\par
Since for all \(j,\)
\[\chi_{i_{r_j} \in N_{\mathbf{i}_c}(0)}+\frac{c(i_{r_j})}{2} \cdot \chi_{i_{r_j} \in N_{\mathbf{i}_c}(1)}+\frac{1}{2} \cdot \chi_{i_{r_j} \in N_{\mathbf{i}_c}(k_j), c(i_{r_j})=k_j>1} \leq 1,\]
the desired inequality ensues when \(|\{i_{r_j}, 1 \leq j \leq s\}|<s\) or \(|\{i_{r_j}, 1 \leq j \leq s\} \cap N_{\mathbf{i}_c}(0)|<s\) because in these cases at least one of the summands in the first sum is at most \(\frac{1}{2}\) and \(s-1+\frac{1}{2}=s-\frac{1}{2}.\) Suppose next
\[|\{i_{r_j}, 1 \leq j \leq s\}|=s, \hspace{0.8cm} R(\mathbf{i}_c):=\{i_{r_j}, 1 \leq j \leq s\}= N_{\mathbf{i}_c}(0).\]
Then 
\[(s-\frac{1}{2})-\sum_{1 \leq j \leq s}{(\chi_{i_{r_j} \in N_{\mathbf{i}_c}(0)}+\frac{c(i_{r_j})}{2} \cdot \chi_{i_{r_j} \in N_{\mathbf{i}_c}(1)}+\frac{1}{2} \cdot \chi_{i_{r_j} \in N_{\mathbf{i}_c}(k_j), c(i_{r_j})=k_j>1})}=-\frac{1}{2},\]
and it suffices to show \(c(v)=k-1\) for all \(v \in N_{\mathbf{i}}(k), k \geq 1\) occurs solely in the configurations described in the statement of the lemma (this would imply (\ref{lastineq2}) is strict in any other scenario, yielding the conclusion). The following observation will be used repeatedly: the existence of an edge marked jointly with \(v\in N_{\mathbf{i}}(k), k \geq 1\) that is neither the last copy of an element of \(\mathcal{O},\) nor the first edge marked jointly with \(v\) entails \(c(v)<k-1\) (there are exactly \(k-1-c(v)\) positions that yield such edges). Such a configuration must have no edge of multiplicity at least \(4:\) when an edge \(e_0\) has \(m(e_0) \geq 4,\) the vertex \(v_0 \in N_{\mathbf{i}}(k_0)\) marked jointly with the second marked copy of \(e_0\) has \(c(v_0)<k_0-1\) insomuch as the aforesaid copy is neither the first edge marked jointly with \(v_0\) (\(v_0\) is marked, whereby \(v_0 \not \in R(\mathbf{i}_c),\) and for alike vertices the first edge incident with \(v_0\) is marked jointly with it), nor the last marked copy of \(e_0\) when \(e_0 \in \mathcal{O}\) (\(m(e_0) \geq 5\) in this case, yielding at least \(3\) marked copies of it). Furthermore, edges with equal endpoints are also prohibited (if \(v_0v_0\) appears among the edges, then its first copy is neither the first marked edge jointly with \(v_0\) nor the last marked copy of an edge of odd multiplicity).
\par
It is proved next that any lacunary cycle \(\mathbf{i}_c\) with
\par
\((i)\) all edges of multiplicity \(2\) or \(3,\) 
\par
\((ii)\) no edge with equal endpoints, 
\par
\((iii)\) all right endpoints \(\{i_{r_j}, 1 \leq j \leq s\}\) unmarked, pairwise distinct, 
\par
\((iv)\) \(|\mathcal{O}| \geq 1,\) and  
\par
\((v)\) \(c(w)=k-1\) for all \(w \in N_{\mathbf{i}_c}(k),k \geq 1 \newline\) 
is of the type described in the lemma statement: this provides the desired conclusion in light of the two observations above. Suppose for the sake of contradiction that there exists a lacunary cycle \(\mathbf{i}_c\) that satisfies \((i)-(v)\) but does not have the structure aforementioned, and take one of minimal length \(m_0>0.\)
\par
If there is an edge \(uv\) with \(m(uv)=2,\) then its copies belong either to the same segment or to distinct ones. In the latter case, the two segments can be rearranged such that the two copies are consecutive: change
\((S_1,u,v,S_2),(S_3,u,v,S_4)\) 
to 
\[(S_1,u,v,u,\overline{S}_3), (\overline{S}_2,v,S_4),\] 
and
\((S_1,u,v,S_2),(S_3,v,u,S_4)\) 
to 
\[(S_1,u,v,u,S_4),
(\overline{S}_2,v,\overline{S}_3),\] 
where 
\[\overline{(u_1,u_2,\hspace{0.05cm}...\hspace{0.05cm},u_k)}:=(u_k,u_{k-1},\hspace{0.05cm}...\hspace{0.05cm},u_1).\]
This yields a contradiction: consider the lacunary cycle \(\mathbf{i}'_c\) obtained by rearranging the two segments aforesaid and leaving out \((u,v,u).\) 
\(\mathbf{i}'_c\) has at most \(s\) components, whereby
\begin{equation}\label{ineqqr}
    |V(\mathbf{i}'_c)| \leq |R(\mathbf{i}'_c)|+|\mathcal{O}(\mathbf{i}'_c)|+|\mathcal{E}(\mathbf{i}'_c)| \leq s+|\mathcal{O}(\mathbf{i}_c)|+|\mathcal{E}(\mathbf{i}_c)|-1=|V(\mathbf{i}_c)|-1
\end{equation}
using that any \(w \not \in R(\mathbf{i}'_c)\) is marked at least once, and the first edge \(e\) incident with \(w \not \in R(\mathbf{i}'_c)\) is marked jointly with it and is the first copy of \(e\) from the justification of (\ref{unmarked2}) (\(\mathcal{E}(\mathbf{j})=\{e \in \mathbf{j}: 2|m(e)\}\)). Since \(V(\mathbf{i}_c)-V(\mathbf{i}'_c) \subset \{v\},\) all inequalities in (\ref{ineqqr}) are equalities, entailing \(\mathbf{i}'_c\) has \(s\) components, absurd since in this case \(V(\mathbf{i}'_c)=V(\mathbf{i}_c)\) (\(v\) appears in \(\mathbf{i}'_c\) as well because it belongs to the second component of the rearrangement aforesaid).
\par
In the former case, either the segment which contains them can be rearranged such that the two copies of \(uv\) are consecutive (change \((S_1,u,v,S_2,u,v,S_3)\) to \((S_1,u,v,u,\overline{S}_2,v,S_3)\)), rendering a contradiction by discarding anew \((u,v,u)\) (this new lacunary cycle \(\mathbf{i}'_c\) also satisfies (\ref{ineqqr}), and both inequalities must be equalities: this gives \(v \not \in V(\mathbf{i}'_c),\) and \(\mathbf{i}'_c\) continues to have \(s\) components, whereby \((i)-(v)\) hold despite \(\mathbf{i}'_c\) having \(m_0-2 \in (0,m_0)\) edges), or the two copies of \(uv\) have different orientations: this second scenario is also impossible. Suppose the copies were \((u,v),(v,u),\) and let \(k_0 \in \mathbb{N}\) satisfy \(v \in N_{\mathbf{i}_c}(k_0);\) \(v\) is marked jointly with \(uv,2|m(uv)=2,\) whereby \(uv\) must be the first edge with which \(v\) is incident (else \(c(v)<k_0-1\)); because \(v\) is marked, it does not belong to \(R(\mathbf{i}_c),\) and the second time \(v\) appears in \(\mathbf{i}_c\) is marked (the corresponding edge is the first copy of an edge \(e_0\) insomuch as it is not a copy of \(uv\)), contradicting \(c(v)=k_0-1.\) 
\par
Suppose now all edges in \(\mathbf{i}_c\) have multiplicity \(3.\) Note \(m_0>3:\) else \(s=1\) (from \((iii)\)), and so \(\mathbf{i}_c=(u,v,u,v,u)\) for \(u \ne v,\) absurd.
Consider the undirected graph formed from \(\mathbf{i}_c\) by discarding duplicate edges: this either is a forest or contains a cycle.
\par
In the first case, there is a vertex \(v\) connected to solely one edge \(uv\) (a leaf). 
\par
\((a)\) If \(v \not \in \{i_{l_j},i_{r_j}, 1 \leq j \leq s\},\) then it is incident with an even number of edges (since \(vv\) does not belong to \(\mathbf{i}_c\)), whereby \(2|m(uv)=3,\) absurd. 
\par
\((b)\) If \(v \in \{i_{l_j}, 1 \leq j \leq s\},\) then the three copies of \(uv\) must be \(\{(u,v,u,v)\},\{(u,v,u),(u,v)\},\) or \(\{(u,v),(u,v),(u,v)\}.\) Consider the lacunary cycle \(\mathbf{i}'_c\) with these three edges left out, and note the number of components can only decrease (any cluster of copies of \(uv\) must have \(v\) one of the endpoints because \(v\) is incident only to \(uv\)). It is shown \(\mathbf{i}'_c\) also satisfies \((i)-(v),\) which suffices to derive a contradiction
(its length is \(m_0-3 \in (0,m_0)\)). Note \(V(\mathbf{i}_c)-V(\mathbf{i}'_c) \subset \{u,v\}:\) suppose \(u \in V(\mathbf{i}_c)-V(\mathbf{i}'_c).\) This entails the segments containing \(uv\) should contain no other edge, and thus (\ref{ineqqr}) can be refined to
\[|V(\mathbf{i}'_c)| \leq |R(\mathbf{i}'_c)|+|\mathcal{O}(\mathbf{i}'_c)|+|\mathcal{E}(\mathbf{i}'_c)| \leq s-1+|\mathcal{O}(\mathbf{i}_c)|+|\mathcal{E}(\mathbf{i}_c)|-1=|V(\mathbf{i}_c)|-2,\]
entailing equality occurs, from which \(u,v\) do not appear in \(\mathbf{i}'_c,\) which inherits the properties \((i)-(v).\) Take now \(u \not \in V(\mathbf{i}_c)-V(\mathbf{i}'_c),\) whereby (\ref{ineqqr}) becomes equality because \(\{v\} \subset V(\mathbf{i}_c)-V(\mathbf{i}'_c)\) from \(v\) being incident solely with \(uv:\) reasoning as below (\ref{ineqq}) implies \(\mathbf{i}'_c\) has exactly \(s\) components, \(V(\mathbf{i}'_c)=V(\mathbf{i}_c)-\{v\},\)
and \(|R(\mathbf{i}'_c)|=s.\) Clearly \((i),(ii),(iv)\) hold for \(\mathbf{i}'_c,\) \((iii)\) ensues from \(R(\mathbf{i}'_c)=R(\mathbf{i}_c),\) while \((v)\) follows due to the function \(c\) not changing and \(u\) not being marked jointly with \(uv\) in \(\mathbf{i}_c\) (otherwise \(v \in R(\mathbf{i}_c)\) since \(v\) is not incident with any edge \(e' \ne uv,\) which gives \(|R(\mathbf{i}'_c)| \leq s-1,\) absurd).
\par
\((c)\) If \(v \in \{i_{r_j}, 1 \leq j \leq s\},\) then \(uv\) must the first edge marked jointly with \(u\) (\(v\) being a leaf implies two copies of \(uv\) are marked jointly with \(u\)). The clusters of \(uv\) are \(\{(v,u),(u,v,u)\}\) since \(|R(\mathbf{i}_c)|=s;\) if \((v,u)\) appears before \((u,v,u)\), then this is the first time \(u\) shows up in \(\mathbf{i}_c,\) and all subsequent edges marked jointly with \(u\) must be last copies; this cannot occur unless \(u \in R(\mathbf{i}_c)\) since there is another edge incident with \(u\) preceding \((u,v,u),\) which is a first copy of an edge, absurd; else \((u,v,u)\) appears before \((u,v),\) entailing anew that either \(u \in R(\mathbf{i}_c)\) or the first edge marked jointly with it is not \(uv,\) absurd. Lastly, note \(u \not \in  R(\mathbf{i}_c)\) due to \(u\) being marked.
\par
In the second case, suppose \(v_1,v_2, \hspace{0.05cm}... \hspace{0.05cm},v_{k_0}\) form a cycle of minimal length: these vertices are pairwise distinct and none belongs to \(R(\mathbf{i}_c)\) (suppose \(v_1 \in R(\mathbf{i}_c):\) then \(v_1v_2\) is the first edge incident with \(v_2\) because two copies of it are marked jointly with \(v_2,\) and by induction on \(r,\) \(v_{r}v_{r+1}\) is the first edge incident with \(v_{r+1},\) entailing \(v_{k_0+1}=~v_1\) is marked, absurd). It is shown next that for some \(k' \leq k_0,s' \leq s,\)
\begin{center}
\(\mathbf{i}_c\) can be rearranged into a collection of cycles \((\mathcal{L}_i)_{1 \leq i \leq k'},\) paths \((\mathcal{S}_j)_{1 \leq j \leq s'},\) and 
\end{center}
\begin{center}
\(\mathcal{L}_{cyc},\) the cycle \((v_1,v_2,\hspace{0.05cm}...\hspace{0.05cm},v_k)\) with the loops \((v_j,v_{j+1},v_j)\) attached at \(v_j\) for \(1 \leq j \leq k,\) 
\end{center}
\begin{center}
such that for all \(1 \leq i \leq k',\) \(\mathcal{L}_i\) contains no element of \(R(\mathbf{i}_c),\) is vertex disjoint with 
\end{center}
\begin{center}
\((\mathcal{L}_{i'})_{1 \leq i' \leq k', i' \ne i},(\mathcal{S}_j)_{1 \leq j \leq s'},\) and has the first vertex in \(\{v_1,v_2, \hspace{0.05cm}... \hspace{0.05cm},v_{k_0}\}.\)     
\end{center}
This suffices to derive a contradiction: mark the vertices in a new lacunary cycle \(\mathbf{i}'_c\) obtained by reading \(\mathcal{L}_{cyc},\) \((\mathcal{S}_j)_{1 \leq j \leq s'},\) and \((\mathcal{L}_i)_{1 \leq i \leq k'}\) in this order (for \(1 \leq r \leq k',\) take \(v_{j_r}\) to be the first vertex in \(\mathcal{L}_r,\) and arbitrary orientations for \((\mathcal{S}_j)_{1 \leq j \leq s}, (\mathcal{L}_i)_{1 \leq i \leq k'}).\) Then
\[|V(\mathbf{i}'_c)| \leq k_0+\sum_{k \geq 1}{n''_k}+\sum_{1 \leq q \leq s'}{\chi_{i_{r'_q} \in N_{\mathbf{i}'_c(0)}}} \leq k_0+\frac{(m_0-3k_0)-(|\mathcal{O}|-k_0)}{2}+s' \leq \frac{m-|\mathcal{O}|}{2}+s=|V(\mathbf{i}_c)|,\]
where \(n''_k=|N_{\mathbf{i}'_c}(k)-\{v_1,v_2,\hspace{0.05cm}...\hspace{0.05cm},v_{k_0}\}|,\) and \((i_{r'_q})_{1 \leq q \leq s'}\) are the right endpoints of \((\mathcal{S}_j)_{1 \leq j \leq s'},\) using that 
\[\frac{(m_0-3k_0)-(|\mathcal{O}|-k_0)}{2} \geq |\mathcal{E}(\mathbf{i}'_c)|+|\mathcal{O}(\mathbf{i}'_c)-\{v_1v_2,v_2v_3,\hspace{0.05cm}...\hspace{0.05cm},v_{k_0}v_1\}|,\]
(via the injective mapping from \(V(\mathbf{i}'_c)-R(\mathbf{i}'_c)\) to \(\mathcal{E}(\mathbf{i}'_c) \cup \mathcal{O}(\mathbf{i}'_c)\) that takes a vertex \(v\) to the first undirected edge with which it is marked jointly). Hence equality occurs above (\(V(\mathbf{i}'_c)=V(\mathbf{i}_c)\)), whereby \(s'=s\) and \((i_{r'_q})_{1 \leq q \leq s'}\) are unmarked and pairwise distinct: this entails the lacunary cycle \(\mathbf{i}''_c\) consisting of \((\mathcal{S}_j)_{1 \leq j \leq s'}\) satisfies \((i)-(v),\) contradicting the minimality of \(m_0\) since \((v)\) can be justified via the following inequalities that hold for any lacunary cycles that satisfy \((i),(ii),(iii),\)
\[|\mathcal{O}(\mathbf{i}_c)| \leq \sum_{u}{c(u)} \leq \sum_{k \geq 1}{(k-1)n_k} \leq \sum_{k \geq 1}{kn_k}-(|V(\mathbf{i}_c)|-|R(\mathbf{i}_c)|),\]
\[\sum_{k \geq 1}{kn_k}-(|V(\mathbf{i}_c)|-|R(\mathbf{i}_c)|) \leq 2|\mathcal{O}(\mathbf{i}_c)|+|\mathcal{E}(\mathbf{i}_c)|-(|\mathcal{O}(\mathbf{i}_c)|+|\mathcal{E}(\mathbf{i}_c)|)=|\mathcal{O}(\mathbf{i}_c)|.\]
Note \(\mathbf{i}''_c\) is nonempty because it contains the elements of \(R(\mathbf{i}_c),\) and the edges incident with them.
\par
Return now to the proof of the claim on the rearrangement of \(\mathbf{i}_c.\) Since each edge has multiplicity \(3,\) the lacunary cycle can be rearranged such that two of the copies of \(v_rv_{r+1}\) have consecutive positions for all \(1 \leq r \leq k_0\) (\(v_{k_0+1}:=v_1),\) without enlarging the number of segments (proceed as in the case \(m(uv)=2\) above using two copies that share orientation). By putting aside the loops of length \(2\) and deleting the remaining copies of \(v_1v_2,v_2v_3,\hspace{0.05cm}...\hspace{0.05cm},v_{k_0}v_1,\) the rest can be rearranged into \(s' \leq s\) segments and cycles with first vertices among \(\{v_1,v_2,\hspace{0.05cm}...\hspace{0.05cm},v_{k_0}\}\): for \(1 \leq r \leq k_0,\) cut in two the segment to which the third copy of \(v_{r}v_{r+1}\) belongs by leaving out this edge; at the end of the process, there are \(s+k_0\) segments (some potentially empty), and their endpoints (including multiplicities) are \(v_1,v_1,v_2,v_2,\hspace{0.05cm}...\hspace{0.05cm},v_{k_0},v_{k_0},i_{l_1},i_{l_2},\hspace{0.05cm}...\hspace{0.05cm},i_{l_s},i_{r_1},i_{r_2},\hspace{0.05cm}...\hspace{0.05cm},i_{r_s}.\) 
For each \(1 \leq j \leq k_{0},\) if there are two segments with endpoints \(v_j,\) join them, and else do nothing. The final configuration consists of \(s\) segments with endpoints \(i_{l_1},i_{l_2},\hspace{0.05cm}...\hspace{0.05cm},i_{l_s},i_{r_1},i_{r_2},\hspace{0.05cm}...\hspace{0.05cm},i_{r_s},\) and \(k' \leq k\) loops, each containing at least one vertex among \(v_1,v_2,\hspace{0.05cm}...\hspace{0.05cm},v_{k_0}\) (suppose a resulting segment had endpoints \(v_r,u, u \ne v_r;\) if solely one segment had one of its endpoints \(v_r,\) then this would be the sole segment containing \(v_r\) and so no other vertex in the multiset would belong to it; else the two segments with endpoints \(v_r\) were not joined because \(v_r\) is the endpoint of a segment, absurd; hence the outcomes are loops containing vertices among \(v_1,v_2,\hspace{0.05cm}...\hspace{0.05cm},v_{k_0}\) or segments with endpoints among \(i_{l_1},i_{l_2},\hspace{0.05cm}...\hspace{0.05cm},i_{l_s},i_{r_1},i_{r_2},\hspace{0.05cm}...\hspace{0.05cm},i_{r_s}\)). 
Now repeat another procedure until no change is produced: if there are loops that share a vertex with any of the remaining segments or loops, then attach the loop of this type with smallest index to the corresponding segment or loop with which it shares a vertex and treat it as a part of it (i.e., reduce the number of loops by \(1\) and throw away its index). This process must end in a finite amount of steps as each iteration shrinks the number of loops by \(1:\) the final graph consists of \(s' \leq s\) segments (some could be empty) and \(k'' \leq k'\) loops, no loop sharing any vertex with any of the remaining loops and segments (in particular, these final loops do not have vertices in \(R(\mathbf{i}_c)\) as the latter are the endpoints of the segments).
\end{proof}

\par
Claim (\ref{laccyc2}) can be now concluded: the lacunary cycles for which (\ref{ineqq}) holds can be treated as in the proof of Proposition~\ref{prop2}. The rest consist of \(s\) paths \(((u_r,\mathcal{L}_{1r},u_r,v_r,\mathcal{L}_{2r},v_r,u_r,\mathcal{L}_{3r},u_r,v_r,\mathcal{L}_{4r},v_r))_{1 \leq r \leq s}\) with 
\((u_r,\mathcal{L}_{1r},u_r),(v_r,\mathcal{L}_{2r},v_r),(u_r,\mathcal{L}_{3r},u_r),(v_r,\mathcal{L}_4,v_r) \in \cup_{q \in \mathbb{N}}{\mathcal{C}(q)} \cup \{(u_r),(v_r), 1 \leq r \leq s\}.\) These cycles contribute at most
\[n^{2s} \cdot n^{(2q+1-3s)/2} \cdot (p(1-p))^{(2q+1-3s)/2+s} \cdot |1-2p|^s \sum_{l_1+l_2+...+l_s=(2q+1-3s)/2}{C_{l_1}C_{l_2}...C_{l_s}} \leq\]
\[\leq n^{q+(1+s)/2} (p(1-p))^{q+(1-s)/2} C_{q-(s-1)/2} \leq \frac{1}{2}n^{q+s}(p(1-p))^qC_q \cdot (2np(1-p))^{-(s-1)/2}=O(n^{q+s}(p(1-p))^qC_q)
\]
via
\[2np(1-p)=2n\tilde{p}(1-\tilde{p}) \geq n\tilde{p} \geq n^w \geq 1, \hspace{0.6cm} C_{k+t} \geq 2^{t-1}C_k, \hspace{0.6cm}\sum_{l_1+l_2+...+l_t=k}{C_{l_1}C_{l_2}...C_{l_t}} \leq C_{k+t-1},\]
the last two inequalities following by induction on \(t \in \mathbb{Z}_{\geq 0},\) and \(C_{k+1}=\sum_{l_1+l_2=k}{C_{l_1}C_{l_2}} \geq (1+\chi_{k>0})C_k.\)
\par
The proofs of (\ref{laccyc1}) and (\ref{laccyc2}) are therefore complete.

\subsection{Even Powers of \(A\)}\label{1.4}

Having pinned down the dominant components for lacunary cycles with fixed split points (Proposition~\ref{propqs}), proceed with even powers of \(A.\)

\begin{proposition}\label{proptra}
There exist universal \(c,C>0\) such that if \(w>0, \tilde{p} \geq n^{w-1}, m^2 \leq c\min{(n^w,\tilde{p}n^{1-w})},\) then
\begin{equation}\label{Amom}
    \mathbb{E}[tr(A^{2m})] \geq (np)^{2m}+2m \cdot (np)^{2m-1}(1-p)(1-\frac{Cm^2}{np})+C_m n^{m+1}(p(1-p))^m(1-\frac{2m^2}{n}),
\end{equation}
\begin{equation}\label{Amom2}
      \mathbb{E}[tr(A^{2m})] \leq (np)^{2m}+2m \cdot (np)^{2m-1}(1-p) (1+\frac{Cm^2}{np})+C_m n^{m+1}(p(1-p))^m(1+C \cdot (\frac{m^2}{\tilde{p}n^{1-w}})^{1/3}).  
\end{equation}   
Furthermore,
\begin{equation}\label{Amom2odd}
      \mathbb{E}[tr(A^{2m+1})] \leq (np)^{2m+1}+(2m+1) \cdot (np)^{2m}(1-p) (1+\frac{Cm^2}{np})+C \cdot C_m n^{m+1}(p(1-p))^m(1+\frac{m^2}{(\tilde{p}n)^{1/2}}),  
\end{equation} 
\[(1-2p) \cdot (\mathbb{E}[tr(A^{2m+1})]-(np)^{2m+1}-(2m+1) \cdot (np)^{2m}(1-p)[1+O(\frac{m^2}{np})]) \geq\]
\begin{equation}\label{Amomodd}
    \geq C_{m-1} n^{m}(p(1-p))^{m-1}(1-2p)^2 \cdot (1-\frac{2(m-1)^2}{n}).
\end{equation}
\end{proposition}

\par
Notice that all terms in both (\ref{Amom}) and (\ref{Amom2}) are necessary: the first always dominate the second (\(c \leq 2^{-2},\) whereby \(\tilde{p}n \geq (2\sqrt{c})^{4/3},\) and \(np \geq \tilde{p}n \geq 2\sqrt{c} \cdot (\tilde{p}n)^{1/4} \geq 2m \geq 2m(1-p)\)), while the relation between them and the third depends on the size of \(m\) (the two orderings change at \(m \propto \frac{\log{n}}{\log{\frac{np}{1-p}}}\)). A notable difference between even and odd powers consists in the third terms in their corresponding lower and upper bounds: for the former these are tight, whereas for the latter they are not, the lower bound being likely the more accurate of the two. 
\par
\begin{proof}[Proof of (\ref{Amom}), (\ref{Amom2})]
Recall the decomposition (\ref{trlac}),
\[\mathbb{E}[tr(A^{2m})]=\sum_{\mathbf{i}=(i_0,i_1,i_2,\hspace{0.05cm}...\hspace{0.05cm},i_{2m-1},i_0,e_1,e_2,\hspace{0.05cm}...\hspace{0.05cm},e_{2m})}{\mathbb{E}[a_{\mathbf{i}}]},\]
where the summation is over lacunary cycles of length \(2m.\) Consider first the extremal lacunary cycles (i.e., \(e_1=...=e_{2m}\)). If no edge is erased (\(e_1=...=e_{2m}=1\)), then the contribution is \(\mathbb{E}[tr(\Tilde{A}^{2m})],\) given by (\ref{evenmom}), 
while if all edges are erased (\(e_1=...=e_{2m}=0\)), then \((np)^{2m}\) arises. It remains to see what occurs when there is a mix of both categories: it is shown that the more erased edges, the larger the overall contribution; in particular, the second term in (\ref{Amom}) comes from cycles with all but two consecutive edges erased. Besides, as the moments of the centered trace are of interest, obtaining the primary contributors from random cycles (i.e., \(e_1+...+e_{2m}>0\)) is crucial.
\par
\underline{Case 1:} Suppose a lacunary cycle has \(s \geq 1\) components, \(2q+1\) unerased edges, and fixed split points. If \(q=0,\) then all expectations are \(0;\) assume next \(q \geq 1.\) The interior vertices of the erased edges (with the notation in (\ref{zeroededges}), their positions are the elements of \(\cup_{1 \leq t \leq s}{\{l_{t}+1,\hspace{0.05cm}...\hspace{0.05cm},r_{t}-1\}}\)) contribute \(n^{2m-2q-1-s}\) (there are \(2m-2q-1-s\) hidden vertices, that is \(|\cup_{1 \leq t \leq s}{\{l_{t}+1,\hspace{0.05cm}...\hspace{0.05cm},r_{t}-1\}}|=2m-2q-1-s\), and these are arbitrary elements of \(\{1,2, \hspace{0.05cm} ... \hspace{0.05cm},n\}\)). Proposition~\ref{propqs} yields the contribution is
\[O(n^{q+s}p^{2m-q-1}C_{q+1}(1-p)^q) \cdot n^{2m-2q-1-s}=O((np)^{2m-q-1}(1-p)^qC_{q+1}).\]
\par
The number of lacunary cycles (up to vertex isomorphisms) with \(E\) unerased edges and \(s\) components is at most
\[(2m)^s \cdot \binom{E-1}{s-1}\]
because there are \(\binom{(E-s)+s-1}{s-1}=|\{(x_1,x_2, \hspace{0.05cm} ... \hspace{0.05cm},x_s):x_1+...+x_s=E, x_i \in \mathbb{N}\}|\) possibilities for the lengths of the segments and at most \((2m)^s\) ways to choose the positions of the left endpoints (elements of \(\{0, 1, \hspace{0.05cm} ... \hspace{0.05cm},2m-1\}\)), from which the overall contribution for \(q,s\) fixed is
\[O((np)^{2m-q-1}(1-p)^qC_{q+1}) \cdot (2m)^s\binom{2q}{s-1}.\]
Using
\[\sum_{s \leq 2q}{(2m)^s\binom{2q}{s-1}}=2m \cdot [(2m+1)^{2q}-(2m)^{2q}] \leq 2q(2m+1)^{2q},\] 
for given \(q,\) the final bound is
\begin{equation}\label{bound1qerased}
    O((np)^{2m-q-1}(1-p)^qqC_{q+1}) \cdot (2m+1)^{2q}.
\end{equation}
\par
\underline{Case 2:} Suppose a lacunary cycle has \(s \geq 1\) components, \(2q\) unerased edges, \(1 \leq q \leq m-1,\) and fixed split points. Similarly to \underline{Case 1}, the interior vertices add a factor of \(n^{2m-2q-s},\) whereby in conjunction with Proposition~\ref{propqs} the contribution is
\[O(n^{q+s}p^{2m-q}(1-p)^{q}C_{q+s}) \cdot n^{2m-2q-s}=O(C_{q+s}(np)^{2m-q}(1-p)^{q})\]
from 
\[\sum_{m_1+...+m_s=m,m_i \geq 0}{C_{m_1}...C_{m_s}} \leq C_{m+s-1}.\]
Arguing in the same vein as in \underline{Case 1}, the overall contribution is
\[O(C_{q+s}) \cdot (np)^{2m-q}(1-p)^q \cdot (2m)^s\binom{2q-1}{s-1},\]
and summing over \(s \leq 2q\) yields
\begin{equation}\label{bound2qerased}
    O(qC_{3q}) \cdot (np)^{2m-q}(1-p)^q \cdot (2m+1)^{2q-1}.
\end{equation}
When \(q=1,\) the dominant graphs are even cycles of length \(2\) with \(s=1,\) and give the second terms in both (\ref{Amom}) and (\ref{Amom2}):
each individual expectation is \(p^{2m-1}(1-p),\) the graphs with \(s=2\) contribute at most 
\[2 \cdot (2m)^2n^{2m-4+2} \cdot p^{2m-1}(1-p)=2m \cdot (np)^{2m-1}(1-p) \cdot \frac{4m}{n} \leq 2m \cdot (np)^{2m-1}(1-p) \cdot \frac{m^2}{np},\] 
whereas those with \(s=1\) yield 
\[2m \cdot n^{2m-3+2} \cdot p^{2m-1}(1-p)=2m \cdot (np)^{2m-1}(1-p).\] 
In light of (\ref{bound1qerased}) and (\ref{bound2qerased}), the remainder generates \(2m \cdot (np)^{2m-1}(1-p) \cdot O(\frac{m^2}{np})\) since \(m^2 \leq c \cdot \tilde{p}n^{1-w} \leq c \cdot np,\) 
\begin{equation}\label{domm1}
    \frac{(np)^{2m-q-1}qC_{q+1}(2m+1)^{2q}(1-p)^q}{2m \cdot (np)^{2m-1}(1-p)}=\frac{qC_{q+1}(2m+1)^{2q}(1-p)^{q-1}}{2m \cdot(np)^{q}} \leq (\frac{8 \cdot 9m^2}{np})^{q},
\end{equation}
\begin{equation}\label{domm2}
    \frac{(np)^{2m-q}(1-p)^{q}qC_{3q}(2m+1)^{2q-1}}{2m \cdot (np)^{2m-1}(1-p)}=\frac{(1-p)^{q-1}qC_{3q}(2m+1)^{2q-1}}{2m \cdot(np)^{q-1}} \leq 48 \cdot (\frac{64 \cdot 9m^2}{np})^{q-1}
\end{equation}
from \(C_l \leq 4^l, q \leq 2^{q-1}, 2m+1 \leq 3m.\) 
\end{proof}

\begin{proof}[Proof of (\ref{Amom2odd}), (\ref{Amomodd})]
Begin with (\ref{Amom2odd}): the extremal lacunary cycles (\(e_{1}=e_2=...=e_{2m+1}\)) yield the first and third terms, and the second ensues via an identical rationale to the one above (for the remainder of the lacunary cycles, solely the parity of the number of unerased edges matters, implying replacing \(2m\) by \(2m+1\) suffices).
\par
Take now (\ref{Amomodd}): the argument for (\ref{Amom2odd}) implies the contribution of lacunary cycles with at least one erased edge is
\[(np)^{2m+1}+(2m+1) \cdot (np)^{2m}(1-p) \cdot [1-O(\frac{m^2}{np})].\]
The remainder is \(tr(\tilde{A}^{2m+1}):\) a consequence of (\ref{boundtilde}) is
\[(1-2p)^{2m+1}\mathbb{E}[tr(\tilde{A}^{2m+1})]=\sum_{\mathbf{i}=(i_0,i_1,\hspace{0.05cm}...\hspace{0.05cm},i_{2m},i_0)}{|\mathbb{E}[\tilde{a}_i]|}\]
as shown immediately after its statement, and the inequality ensues from the contributions of the cycles that are elements of \(\mathcal{C}(m-1)\) with a three loops of length \(1\) attached at the first vertex: i.e., cycles
\[\mathbf{i}=(i_0,i_0,i_0,i_1,\hspace{0.05cm}...\hspace{0.05cm},i_{2m-3},i_0), \hspace{0.5cm} (i_0,i_1,\hspace{0.05cm}...\hspace{0.05cm},i_{2m-3},i_0) \in \mathcal{C}(m-1),\]
and Proposition~\ref{prop1}.
\end{proof}

\subsection{Moments of \(tr(A^{2m})-\mathbb{E}[tr(A^{2m})]\)}\label{1.5}

This subsection justifies (\ref{convhighmom}) via the results on the dominant components of the trace expectations in subsections~\ref{1.1}-\ref{1.4} with \(n^w=(\tilde{p}n)^{1/2}.\) 
\par
Let \(l \geq 2:\) 
\begin{equation}\label{Ahighmom}
    \mathbb{E}[(tr(A^{2m})-\mathbb{E}[tr(A^{2m})])^l]=\sum_{(\mathbf{i}_1,\mathbf{i}_2, \hspace{0.05cm} ... \hspace{0.05cm},\mathbf{i}_l)}{\mathbb{E}[(a_{\mathbf{i}_1}-\mathbb{E}[a_{\mathbf{i}_1}]) \cdot (a_{\mathbf{i}_2}-\mathbb{E}[a_{\mathbf{i}_2}]) \cdot ... \cdot (a_{\mathbf{i}_l}-\mathbb{E}[a_{\mathbf{i}_l}])]},
\end{equation}
with \(\mathbf{i}_1,\mathbf{i}_2, \hspace{0.05cm} ... \hspace{0.05cm},\mathbf{i}_l\) lacunary cycles of length \(2m.\) Construct a simple undirected graph \(\mathcal{G}\) with vertices \(1,2,\hspace{0.05cm}...\hspace{0.05cm},l,\) in which \(ab \in E(\mathcal{G})\) if only if \(\mathbf{i}_a, \mathbf{i}_b\) share an edge that is unerased in both and \(a \ne b.\) The contributors to (\ref{Ahighmom}) are tuples for which all connected components of \(\mathcal{G}\) have size at least two (else the expectation is \(0\) by independence), and the contributions of the components are independent as no two share an edge. High moments are treated in \cite{sinaisosh} via gluing cycles along shared edges. Two cycles \(\mathbf{i},\mathbf{j}\) with a common edge are merged into a cycle \(\mathcal{P}\) of length \(4m,\) by employing the first shared edge \(e\) as a bridge from one to the other: let \(i_{t-1}i_{t}=j_{s-1}j_{s}\) with \(t,s\) minimal in this order (i.e., \(t=\min{\{1 \leq k \leq 2m, \exists 1 \leq q \leq 2m, i_{k-1}i_{k}=j_{q-1}j_{q}\}},\newline s=\min{\{1 \leq q \leq 2m, j_{q-1}j_{q}=i_{t-1}i_{t}\}}\)). Walk along \(\mathbf{i}\) up to \(i_{t-1}i_{t},\) use it as a bridge to switch to \(\mathbf{j},\) traverse all of it, and get back to \(\mathbf{i}\) upon returning to \(j_{s-1}j_s=i_{t-1}i_{t}:\) specifically, if \((i_{t-1},i_{t})=(j_{s},j_{s-1}),\) then
\begin{equation}\label{paste1}
    \mathcal{P}:=(i_0, \hspace{0.05cm} ... \hspace{0.05cm},i_{t-1},i_t,j_{s-2}, \hspace{0.05cm} ... \hspace{0.05cm},j_0,j_{2m-1}, \hspace{0.05cm} ... \hspace{0.05cm}, j_{s},j_{s-1},i_{t+1}, \hspace{0.05cm} ... \hspace{0.05cm},i_{2m});
\end{equation}
else \((i_{t-1},i_{t})=(j_{s-1},j_{s}),\) and 
\begin{equation}\label{paste2}
    \mathcal{P}:=(i_0, \hspace{0.05cm} ... \hspace{0.05cm}, i_{t-1},i_t,j_{s+1}, \hspace{0.05cm} ... \hspace{0.05cm}, j_{2m-1},j_0, \hspace{0.05cm} ... \hspace{0.05cm}, j_{s-1},j_s,i_{t+1}, \hspace{0.05cm} ... \hspace{0.05cm}, i_{2m}).
\end{equation}
In \cite{sinaisosh}, \((i_{t-1},i_{t})\) and \((j_{s-1},j_{s})\) are deleted, giving rise to a cycle of length \(4m-2,\) whose expectation is roughly proportional to that of 
\(\mathcal{P}.\) The authors show that solely tuples whose corresponding graphs have \(l/2\) connected components, each of size \(2,\) yield a non-negligible contribution (this generates \((l-1)!!=|\{\{S_1,S_2,\hspace{0.05cm}...\hspace{0.05cm},S_{l/2}\}:|S_1|=...=|S_{l/2}|=2, S_1 \cup S_2 \cup ... \cup S_{l/2}=\{1,2,\hspace{0.05cm}...\hspace{0.05cm},l\}|\}|).\)
\par
In the current situation, the deletions aforesaid are an issue when \(i_{t-1}i_{t}\) has multiplicity \(3\) in \(\mathcal{P}\) (the expectation of the remaining cycle is \(0,\) whereas this might not be true for \(\mathcal{P}\)), and Lemma~\ref{lem999} replaces this procedure. It yields anew that solely the graphs with \(l/2\) connected components make non negligible contributions, producing the desired conclusion as by virtue of the case \(l=2\) (justified below) and
\[\chi_{2|l} \cdot (l-1)!! \cdot [2(2m)^2(np)^{4m-2}p(1-p) \cdot (1+o(1))]^{l/2}=\]
\begin{equation}\label{llargerthan2}
    =\chi_{2|l} \cdot (l-1)!! \cdot 2^{l/2}(2m)^{l}(np)^{(2m-1)l}(p(1-p))^{l/2} \cdot (1+o(1)).
\end{equation}
\par
Begin with a simplification of (\ref{Ahighmom}). Lemma~\ref{lemma888} below and (\ref{boundtilde}) yield
\[|\mathbb{E}[(a_{\mathbf{i}_1}-\mathbb{E}[a_{\mathbf{i}_1}]) \cdot (a_{\mathbf{i}_2}-\mathbb{E}[a_{\mathbf{i}_2}]) \cdot ... \cdot (a_{\mathbf{i}_l}-\mathbb{E}[a_{\mathbf{i}_l}])]| \leq 2^{l-1}|\mathbb{E}[a_{\mathbf{i}_1} \cdot a_{\mathbf{i}_2} \cdot ... \cdot a_{\mathbf{i}_l}]|.\]

\begin{lemma}\label{lemma888}
    For \(s_1,s_2,\hspace{0.05cm}...\hspace{0.05cm},s_t \in \mathbb{N},0 \leq p \leq 1,\)
    \begin{equation}\label{ineq1999}
        (1-p)^{s_1+s_2+...+s_t-1}+p^{s_1+s_2+...+s_t-1} \geq \frac{1}{2^t} \cdot (p(1-p))^{t-1} \cdot \prod_{1 \leq i \leq t}{[(1-p)^{s_i-1}+p^{s_i-1}]},
    \end{equation}
    \begin{equation}\label{ineq2999}
        |(1-p)^{s_1+s_2+...+s_t-1}-p^{s_1+s_2+...+s_t-1}| \geq \frac{1}{2^t} \cdot (p(1-p))^{t-1} \cdot |(1-p)^{s_1-1}-p^{s_1-1}| \cdot \prod_{2 \leq i \leq t}{[(1-p)^{s_i-1}+p^{s_i-1}]}.
    \end{equation}
\end{lemma}

\begin{proof}
Suppose \(p \in (0,\frac{1}{2}]\) as the claims are not only immediate for \(p \in \{0,1\},\) but also symmetric in \(p,1-p.\) It suffices to show that for \(a \geq 1,\)
\[a^{s_1+s_2+...+s_t-1}+1 \geq \frac{1}{2^t} \cdot a^{t-1} \cdot \prod_{1 \leq i \leq t}{(a^{s_i-1}+1)},\]
\[a^{s_1+s_2+...+s_t-1}-1 \geq \frac{1}{2^t} \cdot a^{t-1}  \cdot (a^{s_1-1}-1)\prod_{2 \leq i \leq t}{(a^{s_i-1}+1)}\]
insomuch as these for \(a=\frac{1-p}{p} \geq 1,\) and \(p \leq 1\) entail the desired results.
\par
The first inequality follows from 
\[\frac{1}{2^t} \cdot a^{t-1}  \cdot \prod_{1 \leq i \leq t}{(a^{s_i-1}+1)} \leq \frac{1}{2^t} \cdot a^{t-1} \cdot \prod_{1 \leq i \leq t}{2a^{s_i-1}}=a^{s_1+s_2+...+s_t-1} \leq a^{s_1+s_2+...+s_t-1}+1,\]
while for the second, employ the first and
\[a^{n+m-1}-1 \geq \frac{a}{2} \cdot (a^{n-1}-1)(a^{m-1}+1)\]
for \(n,m \in \mathbb{N},\) the latter holding because
\[2(a^{n+m-1}-1)- a(a^{n-1}-1)(a^{m-1}+1)=a^{n+m-1}+a^m+a-2-a^{n}=a^{n}(a^{m-1}-1)+(a^m+a-2) \geq 0.\]

\end{proof}

To justify (\ref{llargerthan2}), showing graphs \(\mathcal{G}\) with a connected component of size \(L>2\) yield negligible contributions is key, and this is done by merging tuples of lacunary cycles into one. 
Suppose \(\mathbf{i}_1,\mathbf{i}_2,\hspace{0.05cm}...\hspace{0.05cm},\mathbf{i}_L\) form a connected component: let \(e(\mathbf{i}_k)\) be the first undirected edge in \(\mathbf{i}_k\) that appears in one of the remaining cycles and that is unerased in both, as well as 
\[\{e_1,e_2,\hspace{0.05cm}...\hspace{0.05cm}, e_r\}=\{e(\mathbf{i}_k), 1 \leq k \leq L\}.\] 
Consider the lacunary cycle \(\mathbf{i}\) of length \(2mL+L\) formed from \(\mathbf{i}_1,\mathbf{i}_2,\hspace{0.05cm}...\hspace{0.05cm},\mathbf{i}_L\) by adding an erased edge between the endpoints of consecutive cycles and erasing all copies of  \(e_1,e_2,\hspace{0.05cm}...\hspace{0.05cm}, e_r:\)
\begin{equation}\label{gluedsum}
    \mathbf{i}=(i_{10},i_{11},\hspace{0.05cm}...\hspace{0.05cm},i_{1(2m+1)},i_{20},i_{21},\hspace{0.05cm}...\hspace{0.05cm},i_{2(2m+1)},\hspace{0.05cm}...\hspace{0.05cm},i_{L0},i_{L1},\hspace{0.05cm}...\hspace{0.05cm},i_{L(2m+1)},(e_{k}(\mathbf{i}))_{1 \leq k \leq (2m+1)L}),
\end{equation}
\[\mathbf{i}_j=(i_{j0},i_{j1},\hspace{0.05cm}...\hspace{0.05cm},i_{j(2m+1)}), \hspace{0.2cm} 1 \leq j \leq L,\]
\[e_{k}(\mathbf{i})=\begin{cases}
    0, \hspace{1.5cm} k \in \{2m+1,2 \cdot (2m+1),\hspace{0.05cm}...\hspace{0.05cm},L \cdot (2m+1)\},\\
    0, \hspace{1.5cm} k=(2m+1)\cdot j+t, \hspace{0.2cm} 1 \leq j \leq L, 1 \leq t \leq 2m, \hspace{0.2cm} e_t(\mathbf{i}_j)=(u,v), \hspace{0.2cm} uv \in \{e_1,e_2,\hspace{0.05cm}...\hspace{0.05cm}, e_r\},\\
    e_t(\mathbf{i}_j), \hspace{0.8cm} k=(2m+1)\cdot j+t, \hspace{0.2cm} 1 \leq j \leq L, 1 \leq t \leq 2m, \hspace{0.2cm} e_t(\mathbf{i}_j)=(u,v), \hspace{0.2cm} uv \not \in \{e_1,e_2,\hspace{0.05cm}...\hspace{0.05cm}, e_r\}.
\end{cases}
\]
The map \((\mathbf{i}_1,\mathbf{i}_2,\hspace{0.05cm}...\hspace{0.05cm},\mathbf{i}_L) \to \mathbf{i}\) is injective, and
\[\mathbb{E}[a_{\mathbf{i}_1} \cdot a_{\mathbf{i}_2} \cdot ... \cdot a_{\mathbf{i}_L}]=p^{-L-\sum_{1 \leq j \leq r}{m_{\mathbf{i}}(e_j)}} \cdot \prod_{1 \leq j \leq r}{\mathbb{E}[\Tilde{a}^{m_{\mathbf{i}}(e_j)}_{11}]} \cdot \mathbb{E}[a_{\mathbf{i}}]:\]
such lacunary cycles \(\mathbf{i}\) are analyzed next.
There are at most
\[(2m)^L \cdot (2mL \cdot 2L)^{\sum_{1 \leq j \leq r}{m_{\mathbf{i}}(e_j)}-L}\]
arrangements for the positions of the copies of \((e_j)_{1 \leq j \leq r}:\) each cycle must have one alike edge, generating the first factor, all the other copies have at most \(2mL-L \leq 2mL\) possibilities for their positions, with at most \(2L\) choices for the edges (orientation matters), and their number is \(\sum_{1 \leq j \leq r}{m_{\mathbf{i}}(e_j)}-L.\) Each fixed configuration of these edges (positions and orientations for the copies of \(e(\mathbf{i}_1),e(\mathbf{i}_2),\hspace{0.05cm}...\hspace{0.05cm},e(\mathbf{i}_L)\)) yields a lacunary cycle with at least 
\[L+\sum_{1 \leq j \leq r}{m_{\mathbf{i}}(e_j)}\]
erased edges, and Lemma~\ref{lem999} below renders a crucial upper bound. Before stating this result, additional notation is needed. 
\par
Fix \(t \in \mathbb{N},\) and let \(L,R:\{1,2,\hspace{0.05cm}...\hspace{0.05cm},t\} \to \{1,2,\hspace{0.05cm}...\hspace{0.05cm},2t\}\) be arbitrary functions. These induce a multigraph \(\mathcal{G}\) with vertices \(\{1,2,\hspace{0.05cm}...\hspace{0.05cm},2t\},\) and directed edges \(((L(q),R(q)))_{q \in \{1,2,\hspace{0.05cm}...\hspace{0.05cm},t\}}.\) Denote by \(C(\mathcal{G})\) the maximal clusters of contiguous identical edges: tuples \(c=(t_1,t_1+1,\hspace{0.05cm}...\hspace{0.05cm},t_2)\) with
\[
R(q)=L(q+1), L(q)=R(q+1), \hspace{0.4cm} q \in \{t_1,t_1+1,\hspace{0.05cm}...\hspace{0.05cm},t_2-1\},\]
\[(R(q)-L(q+1))^2+(L(q)-R(q+1))^2>0,  \hspace{0.4cm} q \in \{t_1-1,t_2\}.\] 
Let
\[e(c):=L(t_1)R(t_1),\hspace{0.5cm} |c|:=|\{t_1,t_1+1,\hspace{0.05cm}...\hspace{0.05cm},t_2-1,t_2\}| \]
with indices taken modulo \(t,\) \(L(0):=L(t),R(0):=R(t),\) and
\[C_2(\mathcal{G}):=\{c: c \in C(\mathcal{G}), e(c)=uv,u \neq v, \exists c' \in C(\mathcal{G}), c' \ne c, e(c')=e(c)\}.\] 
For an edge \(e',\) denote by \(s(e')\) the number of pairs of consecutive clusters consisting of copies of it and separated by at least one element of the tuple \((L(q)R(q))_{1 \leq q \leq t}\) when \(e'=uv,u \ne v,\) and \(0\) otherwise: for \(\{c_1,c_2, \hspace{0.05cm}... \hspace{0.05cm},c_k\}:=\{c: c \in C(\mathcal{G}),e(c)=e'\}\) with \(c_k=(t_{1k},t_{1k}+1,\hspace{0.05cm}...\hspace{0.05cm},t_{2k}), t_{11}<t_{12}<...<t_{1k},\) 
\begin{equation}\label{sdef}
 s(e'):=|\{j: 1 \leq j \leq k, L_{t_{2j}+1}R_{t_{2j}+1} \ne e'\}\}| \cdot \chi_{e'=uv,u \ne v},  
\end{equation}
say \(c_{j},c_{j+1}\) are \textit{neighbors}, call them \textit{separated} when \(L_{t_{2j}+1}R_{t_{2j}+1} \ne e',\) and let
\[m_{\mathcal{E}}(e):=|\{q:q \in \{1,2,\hspace{0.05cm}...\hspace{0.05cm},t\} ,L(q)R(q)=e\}|,\]
\[m_{\mathcal{V}}(y):=|\{q: q \in \{1,2,\hspace{0.05cm}...\hspace{0.05cm},t\}, L(q)R(q)=yy',y \ne y',m_\mathcal{E}(L(q)R(q))=1\}|;\]
lastly, for lacunary cycles \(\mathbf{i},\mathbf{j},\) say \(\mathbf{j}\) is \textit{vertex homomorphic} to \(\mathbf{i}\) if the two have the same length \(m,\) and there exists a function \(T:\mathbb{Z} \to \mathbb{Z}\) with \(T(i_{k})=j_k\) for all \(0 \leq k \leq m,\) where \((i_0,i_1,\hspace{0.05cm}...\hspace{0.05cm},i_m),(j_0,j_1,\hspace{0.05cm}...\hspace{0.05cm},j_m)\) are the (ordered) vertices of \(\mathbf{i},\mathbf{j},\) respectively. 

\begin{lemma}\label{lem999}
    Suppose the conditions in Proposition~\ref{proptra} hold, \(t \in \mathbb{N},\) let \(L,R:\{1,2,\hspace{0.05cm}...\hspace{0.05cm},t\} \to\{1,2,\hspace{0.05cm}...\hspace{0.05cm},2t\}\) be arbitrary functions, 
    and \(1 \leq k_1<k_2<...<k_{t} \leq m.\) 
    There exists a universal constant \(C>0\) such that for \(\tilde{p}n \geq n_0,\) the contribution to \(\mathbb{E}[tr(A^m)]\) of lacunary cycles vertex homomorphic to lacunary cycles \(\mathbf{i}\) with
    \[e_{k_q}(\mathbf{i})=0, (i_{k_q-1},i_{k_q})=(L(q),R(q)),  1 \leq q \leq t, \hspace{0.4cm} |\{uv: u \ne v, m_{\mathcal{E}}(uv)=1\}|>0\]
    is at most
    \begin{equation}\label{formulaa}
        (np)^{m} \cdot n^{-\sum_{c \in C(\mathcal{G})}{(|c|-\chi_{c \not \in C_2(\mathcal{G})})-\frac{1}{2}\sum_{e \in E}{(s(e)-2)_{+}}-\sum_{v \in V}{(m_{\mathcal{V}}(v)-1)_{+}}}}\cdot C^{|C(\mathcal{G})|} \cdot (1+\frac{Cm^2}{\tilde{p}n}),
    \end{equation}
    where \(E=\{L(q)R(q): 1 \leq q \leq t\},V=\{L(q),R(q),1 \leq q \leq t\},\) and \(x_{+}:=\max{(x,0)}.\)
\end{lemma}

\begin{proof}
    Proceed by induction on \(t+m:\) \(m=1\) is immediate as \(\mathbb{E}[tr(A)]=np.\) A consequence of Proposition~\ref{proptra} is that the contribution to \(\mathbb{E}[tr(A^m)]\) of lacunary cycles with at least one erased edge is
    \[(np)^m+O(m(np)^{m-1})\]
    because the third term in both bounds can be dispensed with. 
    This suffices to justify the claim for \(t=1.\) 
    \par
    Let now \(t \geq 2,\) and assume \(|c| \leq 2\) for all \(c \in C(\mathcal{G}):\) when there exists a cluster \(c_0\) with \(|c_0|>2,\) the induction hypothesis can be applied by dispensing with \(2 \cdot \lfloor \frac{|c_0|-1}{2} \rfloor \geq 2\) of its edges (drop 
    \[(L(q),R(q))_{t_1+1 \leq q \leq t_2-\chi_{2|t_1+t_2+1}}\] 
    for \(c_0=(t_1,t_1+1,\hspace{0.05cm}...\hspace{0.05cm},t_2),\) take \(t'=t-2\cdot \lfloor \frac{|c_0|-1}{2} \rfloor\) instead of \(t,\) and relabel the lacunary cycles of interest via the natural isomorphism between \(\{1,2,\hspace{0.05cm}...\hspace{0.05cm},k_1,k_1+k_2,\hspace{0.05cm}...\hspace{0.05cm},k_1+k_2+k_3-1\}\) and \(\{1,2,\hspace{0.05cm}...\hspace{0.05cm},k_1+k_3\}\) for \(k_1,k_2,k_3 \in \mathbb{N}\)). The induction hypothesis can be anew applied when there exists \(c \in C(\mathcal{G}),e(c)=uu:\) by discarding it and using the induction hypothesis on such smaller cycles, the sum in the second factor in (\ref{formulaa}) decreases by \(|c|,\) yielding the desired bound as \((np)^{-|c|}n^{|c|}p^{|c|}=1.\) Assume next \(L(q) \ne R(q)\) for all \(1 \leq q \leq t.\)
    \par
    \underline{Case \(1:\)} \(C_2(\mathcal{G}) \ne \emptyset.\)
    \par
    Take \(v \in V\) 
    a minimizer for the function \(d:V \to (0,\infty]\) given by
    \[d(u)=\min_{(x,y):x \ne y,i_x=i_y=u, \exists q,r, x \in \{k_q,k_q-1\},y \in \{k_r,k_r-1\}}{\min{(|x-y|,m-|x-y|)}}:\]
    since \(|C_2(\mathcal{G})|>0,\) the minimum is finite, and consider two positions, \(x,y,\) at which it is attained. Denote the two cycles induced by \(x,y\) (\(i_x=i_y=v\)) along with their lengths by \(\mathcal{L}_1,\mathcal{L}_2, m_1,m_2,\) respectively, with \(m_1 \leq  m_2,\) and let \(\mathcal{G}_1,\mathcal{G}_2\) be the induced graphs by \(\mathcal{G}\) in each of them. Note that no two equal elements of the multiset \((i_{k_q-1},i_{k_q})_{1 \leq q \leq t}\) belong to \(\mathcal{L}_1\) (this would contradict the definition of \(v\)), and \(\mathcal{L}_1\) or \(\mathcal{L}_2\) contains an edge \(L(s)R(s)\) of multiplicity \(1.\)
    \par
    \((i)\) \(\mathcal{L}_1,\mathcal{L}_2\) share no unerased edge: assume first both have an erased edge of multiplicity \(1\) among the initial \(t\) and use the induction hypothesis. Fix \(uw \in E\) with clusters appearing in both \(\mathcal{G}_1\) and \(\mathcal{G}_2\): the definition of \(v\) entails
    \[|C_2(\mathcal{G}_1) \cap \{c,e(c)=uw\}|=0, \hspace{0.5cm} s_{\mathcal{G}_1}(uw)=0\]
    because \(\mathcal{L}_1\) contains exactly one cluster with underlying edge \(uw,\) and thus
    \begin{equation}\label{eqq1}
        |C_2(\mathcal{G}_2) \cap \{c,e(c)=uw\}|=(|C(\mathcal{G}) \cap \{c,e(c)=uw\}|-1) \cdot \chi_{|C(\mathcal{G}) \cap \{c,e(c)=uw\}| \geq 3},
    \end{equation}
    as well as
    \begin{equation}\label{eqq2}
        (s_{\mathcal{G}_1}(uw)-2)_{+}+(s_{\mathcal{G}_2}(uw)-2)_{+} \geq (s_{\mathcal{G}}(uw)-2)_{+}-2+2\chi_{s_{\mathcal{G}}(uw) \leq 2}
    \end{equation}
    since \(s_{\mathcal{G}_1}(uw)+s_{\mathcal{G}_2}(uw)=s_{\mathcal{G}_2}(uw) \geq s_{\mathcal{G}}(uw)-2\) (let \(\{c_1,c_2,\hspace{0.05cm} ...\hspace{0.05cm},c_k\}=\{c: c \in C(\mathcal{G}),e(c)=uw\}\) with the indices defined as above (\ref{sdef}); 
    when \(k \leq 2\) or one of \(\mathcal{L}_1,\mathcal{L}_2\) is given by one of these clusters, the result is clear; otherwise \(k \geq 3,\) and since \(\mathcal{L}_1\) contains solely one such cluster, suppose without loss \(x,y\) belong to the segments between \(c_{i_1-1},c_{i_1}\) and \(c_{i_1},c_{i_1+1},\) respectively, equivalent to \(x,y\) being elements of \((t_{2(i_1-1)}+1,,\hspace{0.05cm}...\hspace{0.05cm},t_{1i_1}-1),\) \((t_{2i_1}+1,,\hspace{0.05cm}...\hspace{0.05cm},t_{1(i_1+1)}-1),\) respectively; 
    this implies \(s_{\mathcal{G}_2}(uw) \geq s_{\mathcal{G}}(uw)-2\) as at most two pairs of separated neighbors, \((c_{i_1-1},c_{i_1}),(c_{i_1},c_{i_1+1}),\) are lost in the split); (\ref{eqq1}) and (\ref{eqq2}) yield the power of \(n\) in the contributions of \(\mathcal{L}_1,\mathcal{L}_2\) increases, from the contribution of \(uw,\) by at most \(2\) as
    \[1+\chi_{|C(\mathcal{G})\cap \{c, e(c)=uw\}|=2}+(1-\chi_{s_{\mathcal{G}}(uw) \leq 2}) \leq 2\]
    (the first two terms 
    correspond to the changes in the contributions of \(c_{i_1}\) and \(c_{i_1+1},\) respectively). This additional factor of \(n^2\) is counteracted by the choices for \(u,w\) that are double counted (by the definition of \(v,\) the sets \(\{u,w\}\) given by such edges are pairwise disjoint) and
    \begin{equation}\label{ineq900}
        1+\frac{C(m_1+m_2)^2}{\tilde{p}n}-(1+\frac{Cm_1^2}{\tilde{p}n})(1+\frac{Cm_2^2}{\tilde{p}n})=\frac{Cm_1m_2}{\tilde{p}n}\cdot (2-\frac{Cm_1m_2}{\tilde{p}n}) \geq \frac{Cm_1m_2}{\tilde{p}n}
    \end{equation}
    can be employed for the last factors in the bound. Lastly, when one of \(\mathcal{L}_1,\mathcal{L}_2\) does not have an erased edge of multiplicity \(1,\) the induction hypothesis can be used similarly (select \(v\) in the loop with such an erased edge: this together with \(m \geq 3,\) a consequence of these configurations having at least one erased edge and two clusters of edges, \(O(C_m(np)^{m/2+1}) \leq (np)^m\) for \(pn \geq n_0,\) and Propositions~\ref{prop1},\ref{prop2} yield the claim).
    \par
    \((ii)\) \(\mathcal{L}_1,\mathcal{L}_2\) share an unerased edge \(uw:\) if \(m(uw) \ne 3,\) then paste \(\mathcal{L}_1,\mathcal{L}_2\) using two copies of \(uw\) (via (\ref{paste1}) or (\ref{paste2})), and subsequently delete both. The preimage of any cycle under this mapping has size \(O(m_1m_2)\)
    (for a cycle of length \(m=m_1+m_2-2,\) select two vertices at distance \(m_1,\) a first vertex for the second cycle and its orientation, yielding \(O(mm_1)=O(m_1m_2)\) as \(m_2 \geq \frac{m}{2}\)), the length decreases by \(2,\) and no additional factor is needed (configurations with \(m(uw) \geq 4\) require no adjustment in virtue of Lemma~\ref{lemmbill} below, and those with \(m(uw)=2\) require \(p(1-p) \leq 1\)), resulting in a contribution of size \(O((np)^{-2}m_1m_2)\) relative to (\ref{formulaa}): in the exponent of the power of \(n\) in (\ref{formulaa}), the first sum can only increase (two pairs of two clusters sharing an underlying edge might become one), the second remains unchanged (\(s(\cdot)\) is invariant under the following cycle operators: \((u,S_1,u,S_2,u) \to (u,S_1,u,\overline{S}_2,u)\) when no cluster of \(\cdot\) has edges both in \(S_1\) and \(S_2,\) \((u,u',u,S,u) \to (u,S,u)\) for \((u,u'),(u',u)\) unerased, and \((v_0,v_1,\hspace{0.05cm}...\hspace{0.05cm},v_{k-1},v_0) \to (v_1,v_2,\hspace{0.05cm}...\hspace{0.05cm},v_{k-1},v_0,v_1),\) with the pasting described above being a composition of these three mappings), and the third clearly does not change. Otherwise, \(m(uw)=3\) and erase the third copy of \(uw:\) the preimages have sizes \(O(m^3),\) their contributions, compared to those of interest, have a supplementary factor of 
    \(O((np)^{-2}m^3)\)
    (the number of edges is reduced by \(2,\) and the contribution of \(uw\) becomes \(p\) from \(p(1-p)(1-2p),\) \(p(1-p) \cdot |1-2p| \leq p\)), and again the second factor in the bound can only decrease by this procedure. The conclusion ensues via the induction hypothesis, (\ref{ineq900}), \(m_1,m_2 \leq m \leq c(\tilde{p}n)^{1/4}, \tilde{p}n \geq C^2,\)
    \begin{equation}\label{ineq800}
        O((np)^{-2}m^3)+O((np)^{-2}m^2)=O((np)^{-2}m^3) \leq \frac{C}{\tilde{p}n} \leq \frac{Cm_1m_2}{\tilde{p}n}.
    \end{equation}
    \par
    \underline{Case \(2:\)} \(C_2(\mathcal{G})=\emptyset.\) 
    \par
    In this case, the desired bound is
    \[(np)^{m} \cdot n^{-\sum_{w}{(m_V(w)-1)_{+}}}\cdot C^{|C(\mathcal{G})|} \cdot (1+\frac{Cm^2}{\tilde{p}n}).\]
    If no two edges among \((L(q)R(q))_{1 \leq q \leq t}\) share an endpoint, the conclusion holds from the observation in the beginning of the proof (the above is at least \(C(np)^m\)). Otherwise let \(v\) be a vertex incident with at least two edges among \((L(q)R(q))_{1 \leq q \leq t},\) and denote by \(\mathcal{L}_1,\mathcal{L}_2\) the cycles induced by two of its copies. 
    \par
    \((iii)\) \(\mathcal{L}_1\) and \(\mathcal{L}_2\) share no unerased edge. If each of them has at least one erased edge, then the induction hypothesis can be employed and note the contribution of any vertex \(w\) does not change: when \(w\) appears solely in one of the two, this is clear, and otherwise the claim follows from 
    \[(m_{\mathcal{V},\mathcal{L}_1}(w)-1)_{+}+(m_{\mathcal{V},\mathcal{L}_2}(w)-1)_{+}+1=(m_\mathcal{V}(w)-1)_{+}\] (as \(w\) appears in both, it must be chosen solely once, yielding a saving of \(1\) in the exponent of \(n\)). Else, one of the loops has no erased edge, and again the result follows (the extra factor of \(n\) is discounted as \(v\) must be selected solely once). 
    \par
    \((iv)\) \(\mathcal{L}_1\) and \(\mathcal{L}_2\) share an unerased edge \(uw.\) Proceed as in \((ii):\) if \(m(uw) \ne 3,\) then delete two of its copies and this gives rise to a contribution of relative size \(O((np)^{-2} \cdot m_1m_2);\) else delete two of its copies and erase one, which generates \(O((np)^{-2}m^3).\)
    \par
    The conclusion follows by the induction hypothesis via \((iii)\) and \((iv)\)  together with (\ref{ineq900}) and (\ref{ineq800}).
\end{proof}

\begin{lemma}\label{lemmbill}
    For \(k \in \mathbb{N}, k \geq 2, p \in [0,1],\)
    \[|(1-p)^{k+1}-(-p)^{k+1}| \leq |(1-p)^{k-1}-(-p)^{k-1}|.\]
\end{lemma}

\begin{proof}
    The cases \(2|k,p \in \{0,1\}\) follow easily as \(k \to a^k\) is decreasing for \(a<1.\) Suppose next \(2|k+1,\) and without loss of generality assume \(p \in (0,\frac{1}{2}].\) Then
    \[(1-p)^{k-1}-p^{k-1}-((1-p)^{k+1}-p^{k+1})=(1-p)^{k-1} \cdot p(2-p)-p^{k-1} \cdot (1-p)(1+p)=\]
    \[=p(1-p) \cdot [(1-p)^{k-2}\cdot (2-p)-p^{k-2} \cdot (1+p)] \geq 0\]
    employing \(1-p \geq p, 2-p \geq 1+p.\)
\end{proof}

Return to the higher moments: fix a tuple \((m(e_j))_{1 \leq j \leq r},\) and use 
Lemma~\ref{lem999} on \(\mathbf{i}\) defined in (\ref{gluedsum}) (\(n^w=(\tilde{p}n)^{1/2}\) for Proposition~\ref{proptra}). In this case, the exponent of the second factor in (\ref{formulaa}) is
\[-L-\sum_{c \in C(\mathcal{G})}{(|c|-\chi_{c \not \in C_2(\mathcal{G})})}-\frac{1}{2}\sum_{e}{(s(e)-2)_{+}}=\]
\[=-2L-(-L+\sum_{c \in C(\mathcal{G})}{(|c|-\chi_{c \not \in C_2(\mathcal{G})})}+\frac{1}{2}\sum_{e}{(s(e)-2)_{+}})=-2L-(-L+\sum_{1 \leq j \leq r}{m(e_j)}+\frac{1}{2}\sum_{1 \leq j \leq r}{(s_j-2)}),\] 
where \(s_j\) denotes the number of cycles to which \(e_j\) belongs. The overall contribution becomes, for a fixed tuple \((m(e_j))_{1 \leq j \leq r},\) 
\[n^{2mL-L}p^{2mL-L/2}(1-p)^{L/2}(2\sqrt{2}m)^{L} \cdot\]
\begin{equation}\label{bdd}
    \cdot O(n^{L-\sum_{1 \leq j \leq t}{m(e_j)}-\frac{1}{2}\sum_{1 \leq j \leq t}{(s_j-2)}} \cdot (4CmL^2p^{-1})^{-L+\sum_{1 \leq j \leq r}{m(e_j)}} \cdot (p(1-p))^{-L/2+r})
\end{equation}
(the length of \(\mathbf{i}\) is \(2mL+L,\) the preimages generate a factor of \((2m)^L \cdot (4mL^2)^{\sum_{1 \leq j \leq r}{m(e_j)}-L},\) and accounting for the erased edges in \(\mathbf{i}\) yields at most \(p^{-L-\sum_{1 \leq j \leq r}{m(e_j)}} \cdot (p(1-p))^{r}\)). Since
\begin{equation}\label{bdfoo0}
    -L/2+r \geq \sum_{1 \leq j \leq r}{(-s_{j}/2+1)}=-\frac{1}{2}\sum_{1 \leq j \leq t}{(s_j-2)}
\end{equation}
and \(np(1-p) \geq 1, \min_{1 \leq j \leq t}{s_j} \geq 2,\) the bound (\ref{bdd}) becomes 
\begin{equation}\label{bdfoo}
    O((4CmL^2n^{-1}p^{-1})^{-L+\sum_{1 \leq j \leq r}{m(e_j)}} \cdot (np(1-p))^{-\frac{1}{2}\sum_{1 \leq j \leq r}{(s_j-2)}})=O(\delta^{-L+\sum_{1 \leq j \leq r}{m(e_j)}}).
\end{equation}
for 
\[\delta=4CmL^2(np)^{-1}=O((np)^{-3/4}).\]
Hence for \(n\) sufficiently large, the contribution of configurations with \(\sigma:=\sum_{1 \leq j \leq r}{m(e_j)}>L\) is at most
\[\sum_{\sigma \geq L+1}{2^{\sigma}\delta^{-L+\sigma}}=\sum_{L<\sigma< 2L}{2^{\sigma}\delta^{-L+\sigma}}+\sum_{\sigma \geq 2L}{2^{\sigma}\delta^{-L+\sigma}} \leq \sum_{L<\sigma<2L}{2^{\sigma}\delta}+\sum_{\sigma \geq 2L}{2^{\sigma}\delta^{\sigma/2}} \leq \delta \cdot 2^{2L}+ 2 \cdot (2\delta^{1/2})^{2L}=O(\delta)\]
via
\[|\{(t,x_1,\hspace{0.05cm}...\hspace{0.05cm},x_t): x_1+...+x_t=\sigma,x_1,\hspace{0.05cm}...\hspace{0.05cm},x_t \in \mathbb{N}-\{1\}\}|=\sum_{1 \leq t \leq \sigma/2}{\binom{\sigma-2t+t-1}{t-1}} \leq \sum_{0 \leq t \leq \sigma-2}{\binom{\sigma-2}{t}}=2^{\sigma-2},\]
as for \(n,m \in \mathbb{N},\) \(|\{(x_1,\hspace{0.05cm}...\hspace{0.05cm},x_m):x_1+...+x_m=n, x_1,\hspace{0.05cm}...\hspace{0.05cm},x_t \in \mathbb{Z}_{\geq 0}\}\}|=\binom{n+m-1}{m-1}.\) 
When \(\sigma=L,\) the contribution is \(o(1)\) unless 
\[m(e_1)=m(e_2)=...=m(e_r)=2,r=L/2\] ((\ref{bdfoo}), \(np(1-p) \to \infty,\) and \(m=O((np)^{1/4})\) yield \(s_1=s_2=...=s_r=2,\sum_{1 \leq j \leq r}{m(e_j)}=L,\) with the first chain of equalities entailing \(m(e_1)=m(e_2)=...=m(e_r)=2\)), and all edges in \(\mathbf{i}_1,\mathbf{i}_2,\hspace{0.05cm}...\hspace{0.05cm},\mathbf{i}_L\) are erased apart from the copies of \(e_1,e_2,\hspace{0.05cm}...\hspace{0.05cm},e_{r}\) (the proof of Lemma~\ref{lem999} yields in the last factor in (\ref{formulaa}) \(1\) can be dropped when at least one edge is unerased),
which cannot occur when \(L>2\) (the underlying graph has \(L/2>1\) connected components). When \(L=2,\) the previous observation yields the result for \(l=2,\) and the general case \(l \geq 2\) follows as well in light of (\ref{llargerthan2}).
\par
This concludes the analysis of higher moments of the centered trace and the proof of Theorem~\ref{th1}.

\section{Eigenvalue CLT}\label{sect2}
This section contains the proof of Theorem~\ref{th2}, in which (\ref{conv1}) plays a fundamental role. Take \(\overline{A}=\frac{1}{\sqrt{p}}A:\) (\ref{conv1}) can be restated as
\begin{equation}\label{normalizedconv}
    \frac{tr(\overline{A}^{2m})-\mathbb{E}[tr(\overline{A}^{2m})]}{2m \cdot (n\sqrt{p})^{2m-1} \sqrt{1-p}} \Rightarrow N(0,2),
\end{equation}
and the claim of Theorem~\ref{th2} becomes
\[\frac{1}{\sqrt{1-p}}(\lambda_1(\overline{A})-\mathbb{E}[\lambda_1(\overline{A})]) \Rightarrow N(0,2).\] 
Begin with a proof sketch of this convergence, meant to highlight the gist of the argument: in what follows, assume \(\frac{\log{n}}{c_1\log{(\tilde{p}n)}}\leq m \leq c_1(\tilde{p}n)^{1/4}\) for \(c_1>0\) universal: note that for \(n^w=(\tilde{p}n)^{1/2},\)
\[\frac{1}{w} \geq \frac{2\log{n}}{\log{(np)}}, \hspace{0.4cm} (\tilde{p}n^{1-w})^{1/2}=n^{w/2}=(\tilde{p}n)^{1/4},\]
entailing previous (and forthcoming) lemmas can be applied.
\par
Slutsky's lemma and (\ref{sqexp}) below entail \(n\sqrt{p}\) can be replaced by \(\mathbb{E}[\lambda_1(\overline{A})]\) in (\ref{normalizedconv}):
\[\frac{tr(\overline{A}^{2m})-\mathbb{E}[tr(\overline{A}^{2m})]}{2m \sqrt{1-p} \cdot (\mathbb{E}[\lambda_1(\overline{A})])^{2m-1}} \Rightarrow N(0,2).\]
Let \(\xi=\lambda_1(\overline{A})-\mathbb{E}[\lambda_1(\overline{A})]:\)
\[\frac{tr(\overline{A}^{2m})-\mathbb{E}[tr(\overline{A}^{2m})]}{2m \sqrt{1-p}\cdot (\mathbb{E}[\lambda_1(\overline{A})])^{2m-1}}=\frac{(\lambda^{2m}_1(\overline{A})-\mathbb{E}[\lambda^{2m}_1(\overline{A})])+\sum_{k \geq 2}{\lambda^{2m}_k(\overline{A})}-\mathbb{E}[\sum_{k \geq 2}{\lambda^{2m}_k(\overline{A})}]}{2m \sqrt{1-p} \cdot (\mathbb{E}[\lambda_1(\overline{A})])^{2m-1}}=\]
\begin{equation}\label{bigeq}
    =\frac{(\xi+\mathbb{E}[\lambda_1(\overline{A})])^{2m}-(\mathbb{E}[\lambda_1(\overline{A})])^{2m}}{2m\sqrt{1-p} \cdot (\mathbb{E}[\lambda_1(\overline{A})])^{2m-1}}+\frac{(\mathbb{E}[\lambda_1(\overline{A})])^{2m}-\mathbb{E}[\lambda^{2m}_1(\overline{A})]}{2m\sqrt{1-p} \cdot (\mathbb{E}[\lambda_1(\overline{A})])^{2m-1}}+\frac{\sum_{k \geq 2}{\lambda^{2m}_k(\overline{A})}-\mathbb{E}[\sum_{k \geq 2}{\lambda^{2m}_k(\overline{A})}]}{2m \sqrt{1-p}\cdot (\mathbb{E}[\lambda_1(\overline{A})])^{2m-1}}.
\end{equation}
\par
Since \(m \geq \frac{1}{2c_1w} \geq \frac{\log{n}}{c_1\log{(np)}},\)
\begin{equation}\label{dom1}
    \frac{\sum_{k \geq 2}{\lambda^{2m}_k(\overline{A})}-\mathbb{E}[\sum_{k \geq 2}{\lambda^{2m}_k(\overline{A})}]}{2m \cdot (\mathbb{E}[\lambda_1(\overline{A})])^{2m-1}}=o_p(1):
\end{equation}
\(\overline{A}=\sqrt{p}vv^T+\frac{1}{\sqrt{p}}\Tilde{A},v=[1 1 \hspace{0.05cm} ... \hspace{0.05cm} 1]^T \in \mathbb{R}^n,\) and Weyl's inequalities yield for \(2 \leq j \leq n,\)
\[\lambda_{j}(\frac{1}{\sqrt{p}}\Tilde{A})+\lambda_n(\sqrt{p}vv^T)=\lambda_{j}(\frac{1}{\sqrt{p}}\Tilde{A}) \leq \lambda_{j}(\overline{A}) \leq \lambda_{j-1}(\frac{1}{\sqrt{p}}\Tilde{A})=\lambda_{j-1}(\frac{1}{\sqrt{p}}\Tilde{A})+\lambda_2(\sqrt{p}vv^T),\]
whereby
\[\sum_{j \geq 2}{\lambda_j^{2m}(\overline{A})} \leq 2p^{-m} \cdot tr(\Tilde{A}^{2m}),\]
which together with Markov inequality and (\ref{evenmom}) concludes (\ref{dom1}): for \(t>0,\)
\[\mathbb{P}(p^{-m} \cdot tr(\Tilde{A}^{2m}) \geq t(n\sqrt{p})^{2m-1}) \leq \frac{2C_m(p(1-p))^{m}n^{m+1}}{t(n\sqrt{p})^{2m-1}p^m} \leq \frac{2 \cdot 4^mn^2}{t(np)^m}=o(1).\]
\par
Return to (\ref{bigeq}): (\ref{dom1}) yields 
\begin{equation}\label{refined31}
    \frac{tr(\overline{A}^{2m})-\mathbb{E}[tr(\overline{A}^{2m})]}{2m\sqrt{1-p} \cdot (\mathbb{E}[\lambda_1(\overline{A})])^{2m-1}}=\frac{(\xi+\mathbb{E}[\lambda_1(\overline{A})])^{2m}-(\mathbb{E}[\lambda_1(\overline{A})])^{2m}}{2m\sqrt{1-p} \cdot (\mathbb{E}[\lambda_1(\overline{A})])^{2m-1}} +\frac{\mathbb{E}[\lambda^{2m}_1(\overline{A})]-(\mathbb{E}[\lambda_1(\overline{A})])^{2m}}{2m\sqrt{1-p} \cdot (\mathbb{E}[\lambda_1(\overline{A})])^{2m-1}}+o_p(1).
\end{equation}
Lemma~\ref{concentration} yields \(\xi=o_p((1-p)^{1/2}\mathbb{E}[\lambda_1(\overline{A})](np)^{-1/4}),\) whereby when this high-probability event holds,
\[|\frac{(\xi+\mathbb{E}[\lambda_1(\overline{A})])^{2m}-(\mathbb{E}[\lambda_1(\overline{A})])^{2m}}{2m\sqrt{1-p} \cdot (\mathbb{E}[\lambda_1(\overline{A})])^{2m-1}}-\frac{\xi}{\sqrt{1-p}}| \leq (1-p)^{-1/2}(2m)^{-1}\sum_{2 \leq k \leq 2m}{\binom{2m}{k}(\mathbb{E}[\lambda_1(\overline{A})])^{-k+1}|\xi|^k} \leq\]
\[\leq (1-p)^{-1/2}(2m)^{-1} \cdot \mathbb{E}[\lambda_1(\overline{A})] \cdot 8m^2\mathbb{E}[\lambda_1(\overline{A})]^{-2}\xi^2=4m(1-p)^{-1/2}\mathbb{E}[\lambda_1(\overline{A})]^{-1}\xi^2=o_p(\xi),\]
using \(\binom{2m}{k} \leq (2m)^k,m \leq c_1(np)^{-1/4},\) from which Slutsky's lemma gives the desired result conditional on 
\begin{equation}\label{ratiotrue}
    \frac{\mathbb{E}[\lambda^{2m}_1(\overline{A})]-(\mathbb{E}[\lambda_1(\overline{A})])^{2m}}{2m \sqrt{1-p}\cdot (\mathbb{E}[\lambda_1(\overline{A})])^{2m-1}} \to 0.
\end{equation}
\par
Subsection~\ref{1.7} derives two concentration results for \(\lambda_1(\overline{A});\) however, these are not strong enough to yield directly (\ref{ratiotrue}), and instead, these ratios, modulo \(o(1)\) terms, are shown to grow linearly in \(m.\)
Subsection~\ref{1.8} uses this simplification and (\ref{refined31}) for two comparable values of \(m\) (i.e., \(m_1,m_2\) with \(c_2 \leq \frac{m_1}{m_2} \leq \frac{1}{c_2}\)) to conclude (\ref{ratiotrue}) and thus derive the eigenvalue CLT in Theorem~\ref{th2}.

\subsection{Concentration of \(\lambda_1(A)\)}\label{1.7}

The discussion above makes apparent the importance of the ratio in (\ref{ratiotrue}): the two lemmas in this subsection provide a proxy for it. Begin with a concentration result for \(\lambda_1(A).\)

\begin{lemma}\label{concentration}
There exists \(c>0\) such that if \(t \geq 4(\epsilon_{1}+\epsilon_{2}),\) \(\mathbb{E}[\lambda_1(A)] \in [np(1-\epsilon_{1}),np(1+\epsilon_{2})], \epsilon_1, \epsilon_2 \in (0,\frac{1}{4}),\) \(w>0, \tilde{p} \geq n^{w-1}, m \in \mathbb{N}, m^2 \leq c\min{(n^w,\tilde{p}n^{1-w})},\) then
\begin{equation}\label{tails}
    \mathbb{P}(|\frac{\lambda_1(A)}{\mathbb{E}[\lambda_1(A)]}-1| \geq t) \leq
    \frac{4 \cdot 16^m(1-p)^mn}{(np)^{m}t^{2m}}.
\end{equation}
\end{lemma}

\begin{proof}
It suffices to show each tail has probability at most \(\frac{2 \cdot 16^m(1-p)^mn}{(np)^{m}t^{2m}}.\) Begin with the right:
\[\mathbb{P}(\frac{\lambda_1(A)}{\mathbb{E}[\lambda_1(A)]}-1 \geq t) \leq \mathbb{P}(||A|| \geq (1+t)\mathbb{E}[\lambda_1(A)]) \leq  \mathbb{P}(||A|| \geq np(1+t)(1-\epsilon_1)).\]
Since 
\[||A||=||pvv^T+\Tilde{A}|| \leq ||pvv^T||+||\Tilde{A}||=np+||\Tilde{A}||,\]
an upper bound for
\[\mathbb{P}(||\Tilde{A}|| \geq np(t(1-\epsilon_1)-\epsilon_1)) \leq \mathbb{P}(||\Tilde{A}|| \geq npt/2)\]
is enough (\(4\epsilon_1(1/2-\epsilon_1)-\epsilon_1=\epsilon_1-4\epsilon_1^2 \geq 0\)). Use (\ref{evenmom}): for some \(c_1>0\) and \(m^2 \leq c_1\min{(n^w,\tilde{p}n^{1-w})},\)
\[\mathbb{P}(||\Tilde{A}|| \geq npt/2) \leq \frac{2C_m(p(1-p))^mn^{m+1}}{(npt/2)^{2m}} \leq \frac{2 \cdot 16^m(1-p)^mn}{(np)^{m}t^{2m}}.\]
Similarly, for the left tail,
\[\lambda_1(A)=\lambda_1(pvv^T+\Tilde{A}) \geq \lambda_1(pvv^T)-||\Tilde{A}||=np-||\Tilde{A}||,\]
whereby
\[\mathbb{P}(\frac{\lambda_1(A)}{\mathbb{E}[\lambda_1(A)]}-1 \leq -t) \leq \mathbb{P}(\lambda_1(A) \leq np(1+\epsilon_2)(1-t)) \leq \mathbb{P}(||\Tilde{A}|| \geq np(t(1+\epsilon_2)-\epsilon_2)) \leq \mathbb{P}(||\Tilde{A}|| \geq npt/2).\]
\end{proof}

Under the binomial expansion given by \(\lambda_1(\overline{A})=\mathbb{E}[\lambda_1(\overline{A})]+(\lambda_1(\overline{A})-\mathbb{E}[\lambda_1(\overline{A})]),\)
Lemma~\ref{moments} yields 
\[\frac{\mathbb{E}[\lambda^{2m}_1(\overline{A})]-(\mathbb{E}[\lambda_1(\overline{A})])^{2m}}{2m \cdot (\mathbb{E}[\lambda_1(\overline{A})])^{2m-1}}=\mathbb{E}[\lambda_1(\overline{A})] \cdot \frac{\mathbb{E}[(\frac{\lambda_1(A)}{\mathbb{E}[\lambda_1(A)]})^{2m}]-1}{2m}=\]
\[=\mathbb{E}[\lambda_1(\overline{A})] \cdot \frac{\mathbb{E}[(\frac{\lambda_1(A)}{\mathbb{E}[\lambda_1(A)]})^{2m}\chi_{|\frac{\lambda_1(A)}{\mathbb{E}[\lambda_1(A)]}-1| \leq (np/K)^{-1/2}}]-1}{2m}+o(1),\]
for \(K \geq c_0.\)

\begin{lemma}\label{moments}
There exist \(c_1,c_2,c_3>0\) such that if \(K>0, w \in (\frac{1+\log{K}}{c_1\log{n}},1),\tilde{p} \geq n^{w-1},\) \(m_0=\lfloor c_2 \min{(n^{w/2},(\tilde{p}n^{1-w})^{1/2})} \rfloor,\) \(m_0 \geq \frac{1}{c_3w},m \in \mathbb{N},m \leq \frac{m_0}{2},\) then
\[\mathbb{E}[(\frac{\lambda_1(A)}{\mathbb{E}[\lambda_1(A)]}-1)^{2m}\chi_{|\frac{\lambda_1(A)}{\mathbb{E}[\lambda_1(A)]}-1| \geq (np/K)^{-1/2}}] \leq \frac{4n}{(K/16)^{m_0}}\cdot (\frac{np}{K})^{-m}.\]
\end{lemma}

\begin{proof}
When \(\delta^{\frac{1}{2m}}>4(\epsilon_1+\epsilon_2)\) with \(\epsilon_1,\epsilon_2\) as in Lemma~\ref{concentration}, inequality (\ref{tails}) yields for \(m \leq \frac{m_0}{2},\)
\[\mathbb{E}[(\frac{\lambda_1(A)}{\mathbb{E}[\lambda_1(A)]}-1)^{2m}\chi_{|\frac{\lambda_1(A)}{\mathbb{E}[\lambda_1(A)]}-1| \geq \delta^{\frac{1}{2m}}}] \leq \frac{4 \cdot 16^{m_0}n}{(np)^{m_0}}\int_{\delta^{\frac{1}{2m}}}^{\infty}{2mt^{2m-1} \cdot t^{-2m_0}dt}=\frac{4 \cdot 16^{m_0}n}{(np)^{m_0}} \cdot \frac{m\delta^{1-m_0/m}}{m_0-m},\]
and in particular, when \(\delta=(np/K)^{-m},\)
\[\mathbb{E}[(\frac{\lambda_1(A)}{\mathbb{E}[\lambda_1(A)]}-1)^{2m}\chi_{|\frac{\lambda_1(A)}{\mathbb{E}[\lambda_1(A)]}-1| \geq (np/K)^{-1/2}}] \leq \frac{4n}{(K/16)^{m_0}} \cdot (\frac{np}{K})^{-m}.\]
It is shown next that some \(\epsilon_1=O(\frac{1}{np}),\epsilon_2=O(\frac{1}{np})\) are valid choices for Lemma~\ref{concentration}, whereby the result ensues as \(\delta^{\frac{1}{2m}}=(np/K)^{-1/2}, 4(\epsilon_1+\epsilon_2) \leq \frac{C'}{np},\) ensuring \(4(\epsilon_1+\epsilon_2)<\delta^{\frac{1}{2m}}\) insomuch as \((np)^{1/2} \geq n^{w/2} \geq C'K^{1/2}.\) It remains to justify
\begin{equation}\label{sqexp}
    \mathbb{E}[\lambda_1(A)] \in [np(1-\epsilon_{1}),np(1+\epsilon_{2})].
\end{equation}
For the upper bound, use (\ref{Amom2}), in which \((np)^{2m_0-1}>2 \cdot C_{m_0}n^{m_0+1}(p(1-p))^{m_0}\) because \(m_0 \geq \frac{1}{c_3w}\) renders \((np/4)^{m_0-1} \geq (n^w/4)^{m_0-1}>8 \cdot n,\) and so
\[\mathbb{E}[\lambda_1(A)] \leq (\mathbb{E}[tr(A^{2m_0})])^{\frac{1}{2m_0}} \leq [(np)^{2m_0}+4m_0(1+C\cdot c) \cdot (np)^{2m_0-1}]^{\frac{1}{2m_0}} \leq np+2(1+C\cdot c),\]
whereas for the lower,
\begin{equation}\label{minn}
    \lambda_1(A) \geq \frac{1}{n}v^TAv=np+\frac{1}{n}\sum_{1 \leq i,j \leq n}{\Tilde{a}_{ij}},
\end{equation}
and
\begin{equation}\label{markov}
    \mathbb{P}(\frac{1}{n}\sum_{1 \leq i,j \leq n}{\Tilde{a}_{ij}} \leq -t) \leq t^{-2} \cdot Var(\frac{1}{n}\sum_{1 \leq i,j \leq n}{\Tilde{a}_{ij}})=O(t^{-2}p(1-p)),
\end{equation}
from which for \(x_{-}:=\max{(-x,0)},\)
\begin{equation}\label{bd1p}
    \mathbb{E}[(\lambda_1(A)-np)_{-}] \leq p^{1/2}(1-p)^{1/2}+\int_{p^{1/2}(1-p)^{1/2}}^{\infty}{\frac{C''p(1-p)}{t^2}dt}=O(p^{1/2}(1-p)^{1/2}).
\end{equation}
\end{proof}

\subsection{Asymptotic Law of \(\lambda_1(A)\)}\label{1.8}

Return to (\ref{bigeq}): given (\ref{refined31}), consider the first ratio on its right-hand side.

\begin{lemma}\label{l4}
There exist \(c_1,c_2>0\) such that if \(w>0, \tilde{p} \geq n^{w-1},\) \(m^2 \leq m_0^2=c_1\min{(n^{w},\tilde{p}n^{1-w})}, \newline m_0 \geq c_2\min{(\log{n},w^{-1})},\) 
then as \(np \to \infty,\)
\[\frac{(\xi+\mathbb{E}[\lambda_1(\overline{A})])^{2m}-(\mathbb{E}[\lambda_1(\overline{A})])^{2m}}{2m(\mathbb{E}[\lambda_1(\overline{A})])^{2m-1}\sqrt{1-p}}=\frac{\xi}{\sqrt{1-p}}(1+o_p(1)).\]
\end{lemma}

\begin{proof}
Take \(u=\alpha \cdot \frac{1}{\sqrt{n}}v+\sqrt{1-\alpha^2} \cdot v^{\perp}\) with \(\alpha \in [0,1],||u||=1,v \cdot v^{\perp}=0, Au=\lambda_1(A)u:\) then
\[\lambda_1(A)=u^TAu=\alpha^2 (np+\frac{1}{n}v^T\tilde{A}v)+\frac{2\alpha\sqrt{1-\alpha^2}}{\sqrt{n}} \cdot v^T\Tilde{A}v^{\perp}+(1-\alpha^2)\cdot (v^{\perp})^T\Tilde{A}v^{\perp}.\]
Let \(t_n=(np)^{1/4}:\) since \(t_n\to \infty,\) with high probability \(\frac{1}{n}v^T\tilde{A}v \leq t_n\sqrt{p(1-p)}\) since 
\begin{equation}\label{varr}
    Var(\frac{1}{n}v^T\tilde{A}v)=Var(\frac{1}{n}\sum_{1 \leq i,j \leq n}{\Tilde{a}_{ij}})=O(p(1-p)),
\end{equation}
whereby
\[\lambda_1(A) \leq \alpha^2 (np+t_n\sqrt{p(1-p)})+(1-\alpha^2+2\alpha\sqrt{1-\alpha^2})||\Tilde{A}|| \leq \]
\[\leq ||\Tilde{A}||+\frac{np+t_n\sqrt{p(1-p)}-||\Tilde{A}||}{2}+\sqrt{(\frac{np+t_n\sqrt{p(1-p)}-||\Tilde{A}||}{2})^2+||\Tilde{A}||^2}\]
\begin{equation}\label{eqqt}
    \leq np+t_n\sqrt{p(1-p)}+\frac{||\tilde{A}||^2}{2 \cdot np/4}
\end{equation}
using
\[2a(\cos{x})^2+2b\cos{x}\sin{x}=a+a\cos{2x}+b\sin{2x} \leq a+\sqrt{a^2+b^2},\]
\(c^2+d \leq (c+\frac{d}{2c})^2\) for \(c>0,\) as well as (\ref{evenmom}) for \(K_n=2(1-p)^{1/2}\exp(c_2^{-1})+(p(1-p))^{1/2}t_n \leq (np)^{1/2}/2\) which give 
\[\mathbb{P}(||\Tilde{A}|| \geq K_n(np)^{1/2}) \leq \frac{2C_{m_0}(p(1-p))^{m_0}n^{m_0+1}}{(K_n(np)^{1/2})^{2m_0}} \leq \frac{2n \cdot 4^{m_0}}{(K_n(1-p)^{-1/2})^{2m_0}} \leq \frac{2}{n}=o(1).\]
Therefore, with high probability
\[\lambda_1(A) \leq np+t_n\sqrt{p(1-p)}+\frac{2||\tilde{A}||^2}{np} \leq np+K_n+2K_n^2.\] 
A lower limit holding with high probability ensues from (\ref{minn}), \(\frac{1}{n}v^T\tilde{A}v \geq -t_n\sqrt{p(1-p)}\) (a consequence of (\ref{varr}) and \(t_n\to \infty\)), 
\[\lambda_1(A) \geq np-K_n,\] 
providing together with (\ref{Amom2}), which entails \(\mathbb{E}[\lambda^{2m}_1(A)] \leq (np+2(1-p))^{2m}\) for \(m \geq \frac{1+w}{c_1w}\) (using \(n\tilde{p} \geq n^w,\) and \((np+2(1-p))^{2m} \geq (np)^{2m}+3m(1-p)(np)^{2m-1}+(1-p)(np)^{2m-1}\)), that
\[\sqrt{p}|\xi| \leq |\lambda_1(A)-np|+|np-\mathbb{E}[\lambda_1(A)]| \leq  |\lambda_1(A)-np|+\mathbb{E}[(\lambda_1(A)-np)_{-}]+(\mathbb{E}[\lambda_1(A)]-np)_{+} \leq\]
\begin{equation}\label{lastt}
    \leq 2K_n^2+2K_n \leq 2K_n^2(1-p)^{-1/2}+2(1-p)^{1/2} \cdot K_n(1-p)^{-1/2} \leq (1-p)^{1/2} CR_n^2
\end{equation}
for \(R_n=K_n(1-p)^{-1/2} \leq 2t_n=2(np)^{1/4}\) (the second inequality uses \(np-\mathbb{E}[\lambda_1(A)] \leq \mathbb{E}[(\lambda_1(A)-np)_{-}],\) the third \(\mathbb{E}[\lambda_1(A)] \leq (\mathbb{E}[\lambda^{2m}_1(A)])^{1/(2m)},\) and the last \(R_n \geq 1).\)
\par
(\ref{lastt}) and \(\mathbb{E}[\lambda_1(A)] \geq np\) (via (\ref{minn}))
\[(\mathbb{E}[\lambda_1(\overline{A})])^{-k+1} \cdot |\xi|^{k-1} \leq (n\sqrt{p})^{-k+1} \cdot (\frac{(1-p)^{1/2}CR^2_n}{p^{1/2}})^{k-1}=(\frac{np}{CR_n^2\sqrt{1-p}})^{-(k-1)}.\]
Finally, the claim follows from
\[\frac{(\xi+\mathbb{E}[\lambda_1(\overline{A})])^{2m}-(\mathbb{E}[\lambda_1(\overline{A})])^{2m}}{2m(\mathbb{E}[\lambda_1(\overline{A})])^{2m-1}\sqrt{1-p}}-\frac{\xi}{\sqrt{1-p}}=\frac{\xi}{\sqrt{1-p}} \cdot \frac{1}{2m}\sum_{2 \leq k \leq 2m}{\binom{2m}{k}(\mathbb{E}[\lambda_1(\overline{A})])^{-k+1}\xi^{k-1}},\]
and \(\binom{2m}{k} \leq (2m)^k,\) \(m^2 \leq c_1(\tilde{p}n)^{1/2} \leq c_1(np)^{1/2},\) 
\[\sum_{2 \leq k \leq 2m}{(2m\delta_n)^{k-1}}=O(m\delta_n)=o(1)\]
for 
\[\delta_n=\frac{CR_n^2\sqrt{1-p}}{2np} \leq \frac{4C(np)^{1/2}}{2np}=\frac{2C}{(np)^{1/2}}.\]
\end{proof}

Theorem~\ref{th2} can now be justified.

\begin{proof}
In what follows, let \(c_3\log{n} \leq m \leq c_4\log{n}, n^w=(\tilde{p}n)^{1/2}\) with \(c_3=\frac{c_4}{8}\) and \(c_4>0\) sufficiently small compared to all the universal constants appearing in the preceding propositions and lemmas. 
Given the growth assumptions on \(\tilde{p},\) this choice of \(w\) guarantees \(\frac{1}{c^2w^2} \leq m^2 \leq c\min{(n^w,\tilde{p}n^{1-w})}\) for \(n\) sufficiently large due to \([c_3\log{n},c_4\log{n}] \subset [\frac{1}{cw},c^{1/2}(np)^{1/4}]\) for all \(n \geq n(c)\) from \(w=\frac{\log{(\tilde{p}n)}}{2\log{n}} \geq \frac{\log{\log{n}}}{\log{n}},(np)^{1/4} \geq B^{1/4}\log{n}.\) Hence Lemma~\ref{l4} can be applied and (\ref{dom1}) holds, transforming (\ref{refined31}) into
\begin{equation}\label{finaleq1}
    \frac{tr(\overline{A}^{2m})-\mathbb{E}[tr(\overline{A}^{2m})]}{2m\sqrt{1-p} \cdot (\mathbb{E}[\lambda_1(\overline{A})])^{2m-1}}=\frac{\xi}{\sqrt{1-p}}(1+o_p(1))+ \frac{\mathbb{E}[\lambda_1(\overline{A})^{2m}]-(\mathbb{E}[\lambda_1(\overline{A})])^{2m}}{2m\sqrt{1-p} \cdot (\mathbb{E}[\lambda_1(\overline{A})])^{2m-1}}+o_p(1).
\end{equation}
Take \(q=\sqrt{p(1-p)},\xi_0=\frac{\lambda_1(A)-\mathbb{E}[\lambda_1(A)]}{q},\alpha=\mathbb{E}[\lambda_1(A)]:\) in virtue of Slutky's lemma, (\ref{finaleq1}) becomes equivalent to
\[\xi_0+(1+o_p(1)) \cdot \frac{\mathbb{E}[(\xi_0 q+\alpha)^{2m}-\alpha^{2m}]}{2mq\alpha^{2m-1}} \Rightarrow N(0,2).\]
Put differently, \(\xi_0=\Tilde{N}-(1+o_p(1))c_m\) with \(\Tilde{N} \Rightarrow N(0,2),\) and
\[c_m=\frac{\mathbb{E}[(\xi_0 q+\alpha)^{2m}-\alpha^{2m}]}{2mq\alpha^{2m-1}}.\]
Theorem~\ref{th2} ensues by using Slutsky's lemma and justifying 
\begin{equation}\label{fine}
    c_m=o(1)
\end{equation}
for some \(m\) with \(c_3 \log{n} \leq m \leq c_4\log{n},\) a range for which Theorem~\ref{th1} is valid by choosing \(c_4 \leq B^{1/4}c_0.\) 
\par
Notice
\begin{equation}\label{c1}
    c_m=c(m,n,p)=\frac{\mathbb{E}[(\xi_0 q+\alpha)^{2m}-\alpha^{2m}]}{2mq\alpha^{2m-1}}=\alpha \cdot \frac{\mathbb{E}[(\frac{\xi_0 q}{\alpha}+1)^{2m}-1-2m \cdot \frac{\xi_0 q}{\alpha}]}{2mq}
\end{equation}
since \(\mathbb{E}[\xi_0]=(1-p)^{-1/2}\mathbb{E}[\xi]=0.\) 
Lemma~\ref{moments} for \(n^{w}=B^{1/2}(\log{n})^2,\) \(m_0 \geq \lfloor c_2B^{1/4}\log{n} \rfloor \geq \frac{c_2B^{1/4}\log{n}}{2}\) 
yields for \(n \geq n(c_1,c_2,B),K=16\exp{(8B^{-1/4}c_2^{-1})},r \leq \frac{m_0}{2}\) that
\[\mathbb{E}[(\frac{\lambda_1(A)}{\mathbb{E}[\lambda_1(A)]}-1)^{2r}\chi_{|\frac{\lambda_1(A)}{\mathbb{E}[\lambda_1(A)]}-1| \geq (np/K)^{-1/2}}] \leq \frac{4}{n^{3}}\cdot (\frac{np}{K})^{-r}.\]
This inequality for \(r \leq \frac{m_0}{2}, m_0 \geq 4c_4\log{n} \geq 4m,\)
\(\frac{\xi_0 q}{\alpha}=\frac{\lambda_1(A)}{\mathbb{E}[\lambda_1(A)]}-1,\) and the binomial theorem give for 
\(t_0=(np/K)^{-1/2},\)
\begin{equation}\label{c2}
    \frac{|\mathbb{E}[(\frac{\xi_0 q}{\alpha}+1)^{2m}-1-2m \cdot \frac{\xi_0 q}{\alpha}]|}{2mq}=\frac{|\mathbb{E}[((\frac{\xi_0 q}{\alpha}+1)^{2m}-1-2m \cdot \frac{\xi_0 q}{\alpha})\chi_{|\frac{\xi_0 q}{\alpha}| \leq t_0}]|}{2mq}+o(\alpha^{-1})
\end{equation}
employing that for \(X={\mathbb{E}[\lambda_1(A)]}-1=\frac{\xi_0 q}{\alpha},\)
\[|\mathbb{E}[((X+1)^m-1-mX)\chi_{|X| \geq (np/K)^{-1/2}}]| \leq \sum_{k \geq 2}{\binom{m}{k}\sqrt{\mathbb{E}[X^{2k}]}} \leq
\frac{2}{n^{3/2}} \cdot (1+(\frac{np}{K})^{-1})^m=O(n^{-3/2})=o(mq\alpha^{-1})\]
because \(c_3 \log{n} \leq m \leq c_4\log{n}, np \geq B(\log{n})^4,q=\sqrt{p(1-p)}\geq 2^{-1/2}\tilde{p}^{1/2} \geq (2n)^{-1/2}\) yield
\[\frac{\alpha n^{-3/2}}{mq} \leq \frac{2np \cdot n^{-3/2}}{(2n)^{-1/2} \cdot c_3 \log{n}}=O((\log{n})^{-1})=o(1)\]
(recall \(\alpha=\mathbb{E}[\lambda_1(A)] 
=\Omega(np)\) via (\ref{sqexp})). Since for \(|x| \leq \frac{1}{16m}, m \geq 2,\)
\[\frac{m^2x^2}{8} \leq (1+x)^{m}-1-mx \leq \frac{5m^2x^2}{8}\]
from 
\[\frac{m^2}{2} \geq \binom{m}{2} \geq \frac{m^2}{4}, \hspace{0.8cm} |(1+x)^{m}-1-mx-\binom{m}{2}x^2| \leq \sum_{k \geq 3}{\binom{m}{k}|x|^k} \leq 2(m|x|)^3 \leq \frac{m^2x^2}{8},\]
(\ref{c1}), (\ref{c2}), and \(mt_0 \leq \frac{Kc_4\log{n}}{(np)^{1/2}} \leq \frac{Kc_4\log{n}}{B^{1/2}(\log{n})^2} \leq \frac{1}{31}\) for \(n \geq n(B,c_4)\) render
\[c(m,n,p)=\Omega(\frac{\alpha}{2mq} \cdot m^2\mathbb{E}[(\frac{\xi_0 q}{\alpha})^2\chi_{|\frac{\xi_0 q}{\alpha}| \leq t_0}])+o(1)\]
or
\[c(m,n,p)=\Omega(\frac{m\alpha}{q} \cdot \mathbb{E}[(\frac{\xi_0 q}{\alpha})^2\chi_{|\frac{\xi_0 q}{\alpha}| \leq t_0}])+o(1),\]
with the constants underlying \(\Omega\) being \(\frac{1}{8},\frac{5}{8}.\)
\par
This identity for \(m_1,m_2\) with \(m_1,m_2 \in [c_3\log{n},c_4\log{n}] \cap \mathbb{Z}, m_1=6m_2\) yields (\ref{fine}) for \(m=m_1:\)
\[(1+o_p(1))c(m_1,n,p)-(1+o_p(1))c(m_2,n,p)=\Omega(c(m_1,n,p)-c(m_2,n,p))+o(1)=\Omega(c(m_1,n,p))+o(1)\]
since \(m_1-5m_2=\frac{m_1}{6},\) while for any \(C>0, \epsilon(C)>0,\) 
\[\mathbb{P}((1+o_p(1))c(m_1,n,p)-(1+o_p(1))c(m_2,n,p)>C)=\mathbb{P}(((1+o_p(1))c(m_1,n,p)+\xi_0)-((1+o_p(1))c(m_2,n,p)+\xi_0)>C) \leq\]
\[\leq \mathbb{P}((1+o_p(1))c(m_1,n,p)+\xi_0>C/2)+\mathbb{P}((1+o_p(1))c(m_2,n,p)+\xi_0<-C/2) \leq (1+\epsilon(C))\mathbb{P}(|N(0,1)|>C/2)<1,\]
the last two results implying
\[c(m_1,n,p)=o(1).\]
\end{proof}

\bibliographystyle{unsrtnat}
\bibliography{references}

\end{document}